\documentclass[11pt]{amsart}

\usepackage{hyperref}
\usepackage{amssymb,amsfonts}
\usepackage{hyperref}
\usepackage{amssymb,textcomp}

\usepackage{enumerate}

\usepackage[utf8]{inputenc}
\usepackage[T1]{fontenc}

\usepackage[mathscr]{eucal}

\usepackage[a4paper, twoside=false, vmargin={2cm,3cm}, includehead]{geometry}

\newtheorem{lemma}{Lemma}[section]
\newtheorem{theorem}[lemma]{Theorem}
\newtheorem*{theorem*}{Theorem}
\newtheorem{corollary}[lemma]{Corollary}

\newtheorem{proposition}[lemma]{Proposition}
\newtheorem*{proposition*}{Proposition}

\newtheorem{conjecture}{Conjecture}
\newtheorem{problem}{Problem}
\newtheorem*{problem*}{Problem}

\theoremstyle{definition}
\newtheorem*{claim*}{Claim}

\newtheorem*{definition}{Definition}

\newtheorem*{remark}{Remark}
\newtheorem*{remarks}{Remarks}

\newcommand{\C}{{\mathbb C}}
\newcommand{\E}{{\mathbb E}}

\newcommand{\N}{{\mathbb N}}

\newcommand{\Q}{{\mathbb Q}}
\newcommand{\R}{{\mathbb R}}

\newcommand{\T}{{\mathbb T}}
\newcommand{\Z}{{\mathbb Z}}

\newcommand{\CI}{{\mathcal I}}

\newcommand{\CX}{{\mathcal X}}

\newcommand{\norm}[1]{\left\Vert #1\right\Vert}
\newcommand{\nnorm}[1]{\lvert\!|\!| #1|\!|\!\rvert}

\DeclareMathOperator{\spec}{Spec}
\DeclareMathOperator{\id}{id}

\begin{document}
	
	\title{Joint ergodicity of  sequences}	
	
	\thanks{The author was supported  by the
		Hellenic Foundation for
		Research and Innovation,
		Project
		No: 1684.}
	
	\author{Nikos Frantzikinakis}
	\address[Nikos Frantzikinakis]{University of Crete, Department of mathematics and applied mathematics, Voutes University Campus, Heraklion 71003, Greece} \email{frantzikinakis@gmail.com}
	\begin{abstract}
	A collection of integer sequences is jointly ergodic if for every ergodic measure preserving system the multiple ergodic averages, with iterates given by this collection of sequences, converge in the mean to the product of the integrals. We give necessary and sufficient conditions
	for joint ergodicity that are flexible enough to recover  the known examples of jointly ergodic sequences and also allow us to answer some related open problems. An interesting feature of our arguments is that they  avoid deep tools from ergodic theory that were previously used to establish similar results. Our approach is primarily based on an ergodic variant of a  technique pioneered by Peluse and Prendiville in order to give  quantitative variants for the finitary version of the polynomial Szemer\'edi theorem. 	
	\end{abstract}


\maketitle

\section{Introduction and main results}\label{S:MainResults}
\subsection{Introduction}
The study   of
multiple  ergodic averages was initiated in the seminal work of Furstenberg~\cite{Fu77}, where an ergodic theoretic proof of Szemer\'edi's theorem on arithmetic progressions was given. Since then a variety of multiple ergodic  averages has been studied, resulting in
 new and far reaching combinatorial consequences. A rather general family of problems is as follows: We are given
  a collection of  integer sequences $a_1,\ldots, a_\ell\colon \N\to \Z$ and an invertible measure preserving system $(X,\mu,T)$. We would like to understand the limiting behavior (as $N\to\infty$) in $L^2(\mu)$  of the averages
\begin{equation}\label{E:MEA}
\E_{n\in [N]} \, T^{a_1(n)}f_1\cdot \ldots \cdot T^{a_\ell(n)}f_\ell
\end{equation}
for all functions $f_1,\ldots, f_\ell\in L^\infty(\mu)$ where  $\E_{n\in [N]}$ denotes the average $\frac{1}{N}\sum_{n=1}^N$. The reader can find a large collection of related convergence
results  in  \cite{Fr16}. We remark though that  the limit is not easy to compute even in the case where the iterates are  integer polynomials. In particular, if the sequences satisfy some non-trivial linear relations, then simple examples of ergodic compact abelian group rotations show that  the averages \eqref{E:MEA} are not going to converge to the product of the integrals of the individual functions.


 The question we seek to answer is under what conditions on the sequences $a_1,\ldots, a_\ell$ the iterates $T^{a_1(n)}, \ldots, T^{a_\ell(n)}$, $n\in \N$,
 behave independently enough, so
 that  for every ergodic system, or totally ergodic system (cf. Section~\ref{SS:JE}), 
 the averages \eqref{E:MEA} converge in $L^2(\mu)$ to the product of the integrals.
  This gives rise to the following notion of ``joint ergodicity'' that was first introduced in a somewhat different setting  in \cite{BB84}.
\begin{definition}
 We say that  the collection of sequences   $a_1,\ldots, a_\ell\colon \N\to \Z$ is
 \begin{enumerate}
 	\item
  {\em jointly ergodic for the 
  	system  $(X,\mu, T)$}, if  for all functions   $f_1,\ldots, f_\ell \in L^\infty(\mu)$ we have
	\begin{equation}\label{E:jem}
	\lim_{N\to\infty} \E_{n\in [N]}\, T^{a_1(n)}f_1 \cdot\ldots \cdot  T^{a_\ell(n)}f_\ell= \int f_1\, d\mu\cdot \ldots \cdot \int f_\ell\, d\mu
	\end{equation}
	where convergence takes place in $L^2(\mu)$.
	
	\item {\em jointly ergodic}, if it is jointly ergodic for every ergodic system.
	\end{enumerate}
\end{definition}
It is a direct consequence of Furstenberg's correspondence principle~\cite{Fu81a} that if the collection of sequences $a_1,\ldots, a_\ell$ is jointly ergodic, then every set of integers with positive upper density contains patterns of the form $m,m+a_1(n),\ldots, m+a_\ell(n)$, for some $m,n\in \N$.
In fact, a stronger property holds (which fails when say $a_1(n)=n,a_2(n)=2n$, see \cite[Theorem~2.1]{BHK05}), namely, for every set of integers $\Lambda$  we have
\begin{equation}\label{E:dl1}
\liminf_{N\to\infty}\E_{n\in[N]}\, \bar{d}(\Lambda\cap (\Lambda +a_1(n))\cap\cdots \cap (\Lambda+a_\ell(n)))\geq (\bar{d}(\Lambda))^{\ell+1},
\end{equation}
where $\bar{d}(\Lambda):=\limsup_{N\to\infty}|\Lambda\cap [1,N]|/N$.


For $\ell=1$, by appealing to the spectral theorem for unitary operators (or the theorem of Herglotz on positive definite sequences),
one can show that  a single sequence $a\colon \N\to \Z$ is (jointly) ergodic if and only if for every $t\in (0,1)$ we have
\begin{equation}\label{E:Spectral}
\lim_{N\to\infty} \E_{n\in[N]}\, e(a(n)t)=0
\end{equation}
where $e(t):=e^{2\pi i t}$. In particular, the sequences $[n^c]$, where  $c\in \R_+\setminus \Z$,
and $[n^2\alpha+n\beta]$, where $\alpha,\beta$ are non-zero real numbers such that
$\alpha/\beta$ is irrational,  are (jointly) ergodic (see \cite[Theorem~A]{BKQW05}). It also follows that a single sequence is jointly ergodic for every totally ergodic system if and only if  \eqref{E:Spectral} holds for every irrational $t\in [0,1)$, a condition that is known to be satisfied, for instance,  by all non-constant  integer polynomials.

For $\ell\geq 2$,     joint ergodicity is a much tougher property to establish.
In \cite{Fr10} it was shown  that the collection of sequences $[n^{c_1}],\ldots, [n^{c_\ell}]$, where $c_1,\ldots, c_\ell\in \R_+\setminus \Z$ are distinct,  is jointly ergodic, and further  examples involving polynomials with real coefficients and other sequences arising from smooth functions and bracket polynomials can be found in  \cite{BLS16,BMR20,BMR21, Fr15,  KK19, K21}. Moreover, a collection of  polynomial sequences $p_1,\ldots, p_\ell\in \Z[t]$ is known to be jointly ergodic
for all totally ergodic systems if and only if the polynomials are rationally independent \cite{FrK05a}; a typical case is the pair $n,n^2$, which was first handled in \cite{FuW96}.
 The proofs of  most of these results rely on deep tools from ergodic theory, such as  the Host-Kra theory of characteristic factors \cite{HK05a}
 and equidistribution results on nilmanifolds. But a general criterion for establishing joint ergodicity, like the one mentioned for $\ell=1$,  is still lacking, the reason being that  we do not have a  good substitute for the spectral theorem for unitary operators that would allow us to represent correlation sequences of the form $\int f_0 \cdot T^{n_1}f_1\cdot \ldots \cdot  T^{n_\ell}f_\ell\, d\mu$, for sparse values of $n_1,\ldots, n_\ell\in \N$,  in a useful way.

   Using ergodic rotations on compact abelian groups we see that a necessary condition for joint ergodicity for the collection of sequences $a_1,\ldots, a_\ell\colon \N\to \Z$ is that
for all $t_1,\ldots, t_\ell\in [0,1)$, not all of them zero, we have
$$
\lim_{N\to\infty} \E_{n\in[N]}\, e(a_1(n)t_1+\cdots+a_\ell(n)t_\ell)=0.
$$
Another necessary condition is the ``good for seminorm estimates'' property  that is defined in  Section~\ref{SS:JE} and is known to be satisfied for an ample supply of sequences.
The main
objective of this article is to prove the rather surprising fact  that these two conditions are also sufficient for joint ergodicity. We do this in  Theorem~\ref{T:JointErgodicity}.
These  necessary and sufficient conditions  enable us to recover painlessly all the known results about  joint ergodicity of  sequences that we are aware of and also 
prove some new ones. For example,  in Theorem~\ref{T:Hardy} we make progress towards  a conjecture in
\cite[Problem~2]{Fr10} 
regarding joint ergodicity of Hardy field sequences,
improving upon the best known results in \cite{BMR21,Fr10}, and in Theorem~\ref{T:HardyPolyChar} we verify  a conjecture from   \cite[Conjecture~6.1]{BMR21}. A special case of this last result  asserts that a collection of sequences  of the form
 $[\sum_{i=1}^k\alpha_in^{b_i}]$, where  $\alpha_1,\ldots, \alpha_k\in \Q$, $b_1,\ldots, b_k\in (0,+\infty)$ are positive,
  is jointly ergodic for all totally ergodic systems if and if it is linearly independent;
  this was previously known when all the exponents are  integers \cite{FrK05a} or none of the exponents is an integer \cite{BMR21}.


An interesting feature of our main result is that its proof avoids deep tools from ergodic theory  and
leads to vastly simpler proofs of several known results. In particular,
our arguments
do not rely on the Host-Kra theory of characteristic factors or on equidistribution results on nilmanifolds, as previous results in the area did. Moreover,   in  Corollary~\ref{C:Nilsystems} we show that our main result has some far-reaching consequences related   to equidistribution properties of general sequences on nilmanifolds. Our approach is to adapt and utilize in our ergodic theory setup  a technique developed by Peluse~\cite{P19a} and Peluse and Prendiville~\cite{PP19}  (see also \cite{P20} for an exposition of the technique) and used in \cite{P19b,PP19}  to establish quantitative results for finitary variants of  special cases of   the polynomial Szemer\'edi theorem.
 Although the main ideas are elementary, they are a bit cumbersome to implement in full generality,  so in order to facilitate reading,   we first present our argument
 in the simpler case of the multiple ergodic averages 
$$
\E_{n\in [N]}\,  T^nf\cdot T^{n^2}g.
$$
It is a classical result of Furstenberg and Weiss that these averages converge in $L^2(\mu)$ to the product of the integrals for totally ergodic systems,  and that the rational Kronecker factor is a characteristic factor for the above averages.
We present in  Section~\ref{S:FW} a proof of this result
using the general principles of the argument used in \cite{PP19}, but  also take advantage of various simplifications that our infinitary setup allows (the complete argument is only a few pages long). Subsequently, in Sections~\ref{S:Preparation} and  \ref{S:Main} we extend this method to give a proof of our main result, Theorem~\ref{T:JointErgodicity}, which gives necessary and sufficient conditions for joint ergodicity. We then use it in  Section~\ref{S:Corollaries} to give a result about characteristic factors (Theorem~\ref{T:CharacteristicFactors}) that also applies to arbitrary collections of rationally independent integer polynomials.
In Section~\ref{S:Hardy}, by applying the previous results, we recover known results for Hardy field sequences and also  address some  related open problems. An additional advantage of our approach is that it relies on   vastly simpler arguments than previously used.


 Finally, in Section~\ref{S:Flows}, we deal
with similar problems for flows. This section  is independent from the previous ones, and the methodology used  is completely different   than the one used to cover discrete time averages. We prove pointwise convergence  results that apply to actions of not necessarily commuting transformations. We remark that even for commuting transformations, the corresponding results for discrete time actions are not known, and for general non-commuting transformations they are known to be false~\cite[Theorem~1.7]{FrLW11}.
Our main result regarding flows  is given in Theorem~\ref{T:flows},  and  representative special cases  assert that
 for all ergodic measure preserving actions  $T^t, S^t, R^t$, $t\in\R$,  on a probability space $(X,\mathcal{X},\mu)$ and functions  $f,g,h\in L^\infty(\mu)$,
the following limits exist
$$
\lim_{y\to+\infty} \frac{1}{y} \int_0^y f(T^tx)\cdot g(S^{t^2}x)\, dt, \qquad
\lim_{y\to+\infty} \frac{1}{y} \int_0^y f(T^{2^t}x)\cdot g(S^{3^t}x)\cdot h(R^{4^t}x)\, dt
$$
pointwise  for $\mu$-almost every $x\in X$,  and are equal to the product of the integrals of the individual functions.
We also establish analogous  results without any ergodicity assumptions.
When the actions $T^t$ and $S^t$ commute, the corresponding result  for the first average was recently obtained
 in \cite{CDKR21}, using harmonic analysis techniques from \cite{CDR21}.

\subsection{Definitions and notation} \label{SS:notation}

With $\N$ we denote the set of positive integers and with $\Z_+$ the set of non-negative integers.
With  $\R_+$ we denote the set of non-negative real numbers.  For $t\in \R$ we let $e(t):=e^{2\pi i t}$. With   $[t]$ we denote the integer part of $t$ and with $\{t\}$ the fractional part of $t$.
 With $\T$ we denote the one dimensional torus and we often identify it with $\R/\Z$ or  with $[0,1)$.
 With  $\Re(z)$ we denote the real part of the complex number $z$.

For  $N\in\N$  we let  $[N]:=\{1,\dots,N\}$ (not to be confused with the integer part). If $a\colon \N^s\to \C$ is a  bounded sequence for some $s\in \N$ and $A$ is a non-empty finite subset of $\N^s$,  we let
$$
\E_{n\in A}\,a(n):=\frac{1}{|A|}\sum_{n\in A}\, a(n).
$$


If $a,b\colon \R_+\to \R$ are functions we write
 $a(t)\prec b(t)$ if $\lim_{t\to +\infty} a(t)/b(t)=0$.
We say that the function  $a\colon \R_+\to \R$ has {\em  at most polynomial growth} if there exists $d\in \N$ such that $a(t)\prec t^d$.

\subsection{Joint ergodicity of general sequences}\label{SS:JE}
A {\em measure preserving system}, or simply {\em a system}, is a quadruple $(X,\mathcal{X},\mu,T)$
where $(X,\mathcal{X},\mu)$ is a Lebesgue probability space and $T\colon X\to X$ is an invertible, measurable,  measure preserving transformation.
We typically omit the $\sigma$-algebra $\mathcal{X}$  and write $(X,\mu,T)$. Throughout,  for $n\in \N$ we denote  by $T^n$   the composition $T\circ  \cdots \circ T$ ($n$ times) and let $T^{-n}:=(T^n)^{-1}$ and $T^0:=\id_X$. Also, for $f\in L^\infty(\mu)$ and $n\in\Z$ we denote by  $T^nf$ the function $f\circ T^n$.

We say that the system $(X,\mu,T)$ is {\em ergodic}  if the only functions $f\in L^1(\mu)$ that
satisfy $Tf=f$ are the constant ones. It is {\em totally ergodic}, if $(X,\mu,T^d)$ is ergodic for every $d\in \N$.

To facilitate discussion we introduce the following notions.
\begin{definition}
	If $(X,\mu,T)$ is   a  system we  let
	\begin{enumerate}
		\item 	$\spec(T):=\{t\in [0,1)\colon Tf=e(t)
\,  f  \text{ for some non-zero } f\in L^2(\mu)\}$.

	\item
	$ \mathcal{E}(T):=\{f\in L^\infty(\mu)\colon Tf=e(t)\, f   \text{ for some } t\in [0,1)  \text{ and } |f|=1\}$.
	\end{enumerate}
	\end{definition}
For  the definition of the ergodic seminorms $\nnorm{\cdot}_s$ we refer the reader to  Section~\ref{SS:GHK}.
\begin{definition}
	We say that  the collection of sequences  $a_1,\ldots, a_\ell\colon \N\to \Z$ is:

\begin{enumerate}
\item {\em   good for seminorm estimates for the system $(X,\mu,T)$}, if  there exists $s\in \N$ such that whenever  $f_1,\ldots, f_\ell \in L^\infty(\mu)$
satisfy  $\nnorm{f_m}_{s}=0$  for some $m\in \{1,\ldots, \ell\}$  and $f_{m+1},\ldots, f_\ell\in \mathcal{E}(T)$, we have
$$
\lim_{N\to\infty} \E_{n\in [N]}\, T^{a_1(n)}f_1\cdot\ldots \cdot  T^{a_\ell(n)}f_\ell= 0
$$
where convergence takes place in $L^2(\mu)$.
 It is {\em good for seminorm estimates}, if it is good for seminorm estimates for every ergodic system.




\item {\em good for equidistribution  for the system $(X,\mu,T)$}, if for all  $t_1,\ldots, t_\ell\in \spec(T)$,  not all of them $0$, we have
\begin{equation}\label{E:equidistribution}
\lim_{N\to\infty} \E_{n\in[N]}\,  e(a_1(n)t_1+\cdots+ a_\ell(n)t_\ell) =0.
\end{equation}
It is {\em good for equidistribution} if it is good for equidistribution for every system, or  equivalently, if \eqref{E:equidistribution} holds for all $t_1,\ldots, t_\ell\in [0,1)$, not all of them $0$.


\item {\em  good for irrational equidistribution for   the system $(X,\mu,T)$}, if for all  $t_1,\ldots, t_\ell\in \spec(T)$, not all of them rational, equation \eqref{E:equidistribution} holds.
It is  {\em good for irrational equidistribution}, if it is good for irrational equidistribution for every system, or equivalently, if \eqref{E:equidistribution} holds for all $t_1,\ldots, t_\ell\in [0,1)$, not all of them rational.
\end{enumerate}
\end{definition}
\begin{remarks}
	$\bullet$ Using a standard argument (see for example \cite[Lemma~4.12]{Fr15})
	we get that if  the sequences $a_1,\ldots, a_\ell$ are good  for seminorm estimates    for
	the product system $(X\times X, \mu\times\mu, T\times T)$ for some $s\in \N$, then  for all  $f_1,\ldots, f_\ell\in L^\infty(\mu)$ with $\nnorm{f_m}_{s+1}=0$  for some $m\in [\ell]$ and
	every bounded sequence of complex numbers $(c_n)$, we have
	$
	\lim_{N\to\infty}\E_{n\in [N]}\, c_n\, T^{a_1(n)}f_1\cdot\ldots \cdot  T^{a_m(n)}f_m= 0
	$
	in $L^2(\mu)$.  We are going to use this fact in Section~\ref{SS:CF}.
	
		$\bullet$ If $\mu=\int \mu_x\, d\mu$ is the ergodic decomposition of the measure $\mu$, it is known that if  $\nnorm{f}_{s,\mu}=0$, then  $\nnorm{f}_{s,\mu_x}=0$ for $\mu$-almost every $x\in X$  \cite[Chapter~8, Proposition~18]{HK18}. As a consequence,  if a collection of sequences is good for seminorm estimates for every ergodic system, then it is also good for seminorm estimates for arbitrary systems.
	
$\bullet$	If the sequences $a_1,\ldots, a_\ell$ are  good for irrational equidistribution, then
	the sequences $1,a_1,\ldots, a_\ell$  have to be  linearly independent. 
	 Indeed, suppose that   $c_0+c_1a_1+\cdots+c_\ell a_\ell=0$ for some $c_0,c_1,\ldots, c_\ell\in \R$ not all of them $0$. After multiplying by an appropriate constant  we can assume that at least one of the numbers  $c_1,\ldots, c_\ell$ is irrational and    this  immediately implies that  the collection $a_1,\ldots, a_\ell$ is  not good for irrational equidistribution.
	We remark also, that although  
	  for collections of polynomial sequences linear independence implies good irrational equidistribution,  this is  not true in general, for instance, it fails when have a single sequence that  equals $2^n$ or $[n\sqrt{2}]$.

$\bullet$ If $c_1,\ldots, c_k$ are distinct positive real numbers, then it is known that the sequences $[n^{c_1}],\ldots, [n^{c_k}]$  are  good for seminorm estimates \cite[Theorem~2.9]{Fr10} and it is not hard to show using known equidistribution results that they are good for irrational equidistribution (see Proposition~\ref{L:HardyPoly} below). If in addition the exponents  $c_1,\ldots, c_k$ are not integers, then the sequences are also
 good for equidistribution (see Proposition~\ref{L:WeylHardy} below).
	\end{remarks}




If a collection of sequences is jointly ergodic for a system, then it  is good for seminorm estimates (since  $\nnorm{f}_1=0$ implies $\int f\, d\mu$=0) and equidistribution for this system (since verifying \eqref{E:jem} for eigenfunctions of the system implies \eqref{E:equidistribution}).
It is a rather surprising fact  that the converse is also true; this is the context
of our main result that we now state.
\begin{theorem}\label{T:JointErgodicity}
	Let $(X,\mu,T)$ be an ergodic system and   $a_1,\ldots, a_\ell\colon \N\to \Z$  be sequences. Then the following two conditions are equivalent:
  	\begin{enumerate}
  	\item
	The sequences $a_1,\ldots, a_\ell$ are jointly ergodic for $(X,\mu,T)$.
  	  	
  	\item The sequences $a_1,\ldots, a_\ell$ are good for seminorm estimates and equidistribution for $(X,\mu,T)$.
  \end{enumerate}
\end{theorem}
\begin{remarks}
$\bullet$
Even when $\ell=1$ and the sequence $a_1$ is strictly increasing, we cannot omit the seminorm condition from the equivalence. Indeed, if a system is weakly mixing but not strongly mixing, then any  sequence is good for equidistribution for this system, but clearly, there exists   a strictly increasing sequence that is not  ergodic for this system. See though Problem~\ref{Problem1} below
for a related open problem.

$\bullet$ Even  if we use  the full force of the Host-Kra theory of characteristic factors~\cite{HK05a}, it is not clear how to prove Theorem~\ref{T:JointErgodicity}. So the more elementary approach we follow 
offers substantial technical advantages.

$\bullet$ It would be interesting to examine if  the arguments used in
\cite{KMT21} (which also crucially use ideas from \cite{P19a, PP19})
can be modified to give similar necessary conditions so that  \eqref{E:jem}
 holds pointwise; of course, in this case one has to use stronger quantitative variants of the conditions used in $(ii)$.
	\end{remarks}


Since the spectrum of a totally ergodic system contains only  irrationals in $(0,1)$ and the number $0$, we immediately deduce the following result:
\begin{corollary}\label{C:JointTotalErgodicity}
The following two conditions are equivalent:
	\begin{enumerate}
		\item
		The sequences $a_1,\ldots, a_\ell$ are jointly ergodic for every totally ergodic system.
		
		\item The sequences $a_1,\ldots, a_\ell$ are good for seminorm estimates and  \eqref{E:equidistribution} holds for all $t_1,\ldots,t_\ell\in [0,1)$
		that are irrational or $0$, but not all of them $0$.
	\end{enumerate}
\end{corollary}
 In particular,
this is   satisfied  by collections of linearly independent polynomials with integer coefficients and zero constant terms (which are also known to be good for seminorm estimates), thus providing a simpler proof of the main result in \cite{FrK05a}. See also Theorem~\ref{T:HardyPolyChar} below for an application to   a vastly more general  collection of Hardy field sequences.

 The second application  is a strong multiple recurrence property, that is not shared for example by collections of linear sequences, or more general polynomial sequences.
\begin{corollary}\label{C:MultipleRecurrence}
	Let   $a_1,\ldots, a_\ell\colon \N\to \Z$  be sequences  that are good for seminorm estimates and  equidistribution for the
	system $(X,\mu,T)$. Then for every  set $A\in \mathcal{X}$ we have
		$$
	\lim_{N\to\infty} \E_{n\in [N]}\,  \mu(A\cap T^{-a_1(n)}A\cap \cdots \cap T^{-a_\ell(n)}A)\geq (\mu(A))^{\ell+1}.
	$$
\end{corollary}
 By invoking the correspondence principle of Furstenberg \cite{Fu81a}, one deduces that if the sequences $a_1,\ldots, a_\ell\colon \N\to \Z$ are good for seminorm estimates and equidistribution, then  for every set of integers $\Lambda$ equation \eqref{E:dl1} holds.

	

We also record a result where the assumption of seminorm estimates can be omitted. It allows us to infer equidistribution properties for general nilsystems\footnote{ 	A  $k$-step  nilsystem is
	a system of the form $(X, m_X, T_a)$,  where $X=G/\Gamma$ is a $k$-step nilmanifold (i.e.  $G$ is a $k$-step nilpotent Lie group
	and $\Gamma$ is a discrete cocompact subgroup of $G$), $a\in G$,  $T_a\colon X\to X$ is defined by  $T_{a}(g\Gamma) \mathrel{\mathop:}= (ag) \Gamma$, $g\in G$, and  $m_X$  is
	the normalized Haar measure on $X$.}  from the special case of compact abelian group rotations. Similar results were previously known for collections of polynomial sequences \cite{L05} (but with stronger equidistribution assumptions) and certain Hardy field sequences of at most  polynomial growth \cite{Fr09,R21}.
\begin{corollary}\label{C:Nilsystems}
	Let $(X,\mu,T)$ be an ergodic   nilsystem. Then the following two conditions are equivalent:
	\begin{enumerate}
	\item The sequences $a_1,\ldots, a_\ell$ are jointly ergodic for $(X,\mu,T)$.

		\item
		The sequences $a_1,\ldots, a_\ell$ are good for equidistribution for $(X,\mu,T)$.
	\end{enumerate}
\end{corollary}
\begin{remark}
	Roughly speaking, this statement asserts that a collection of sequences is jointly ergodic for a nilsystem if and only if it is jointly ergodic for its Kronecker factor.
\end{remark}
The previous result implies that if the sequences $a_1,\ldots, a_\ell$ are good for equidistribution for the ergodic nilsystem  $(X,\mu,T)$, then for all  $f_1,\ldots, f_\ell\in C(X)$ we have
\begin{equation}\label{E:qwe}
\lim_{N\to\infty}\E_{n\in [N]}\,  f_1(T^{a_1(n)}x)\cdot \ldots \cdot f_\ell(T^{a_\ell(n)}x)=\int f_1\, d\mu\cdots \int f_\ell\, d\mu
\end{equation}
where convergence  takes place in $L^2(\mu)$. It is known (see \cite[Example~5]{FrK05a}) that the convergence in \eqref{E:qwe}
cannot be strengthened to pointwise convergence  for every $x\in X$. We do not know if under the stated assumptions  the convergence in \eqref{E:qwe}  can be strengthened to pointwise  convergence for $m_X$-almost every $x\in X$ and all $f_1,\ldots, f_\ell\in C(X)$.

 For our last application we assume that $X$ is a compact metric space and $T\colon X\to X$ is a homeomorphism. The system $(X,T)$ is called a {\em topological dynamical system} and it is called {\em minimal} if every orbit $\{T^nx\colon n\in \N\}$ is dense in $X$. Using Theorem~\ref{T:JointErgodicity} and arguing as in \cite[Section~2.3]{Fr15} we deduce the following:
\begin{corollary}
		Let   $a_1,\ldots, a_\ell\colon \N\to \Z$  be sequences  that are good for seminorm estimates and  equidistribution. Then
		\begin{enumerate}
			\item For every  minimal system $(X,T)$,  for a residual   set of $x\in X$ we have
			$$
			\overline{\{(T^{a_1(n)}x,\ldots,  T^{a_\ell(n)}x ) \colon n\in \N\}}=X\times \cdots\times X.
			$$
			\item For every topological dynamical  system $(X,T)$,  there exists $x\in X$ such  that
			$$
			\overline{\{(T^{a_1(n)}x,\ldots,  T^{a_\ell(n)}x ) \colon n\in \N\}}=\overline{\{T^nx\colon n\in \N\}} \times \cdots\times \overline{\{T^nx\colon n\in \N\}}.
			$$
		\end{enumerate}
\end{corollary}

\subsection{Characteristic factors for general sequences}
If some of the sequences $a_1,\ldots, a_\ell$ are integer polynomials, then
Theorem~\ref{T:JointErgodicity} is not applicable to  all ergodic systems because the equidistribution property is not satisfied. We will state a  result about characteristic factors  that gives useful information in such cases and is strong enough to easily imply related convergence and multiple recurrence results.

\begin{definition}
Let $(X,\mu,T)$ be a system (not necessarily ergodic).
\begin{enumerate}
\item    The {\em rational Kronecker factor}  $\mathcal{K}_{rat}(T)$  of a system $(X,\mu,T)$ is the $L^2(\mu)$ closure of the linear span of rational eigenfunctions of the system, meaning, non-zero functions $f\in L^2(\mu)$ such that $Tf=e(t) f$ for some $t\in \Q$.

\item We say that the rational Kronecker factor $\mathcal{K}_{rat}(T)$ is a {\em characteristic factor for the sequences $a_1,\ldots,a_\ell\colon \N\to \Z$}, if  whenever at least one of the functions $f_1,\ldots, f_\ell \in L^\infty(\mu)$  is orthogonal to $\mathcal{K}_{rat}(T)$,  we have
 $$
	\lim_{N\to\infty} \E_{n\in [N]}\, T^{a_1(n)}f_1 \cdot\ldots \cdot  T^{a_\ell(n)}f_\ell=0
	$$
	where convergence takes place in $L^2(\mu)$.
\end{enumerate}
\end{definition}

 For the purposes of the next result we will also need a strengthening of a notion introduced in the previous subsection.
 \begin{definition}
 	We say that  the collection of sequences  $a_1,\ldots, a_\ell\colon \N\to \Z$  is
 	 {\em  very good for seminorm estimates for the system $(X,\mu,T)$}, if  there exists $s\in \N$ such that if $f_1,\ldots, f_\ell \in L^\infty(\mu)$
 		and $\nnorm{f_m}_{s}=0$ for some $m\in [\ell]$, then
 		$$
 		\lim_{N\to\infty} \E_{n\in [N]}\, T^{a_1(n)}f_1\cdot\ldots \cdot  T^{a_\ell(n)}f_\ell= 0
 		$$
 		where convergence takes place in $L^2(\mu)$. It is {\em very good for seminorm estimates}, if it is very good for seminorm estimates for every ergodic system.
 \end{definition}
It is  easy to verify that  if for a given collection of sequences the rational Kronecker factor is characteristic for  every measure preserving system, then this  collection of sequences is  very  good for seminorm estimates (with $s=2$) and good for irrational equidistribution. It turns out that the converse is also true. This is the context of our next result.
\begin{theorem}\label{T:CharacteristicFactors}
	Let    $a_1,\ldots, a_\ell\colon \N\to \Z$  be sequences. Then
 the following two conditions are equivalent:
	\begin{enumerate}
		\item
	 For every system $(X,\mu,T)$	the factor $\mathcal{K}_{rat}(T)$ is a characteristic factor for the sequences $a_1,\ldots,a_\ell$.
		
		\item The sequences $a_1,\ldots, a_\ell$ are very good for seminorm estimates and good for irrational equidistribution.
	\end{enumerate}
\end{theorem}
\begin{remarks}
$\bullet$ Although it is not  particularly  hard to deduce this result from   Theorem~\ref{T:JointErgodicity},  we use in a crucial way the fact that
characteristic factors of the averages we are interested in are inverse limits of systems with finite rational spectrum, a property that we do not know how to prove without appealing to the main result in \cite{HK05a}.

$\bullet$ If in addition to the properties in Part~$(ii)$ we assume that for all rational $t_1,\ldots, t_\ell\in [0,1)$ the averages $\E_{n\in[N]}\, e(a_1(n)t_1+\cdots+a_\ell(n)t_\ell)$ converge as $N\to\infty$, then it is easy  to deduce from Theorem~\ref{T:CharacteristicFactors} that  the averages \eqref{E:MEA}  converge in the mean for all systems and  functions in $L^\infty(\mu)$.
\end{remarks}

Next, we give a consequence of the previous result to multiple recurrence.


\begin{definition}
We say that a collection of sequences $a_1,\ldots, a_\ell \colon \N\to \Z$ has {\em good divisibility properties}, if for every $r\in \N$ we have
 \begin{equation}\label{E:dnN}
 \bar{d}(\{n\in \N\colon a_1(n)\equiv 0, \ldots, a_\ell(n)\equiv 0 \!  \!  \! \pmod{r}\})>0.
 \end{equation}
 \end{definition}

\begin{corollary}\label{C:MultipleRecurrence'}
	 Let   $a_1,\ldots, a_\ell\colon \N\to \Z$  be sequences with good divisibility properties  that are very good for seminorm estimates and   good for    irrational equidistribution. Then for every
	system $(X,\mu,T)$, set $A\in \mathcal{X}$, and every $\varepsilon>0$, we have
	$$
	\mu(A\cap T^{-a_1(n)}A\cap \cdots \cap T^{-a_\ell(n)}A)\geq (\mu(A))^{\ell+1}-\varepsilon
	$$
	for a set of $n\in \N$ with positive upper density (or lower density if in \eqref{E:dnN} we use $\underline{d}$).
\end{corollary}

 By invoking the correspondence principle of Furstenberg~\cite{Fu81a} one deduces that under the previous assumptions on $a_1,\ldots, a_\ell$, for every set of integers $\Lambda$   and      $\varepsilon>0$ we have
 $$
 \bar{d}(\Lambda \cap (\Lambda+ a_1(n)) \cap \cdots\cap (\Lambda+a_\ell(n)) )\geq (\bar{d}(\Lambda))^{\ell+1}-\varepsilon
 $$
  for a set of $n\in \N$ with positive upper density.
 Moreover, arguing as in \cite[Theorem~2.2]{FrK06}\footnote{It is crucial  that Corollary~\ref{C:MultipleRecurrence'} holds for all systems; it is not known how to deduce  such finitistic consequences from  multiple recurrence results that are only known  for all ergodic systems.}  we get the following finitistic variant of this result: For every $\delta>0$ and $\varepsilon>0$ there exists $N(\delta, \varepsilon)$, such that for all $N>N(\delta,\varepsilon)$, for every set of integers $\Lambda\subset [N]$ with $|\Lambda|\geq \delta N$, there is some $n\in \N$ such that $\Lambda$ contains at least $(1-\varepsilon)\delta^{\ell+1}N$ patterns of the form $m,m+a_1(n),\ldots, m+a_\ell(n)$.
 Some articles in the  literature, refer to such values of $n\in \N$ as ``popular differences''.

\subsection{Joint ergodicity of Hardy field sequences}\label{SS:16}
In this subsection we record some rather  straightforward applications of the
 previous results  for
collections of  Hardy field sequences (defined  in  Section~\ref{SS:Hardy}) that satisfy some
growth assumptions that we define next. Our method also gives alternative  proofs of known results from \cite{BMR21, Fr10, FrK05a, KK19} that avoid the  theory of characteristic factors and equidistribution results on nilmanifolds. For the purposes of this subsection we assume that all Hardy fields $\mathcal{H}$ considered have
the following property:
\begin{equation}\label{E:Hardy}
\text{if } a,b\in \mathcal{H }, \text{ then } a\circ b^{-1}\in \mathcal{H } \text{ and } a(\cdot+h)\in \mathcal{H}  \text{ for every } h\in \R_+.
\end{equation}
It is known that there
exists a Hardy field that has this property and contains all logarithmic-exponential functions (the example of Pfaffian functions is mentioned in \cite{Bos94});  so this covers all the interesting examples considered in practice, like
functions of the form $t^{c_1}(\log{t})^{c_2} (\log\log{t})^{c_3}$ where $c_1,c_2,c_3\in \R$.


\begin{definition}
We say that a function $a\colon \R_+\to \R$
\begin{enumerate}


\item is {\em tempered} if
	$t^k\log{t}\prec a(t)\prec t^{k+1}$ for some $k\in \Z_+$.

\item {\em stays logarithmically away from rational polynomials,} if for every  $p\in \Q[t]$ we have
$$
\lim_{t\to +\infty} \frac{|a(t)-p(t)|}{\log{t}}=+\infty.
$$
 \end{enumerate}
We denote by $\mathcal{T}$
the class of all  tempered functions and by $\mathcal{T}+\mathcal{P}$ the class of all linear combinations of   tempered and polynomial functions with real coefficients.
\end{definition}

It is known that if a sequence from a Hardy field has at most  polynomial growth and
stays logarithmically away from constant multiples of rational polynomials, then it is ergodic; this follows quite easily from the spectral theorem  and \cite{Bos94} (see the discussion around \eqref{E:Spectral}), the details can be found in  \cite[Theorem~A]{BKQW05}.
A possible generalization of this result that covers finite collections of Hardy field sequences is the following:
\begin{conjecture}[\cite{Fr10}, \cite{Fr16}]\label{Conj1}
	Let $a_1,\ldots, a_\ell\colon [c,+\infty)\to \R$  be functions from a Hardy field  of at most polynomial growth   such  that  every non-trivial linear combination of these functions
	stays logarithmically away from rational polynomials. Then the sequences $[a_1(n)],\ldots, [a_\ell(n)]$ are jointly ergodic.
\end{conjecture}
A variant of the conjecture that uses appropriately chosen weighted averages
in place of Ces\`aro averages was recently verified in \cite[Theorem~B]{BMR21}.
In the form originally  stated
 the conjecture
has been verified under  either of the following, partially complementary, additional assumptions:\footnote{Note that conditions  $(a)$ and $(b)$ are in general position;    the collections   $\{t^c, t^{c-1}+\log{t}\}$,  $\{t^c+t\log{t},t^{c-1}\}$, where $c\in (1,\infty)\setminus \Z_+$,
	satisfy $(a)$ but not $(b)$, and the collections  $\{t^a,t^a+t^b\}$, where $a,b\in \R_+\setminus \N$ and $b<a$,  satisfy $(b)$ but not $(a)$.}
\begin{enumerate}[(a)]
\item The functions $a_1,\ldots, a_\ell$ are tempered and  have different  growth rates, meaning $a_1\prec a_2\prec\cdots \prec a_\ell$,  see  \cite[Theorem~2.3]{Fr15} (see also \cite[Theorem~2.6]{Fr10} for a simpler argument that works for a slightly more restricted class of functions).
	
\item  The functions $\sum_{j=1}^\ell c_j a_j^{(k_j)}$  are in $\mathcal{T}+\mathcal{P}$  for all $c_1,\ldots, c_\ell \in \R$ and $k_1,\ldots, k_\ell\in \Z_+$, see \cite[Corollary~B1]{BMR21}.
\end{enumerate}

The next result improves upon the previously mentioned results
and 
 establishes Conjecture~\ref{Conj1} when the  functions $a_1,\ldots, a_\ell$  and their differences are in $\mathcal{T}+\mathcal{P}$.
Its proof is primarily based on Theorem~\ref{T:JointErgodicity}.
 \begin{theorem}\label{T:Hardy}
 	Let $a_1,\ldots, a_\ell\colon [c,+\infty)\to \R$  be functions from a Hardy field. Suppose that the functions and their differences are in $\mathcal{T}+\mathcal{P}$ and  every non-trivial linear combination of $a_1,\ldots, a_\ell$
 	stays logarithmically away from rational polynomials.
 	Then the sequences $[a_1(n)],\ldots, [a_\ell(n)]$ are jointly ergodic.
 \end{theorem}
 One easily verifies that the result applies for example if  $a_1,\ldots, a_\ell$ are
 \begin{enumerate}[(a)]
 \item  functions  of the form $\sum_{i=1}^k\alpha_it^{b_i}$ where $\alpha_1,\ldots, \alpha_k\in \R$, $b_1,\ldots, b_k\in (0,+\infty)$, and every non-trivial linear combination
 of $a_1,\ldots, a_\ell$
 is not a rational polynomial (in  \cite{KK19} this was established when $b_1,\ldots, b_k\in \Z_+$).

 \item  linearly independent functions  of  the form $\sum_{i=1}^k\alpha_i \, t^{b_i}(\log{t})^{c_i}$, where  $b_i\in \R_+\setminus\Z_+$,  $\alpha_i, c_i\in \R$, $i=1,\ldots, k$.
 \end{enumerate}
Both cases were also previously covered by  \cite[Corollary~B1]{BMR21}.\footnote{An example of a collection of functions  that satisfy the
 assumptions of Theorem~\ref{T:Hardy} but is not covered by previous results is
 $\{t^a,t^b, t^a+t^b +t^k(\log{t})^c \}$, where $a,b\in (1,\infty)\setminus \N$ are distinct, $k\in \N$ is smaller than $\max\{a,b\}$, and $c\in (0,1)$.} Let us also remark that our method of proof reduces Conjecture~1 to showing that under the assumptions of this conjecture the sequences
  $[a_1(n)],\ldots, [a_\ell(n)]$ are good for seminorm estimates; but verifying this property appears to be a challenging problem.

 We also give a result that applies to a larger class of sequences  that includes polynomial sequences. The proof of Part~$(i)$ is based on Theorem~\ref{T:JointErgodicity}  and the proof of Part~$(ii)$ is based on Theorem~\ref{T:CharacteristicFactors}.
\begin{theorem}\label{T:HardyPolyChar}
Let $a_1,\ldots, a_\ell\colon [c,+\infty)\to \R$  be  functions from a Hardy field. Suppose that the functions and their differences are in $\mathcal{T}+\mathcal{P}$,  and that every non-trivial linear combination of $a_1,\ldots, a_\ell$, with at least one coefficient irrational, stays logarithmically away from rational polynomials.
Then
\begin{enumerate}
\item the sequences $[a_1(n)],\ldots, [a_\ell(n)]$ are  jointly ergodic for totally ergodic systems.

\item   the rational Kronecker factor is characteristic for $[a_1(n)],\ldots, [a_\ell(n)]$.
\end{enumerate}
\end{theorem}
One easily verifies that the result applies for example if  $a_1,\ldots, a_\ell$ are
\begin{enumerate}[(a)]
\item  linearly independent functions of the form    $\sum_{i=1}^k\alpha_it^{b_i}$ where  $\alpha_1,\ldots, \alpha_k\in \Q$ and  $b_1,\ldots, b_k \in (0,+\infty)$  (this generalizes the main result   of \cite{FrK05a}, which covers the case of integer polynomials), or, more generally,

\item functions  of the form    $\sum_{i=1}^k\alpha_it^{b_i}$ where  $\alpha_1,\ldots, \alpha_k\in \R$ and  $b_1,\ldots, b_k\in  (0,+\infty)$, such that every linear combination of $a_1,\ldots, a_\ell$, with at least one coefficient irrational, is not a rational polynomial. We remark that linear independence does not suffice in this case; to see this,  take $k=1$ and $a_1(t)=\alpha t$ where $\alpha$ is irrational, then the sequence $[a_1(n)]$ is not ergodic for the (totally ergodic) system $(\T,m_\T, R)$, where $Rx=x+1/\alpha$, $x\in \T$.
\end{enumerate}
The previous special cases  resolve \cite[Conjecture~6.1]{BMR21}.

\subsection{Joint ergodicity for nilsystems}\label{SS:17}
 Combining  Corollary~\ref{C:Nilsystems} and known equidistribution results on the circle  from \cite{BBK02, BK90, BP98, Kar71},
we can  deduce some equidistribution results on general nilmanifolds. These results  do not appear to be
accessible from  known techniques related to qualitative or quantitative equidistribution  of polynomial sequences on nilmanifolds, since these techniques require the functions involved to have some derivative vanishing (in the sense that it converges to zero at infinity).

\begin{theorem}\label{T:NilsystemSpecial}
	The following collections of sequences are jointly ergodic for every ergodic 
	nilsystem:
	\begin{enumerate}
		\item $[n^{(\log{n})^{c_1}}], \ldots, [n^{(\log{n})^{c_\ell}}]$, where $c_1,\ldots, c_\ell\in (0,1/2)$ are distinct.
		
	
	\item $[p_1(n)\sin{n}],\ldots, [p_\ell(n)\sin{n}]$,  where $1,p_1,\ldots, p_\ell\in \Z[t]$ are linearly independent polynomials.
	
	\item $[n^k\sin{n}], [n^k\sin(2n)], \ldots, [n^k\sin(\ell n)]$, where $k,\ell\in \N$.
	\end{enumerate}
\end{theorem}
\begin{remark}
	We do not know if for these sequences equation \eqref{E:qwe} holds pointwise for every $x\in X$ (or even for $m_X$-almost every $x\in X$) and all $f_1,\ldots, f_\ell\in C(X)$.
	\end{remark}
It follows from Theorem~\ref{T:JointErgodicity} that joint ergodicity (for arbitrary ergodic systems) of  the previous collections of sequences would follow if one establishes that they are good for seminorm estimates; see Problem~\ref{Problem2} below.

\subsection{Joint ergodicity for flows}
A $1$-parameter measure preserving action (or flow) on a probability space $(X,\mathcal{X},\mu)$ is a  family $T^t$, $t\in \R$,
of invertible measure preserving transformations,
that satisfy $T^{s+t}=T^s\circ T^t$, $s,t\in \R$ ($T^0$ is the identity transformation), and
such that the map
$\R\times X\to X$, defined by $(t,x)\mapsto T^tx$, is measurable ($X$ is equipped with $\mathcal{X}$ and $\R$ with the Borel $\sigma$-algebra).

We give a joint ergodicity result  for  not necessarily commuting  measure preserving flows, when the iterates
are given by functions that satisfy suitable growth conditions. It turns out that in the case of flows we can also prove pointwise convergence results; in contrast, for discrete time actions the corresponding results are false
without any commutativity assumptions (even for mean convergence, see \cite[Theorem~1.7]{FrLW11}) and for general commuting  actions they are unknown.

\begin{theorem} \label{T:flows}
	Let $a_1,\ldots, a_\ell\colon [c,+\infty)\to \R_+$  be functions from a Hardy field. Suppose that there exists  
	$\delta>0$ such that  $t^\delta\prec a_1(t)$ and
	$(a_{j+1}(t))^\delta\prec a_j(t)\prec (a_{j+1}(t))^{1-\delta}$ for $j=1,\ldots, \ell-1$.
	Then for all measure preserving actions  $T_1^t,\ldots, T_\ell^t$, $t\in\R$,  on a probability space $(X,\mathcal{X},\mu)$ and $f_1,\ldots, f_\ell\in L^\infty(\mu)$, we have
	\begin{equation}\label{E:flows}
	\lim_{y\to+\infty} \frac{1}{y} \int_0^y f_1(T_1^{a_1(t)}x)\cdot\ldots\cdot f_\ell(T_\ell^{a_{\ell}(t)}x)\, dt=\tilde{f_1}\cdots \tilde{f_\ell}
	\end{equation}
	pointwise for $\mu$-almost every $x\in X$, where for $j=1,\ldots, \ell$ we denote by  $\tilde{f_j}$  the orthogonal projection of $f_j$ on the space of functions that are $T_j^t$-invariant for every $t\in \R$.
\end{theorem}
\begin{remarks}
$\bullet$ When $\ell=2$ and  $T_1,T_2$ commute, the result was established in \cite{CDKR21} (for $a_1(t)=t$, $a_2(t)=t^2$, but for $\ell=2$ the argument given there probably extends to the more general class of functions that we consider). The argument in \cite{CDKR21} uses  the  transference principle of Calder\'on in order to get access to a harmonic analysis result from \cite{CDR21}; so it does not seem to extend to the case of non-commuting flows or to the case $\ell\geq 3$.


$\bullet$ If  the transformations $T_1,\ldots, T_\ell$ commute, and $a_1,\ldots, a_\ell$ are given by polynomials,  then  mean convergence for \eqref{E:flows} was established in \cite{Au11c} without any growth assumptions (and previously in \cite{Po11} when $T_1=\cdots=T_\ell$). Moreover, other related convergence results for flows
can be found in \cite{BLM11}.
\end{remarks}

The result applies for the collection of functions $t^{c_1},\ldots, t^{c_\ell}$ where $c_1,\ldots, c_\ell\in (0,+\infty)$ are distinct (or linear combinations of such functions, assuming they have different growth), and  also applies to some collections of fast growing functions like $b_1^t,\ldots, b_\ell^t$, where $b_1,\ldots, b_\ell\in (1,+\infty)$ are distinct.
But our method does not work for collections of functions with the  same or ``not substantially different growth'', like $t^2,t^2+t$ and $t, t\log{t}$,  or for functions with  ``very different growth'', like $t,2^t$.

The method we use to prove Theorem~\ref{T:flows} is different (and in fact much simpler) than the one used for
the results on discrete time averages given on previous subsections.
 We take advantage of the fact that functions like $t\mapsto t^2$ are onto   $\R_+$ (a property that crucially fails in $\Z_+$); this  enables us to  make a change of  variables (see Lemma~\ref{L:CV})
and reduce matters to the case where all the iterates have sublinear growth. This is an important simplification  and is what allows us to prove much stronger results than those for discrete-time actions.

\subsection{Open problems}
Theorem~\ref{T:JointErgodicity}  shows that  the sequences $a_1,\ldots, a_\ell\colon \N\to \Z$ are jointly ergodic if and only if
they are $(i)$ good for  seminorm estimates   and $(ii)$ good for equidistribution. It
 is a tantalizing problem to determine whether the  equidistribution assumption alone suffices to
deduce joint ergodicity (in which case condition $(i)$ can be inferred from condition $(ii)$).
Although at first this may seem like an unlikely scenario, we have not been able to construct a counterexample, which leaves the following question open:
\begin{problem}\label{Problem1}
If the sequences $a_1,\ldots, a_\ell\colon \N\to \Z$ are good  for equidistribution, are they
 jointly ergodic?
\end{problem}
\begin{remark}
 The corresponding question for pointwise convergence has a negative answer even when $\ell=1$. For instance if $a(n):=n^3+[(\log{n})^a]$, $n\in \N$, where $a\in \big(1,\frac{6}{5}\big)$, then it follows from \cite[Theorem~3.2]{BKQW05}  that the sequence $a(n)$ is good for the mean ergodic theorem (hence good for equidistribution), and from \cite[Theorem~3.7]{BKQW05} it follows that it is bad for the pointwise ergodic theorem.   So one has to impose  stronger quantitative equidistribution assumptions on the sequences in order to get  legitimate pointwise convergence criteria.
\end{remark}
In support of a positive answer to the previous question let us mention that the answer is positive for $\ell=1$; this follows from the spectral theorem (see the discussion around \eqref{E:Spectral}). Moreover, Corollary~\ref{C:Nilsystems} implies that if the   sequences $a_1,\ldots, a_\ell$ are good for equidistribution, then they are jointly ergodic for nilsystems.

We also record a problem regarding special sequences arising from smooth functions without a vanishing derivative.
\begin{problem}\label{Problem2}
Show that  the collections of the sequences in Theorem~\ref{T:NilsystemSpecial} are
 good for seminorm estimates.
\end{problem}
It is known that these collections of sequences  are good for equidistribution (see the  proof of Theorem~\ref{T:NilsystemSpecial} for relevant references), so solving this problem for any of the examples given in Theorem~\ref{T:NilsystemSpecial} would imply (by Theorem~\ref{T:JointErgodicity}) that the corresponding collection  of sequences is jointly ergodic. We remark that  related multiple recurrence results and combinatorial consequences are largely unknown for these collections of sequences.

\subsection{Recent developments}   We give a summary of some applications of the results and techniques of this article that took place after the first version of this article appeared on the arXiv.
\begin{itemize}
\item In \cite{BM21},  Best and Ferr\'e Moragues extended   Theorem~\ref{T:JointErgodicity}
 to the setting of certain countable ring actions, which include countable  fields of characteristic zero and also rings of integers of  number fields.

\item  In \cite{DKS19}, Donoso, Koutsogiannis, and Sun, and in \cite{DFKS21} the same authors together with Ferr\'e Moragues, they  used  the extension of Theorem~\ref{T:JointErgodicity} given in  \cite{BM21}, among other things,  to obtain several
joint ergodicity results for commuting transformations with polynomial iterates.

\item In \cite{Fr21b}, the author used Theorem~\ref{T:JointErgodicity} to prove joint ergodicity results for sequences given by  fractional powers of the prime numbers.

\item  In \cite{Ts21}, Tsinas used Theorem~\ref{T:JointErgodicity}   to answer conjectures from \cite{Fr10} and \cite{Fr16} concerning joint ergodicity properties of Hardy field sequences, in particular, he established Conjecture~\ref{Conj1} of Section~\ref{SS:16}.

\item Finally, in \cite{FK22},  the author and Kuca extended Theorem~\ref{T:JointErgodicity}  and obtained joint ergodicity criteria for commuting transformations. These criteria were then used, together with other results,
to  answer conjectures concerning joint ergodicity properties of independent polynomial sequences. 
\end{itemize}

\subsection{Acknowledgement} I would like to thank the referee for useful remarks and suggestions.

 \section{The Furstenberg-Weiss theorem} \label{S:FW}

The goal of this section is to explain the basic principles used in the proof of our main results  in a simple yet  interesting setting that allows us to  avoid much of the technicalities that appear in the proof of    Theorem~\ref{T:JointErgodicity} and Theorem~\ref{T:CharacteristicFactors}.
We do this by giving a proof of the following result that was originally proved by Furstenberg and Weiss:\footnote{Another reason why we choose to
	prove Theorem~\ref{T:FW}  separately is in order to have an ``elementary proof'' on record, since it is not covered by  Theorem~\ref{T:JointErgodicity} (it is covered by  Theorem~\ref{T:CharacteristicFactors}, but this result depends on  deep results from \cite{HK05a}).}
\begin{theorem}[Furstenberg-Weiss~\cite{FuW96}]\label{T:FW}
 Let $(X,\mu,T)$ be a system and $f_1,f_2\in L^\infty(\mu)$. If either $f_1$ or $f_2$ is orthogonal to $\mathcal{K}_{rat}(T)$, then
 \begin{equation}\label{E:aa0}
\lim_{N\to\infty} \E_{n\in[N]} \, T^nf_1\cdot T^{n^2}f_2=0
 \end{equation}
 in $L^2(\mu)$.
 \end{theorem}
It is easy  to deduce from the previous result  a related multiple recurrence statement and then, via the correspondence principle of Furstenberg,  conclude that any set of integers with positive upper density contains patterns of the form $m,m+n,m+n^2$, for some $m,n\in\N$. This combinatorial  result was first proved in \cite{BL96} using ergodic theory and a proof that avoids ergodic theory (and produces reasonable quantitative bounds) was given in \cite{PP19}.

\subsection{Proof strategy}
Let us briefly describe the main idea of the proof, which is an adaptation of a technique used in \cite{PP19} to our ergodic setup, taking also into account various simplifications that our infinitary setting allows. If
\eqref{E:aa0} fails, then using  standard arguments,  we deduce that $\nnorm{\tilde{f}_2}_3>0$ where (the next limit is a weak limit)
$$
\tilde{f}_2:=\lim_{k\to\infty} \E_{n\in [N_k]} \, T^{-n^2}g_k\cdot T^{-n^2+n}\overline{f}_1
$$
for some $g_k\in L^\infty(\mu)$, bounded by $1$, and $N_k\to\infty$. The reader will find the details in Steps~1 and 2 below.
This then easily implies that
 \begin{equation}\label{E:Delta}
\liminf_{N\to\infty} \E_{n_1\in[N]}
\, \Re \Big(\int \Delta_{n_1}\tilde{f}_2\cdot \chi_{n_1}\, d\mu\Big)>0,
\end{equation}
where $\Delta_{n_1}f=T^{n_1}f\cdot \overline{f}$, $n_1\in\N$,   and $\chi_{n_1}$, $n_1\in\N$,  are  appropriate eigenfunctions of the system $(X,\mu,T)$ with unit modulus.
Suppose for convenience that the operators $\Delta_{n_1}$ and $\E_{n\in [N_k]}$ commute; this is of course not true, but another  convenient property holds and can be used as a substitute. Using this simplifying assumption we have
$$
\liminf_{N\to\infty} \E_{n_1\in[N]} \lim_{k\to\infty}\E_{n\in [N_k]}
\, \Re \Big(\int T^{-n^2}(\Delta_{n_1}g_k) \cdot T^{-n^2+n}(\Delta_{n_1}\overline{f}_1) \cdot \chi_{n_1}\, d\mu\Big)>0.
$$
At this point we have a much simpler problem to work with. Indeed,
after composing  with $T^{n^2}$ and using that $T^{n^2}\chi_{n_1}=e(n^2\alpha_{n_1})\cdot \chi_{n_1}$, for some $\alpha_{n_1}\in \spec(T)$, $n_1\in\N$, and then using the Cauchy-Schwarz inequality, we deduce that
$$
\liminf_{N\to\infty} \E_{n_1\in[N]} \limsup_{k\to\infty}\norm{\E_{n\in [N_k]}
\,  e(n^2\alpha_{n_1}) \cdot T^n(\Delta_{n_1}\overline{f}_1)}_{L^2(\mu)}> 0.
$$
Using the spectral theorem for unitary operators\footnote{ The consequence we use is that for every system $(X,\mu,T)$ and $f\in L^2(\mu)$ there exists a positive bounded measure $\sigma$ on $\T$  such that $\norm{\sum_{n=1}^N\, c_n\, T^nf}_{L^2(\mu)}=\norm{\sum_{n=1}^Nc_n\, e(nt)}_{L^2(\sigma(t))}$ for all $c_1,\ldots, c_N\in \C$.}
 and then Fatou's Lemma, we get that for some positive measures $\sigma_{n_1}$ on $\T$, $n_1\in \N$,  with uniformly bounded total mass, we have
$$
\liminf_{N\to\infty} \E_{n_1\in[N]}  \norm{\limsup_{k\to\infty} \big|\E_{n\in [N_k]}
\, e(n^2\alpha_{n_1}+nt)\big|}_{L^2(\sigma_{n_1}(t))}> 0.
$$
Using  the  good equidistribution properties of sequences of the form $n\alpha+n^2\beta\pmod{1}$ for $\beta$ irrational, we deduce  that the numbers $\alpha_{n_1}, n_1\in\N$, belong to a finite set of rationals. We conclude that for appropriate $\Lambda_N\subset [N]$, $N\in \N$,  with  $\liminf_{N\to\infty}|\Lambda_N|/N>0$, upon replacing $\E_{n\in[N]}$ with  $\E_{n\in \Lambda_N}$ (which we can do by being more careful),
equation  \eqref{E:Delta} holds with $\chi_{n_1}, n_1\in \Lambda_N,$ that depend only on $N\in \N$.
 It is then pretty straightforward  to deduce that
 $\nnorm{\tilde{f}_2}_2>0$ (as in the last part of Step~4 below). The details of the previous argument and the necessary adjustments,   are given in Steps~3 and 4.
  Given this, it is an easy matter to deduce in Step~5 that the functions $f_1$ and $f_2$ are not orthogonal to $\mathcal{K}_{rat}(T)$, contradicting the assumptions of Theorem~\ref{T:FW}.

 \subsection{Step 1} (Characteristic factors)
 Our first step is to show that if
 \eqref{E:aa0} fails, then it also fails for some function  $f_2$ of special form. More precisely the following holds:
 \begin{proposition}\label{P:1}
 Let $(X,\mu,T)$ be a system and $f_1,f_2\in L^\infty(\mu)$ be such that
 \begin{equation}\label{E:aa1}
\limsup_{N\to\infty}\norm{ \E_{n\in [N]} \, T^nf_1\cdot T^{n^2}f_2}_{L^2(\mu)}> 0.
 \end{equation}
  Then  there exist $N_k\to\infty$ and $g_k\in L^\infty(\mu)$, with  $\norm{g_k}_{L^\infty(\mu)}\leq 1$,  $k\in\N$, such that
 for
 \begin{equation}\label{E:aa2}
 \tilde{f}_2:=\lim_{k\to\infty} \E_{n\in [N_k]} \, T^{-n^2}g_k\cdot T^{-n^2+n}\overline{f}_1,
 \end{equation}
 where the limit is a weak limit (note that then $\tilde{f}_2\in L^\infty(\mu)$), we have
 \begin{equation}\label{E:aa3}
 \limsup_{N\to\infty}\norm{ \E_{n\in [N]} \, T^nf_1\cdot T^{n^2}\tilde{f}_2}_{L^2(\mu)} >0.
 \end{equation}
 \end{proposition}
\begin{proof} We can assume that both $f_1$ and $f_2$ are bounded by $1$.
For fixed $f_1\in L^\infty(\mu)$ we let $\mathcal{C}=\mathcal{C}(f_1)$ be the $L^2(\mu)$ closure  of  all  linear combinations of all
 subsequential  weak-limits of sequences of the form
$$
\E_{n\in [N]} \, T^{-n^2}g_N\cdot T^{-n^2+n}\overline{f}_1,
$$
where $g_N\in L^\infty(\mu)$, $N\in \N$,   are all bounded by $1$.

 We first  claim that if  $h\in L^\infty(\mu)$  is orthogonal to the subspace $\mathcal{C}$, then
 $$
 \lim_{N\to\infty}  \E_{n\in [N]} \, T^nf_1\cdot T^{n^2}h=0
 $$
 in $L^2(\mu)$.
 Indeed, if this is not the case, then  there exist  $a>0$ and $N_k\to \infty$  such that
 $$
 \norm{ \E_{n\in [N_k]} \, T^nf_1\cdot T^{n^2}h}_{L^2(\mu)}\geq a, \quad k\in \N.
 $$
 If we define the functions
 $$
 g_k:=
  \E_{n\in [N_k]} \, T^nf_1\cdot T^{n^2}h, \quad k\in \N,
 $$
 which all have  $L^\infty(\mu)$-norm bounded by $\norm{h}_{L^\infty(\mu)}$, we deduce that
 \begin{equation}\label{E:a2}
 \E_{n\in [N_k]} \, \int \overline{g}_k\cdot  T^nf_1\cdot T^{n^2}h\, d\mu \geq a^2,\quad k\in \N.
 \end{equation}
 By passing to a subsequence, we can assume that the averages
$ \E_{n\in [N_k]} \, T^{-n^2} g_k\cdot T^{-n^2+n}\overline{f}_1$, which are functions bounded by $\norm{h}_{L^\infty(\mu)}$,  converge in the weak topology as $k\to\infty$ say to a function $f\in \mathcal{C}$.
Then  composing with $T^{-n^2}$ in \eqref{E:a2} we deduce that
$$
\int h\cdot \overline{f}\,  d\mu\neq 0,
$$
contradicting our assumption  that $h$ is orthogonal to the subspace $\mathcal{C}$. This proves our claim.

By applying the previous claim for $h:=f_2-\E(f_2|\mathcal{C})$, we conclude that
$$
 \E_{n\in [N]} \, T^nf_1\cdot T^{n^2}(f_2-\E(f_2|\mathcal{C}))\to 0
$$
in $L^2(\mu)$ as $N\to \infty$,
where $\E(f_2|\mathcal{C})$ denotes the orthogonal projection  of $f_2$ onto the closed subspace $\mathcal{C}$.
Hence, if \eqref{E:aa1} holds,
 then
$$
 \E_{n\in [N]} \, T^nf_1 \cdot T^{n^2}\E(f_2|\mathcal{C}) \not\to 0
 $$
 in $L^2(\mu)$ as $N\to \infty$.
 Using the definition of $\mathcal{C}$ and an approximation argument, we get that  there exist $N_k\to\infty$ and $g_k\in L^\infty(\mu)$, $k\in\N$, all bounded by $1$, such that for $\tilde{f}_2$ as in \eqref{E:aa2} we have that
 \eqref{E:aa3} holds. Lastly,  since $f_1$ and $g_k$, $k\in \N$, all have $L^\infty(\mu)$ norm bounded by $1$, the same holds for $\tilde{f}_2$. This completes the proof.
 \end{proof}

 \subsection{Step 2} (Seminorm estimates) We state some seminorm estimates that were proved in a slightly different form in \cite{FuW96}. The technique is standard and we sketch it for completeness.
  \begin{proposition}\label{P:2}
 	Let $(X,\mu,T)$ be an ergodic  system and $f_1,f_2\in L^\infty(\mu)$ be such that   $\nnorm{f_2}_{3}=0$. Then
 	\begin{equation}\label{E:aa1'}
 	\lim_{N\to\infty}\E_{n\in[N]} \, T^nf_1\cdot T^{n^2}f_2=0
 	\end{equation}
 	where convergence takes place in $L^2(\mu)$.
 \end{proposition}
 \begin{proof}
 		Using   the van der Corput Lemma (see for example  \cite[Lemma~1.1]{FuW96}), composing with $T^{-n}$,  and then using 	
 		the Cauchy-Schwarz inequality,   we get that instead of \eqref{E:aa1'} it  suffices to show that for every  $m\in \N$ we have
 	$$
 	\lim_{N\to \infty}  \E_{n\in [N]} \,   T^{(n+m)^2-n}f_2\cdot T^{n^2-n}\overline{f}_2= 0
 	$$
 	in $L^2(\mu)$. 
 	Using   the van der Corput lemma again, composing with $T^{-n^2}$,  and then using 	
 	the Cauchy-Schwarz inequality,   we further reduce matters to showing  the following:  If  $\nnorm{f_2}_{3}=0$, then for all  $g,h\in L^\infty(\mu)$ and all   $a,b,c\in \N$  with $a\neq b, a\neq c$, we have
 		\begin{equation}\label{E:abc}
 	\lim_{N\to \infty}  \E_{n\in [N]} \,   T^{an}f_2\cdot T^{bn}g\cdot T^{cn}h= 0
 	\end{equation}
 	in $L^2(\mu)$. 
 	 This follows from  \cite[Theorem~8]{L05'}.\footnote{The argument in \cite{L05'} is non-trivial. One can avoid  it by using that  $\nnorm{f_2}_4=0$ implies the identity \eqref{E:abc}  (see for example \cite[Chapter~21, Proposition~7]{HK18}). The drawback is that  the use of the $4$-th seminorm  would complicate our subsequent arguments a bit; but still a proof that avoids deep machinery can be given by using the argument in  Section~\ref{S:Main}.}
\end{proof}

 \subsection{Step 3} (Seminorms of averages of functions)
 Our next goal is to show  that if the seminorm of an average of functions is positive, then some related positiveness property holds for the individual functions. This is a crucial property for our argument;  we state it here only in the form needed for the proof of Theorem~\ref{T:FW}, and prove the extension needed for Theorem~\ref{T:JointErgodicity} in Proposition~\ref{P:3'} below.

\begin{definition}  If $(X,\mu,T)$ is a system and  $f\in L^\infty(\mu)$, then for $n\in\Z$ we let $\Delta_nf:=T^nf\cdot \overline{f}$.
\end{definition}

\begin{proposition}\label{P:3}
	Let $(X,\mu,T)$ be an ergodic system, $f_{n,k}\in L^\infty(\mu)$, $n,k\in \N$, be bounded by $1$,  and $f\in L^\infty(\mu)$ be defined by
	$$
	f:=\, \lim_{k\to\infty}\E_{n\in [N_k]}\, f_{n,k},
	$$
	for some $N_k\to\infty$, where the average is assumed to converge weakly.
	If 	$\nnorm{f}_{3}>0$,
		then there exist  $a>0$, a subset $\Lambda$ of $\N$ with positive lower density, and $\chi_{n_1}\in \mathcal{E}(T)$, $n_1\in \N$,  such that
	$$
 \Re \Big( \int \Delta_{n_1}f \cdot \chi_{n_1}\, d\mu\Big)>a, \quad  n_1\in \Lambda,
	$$
	and 	
		$$
	\liminf_{N\to\infty}\E_{n_1,n_1'\in \Lambda\cap [N]}\,  \limsup_{k\to\infty}\E_{n\in[N_k]}\,  \Re \Big(\int \Delta_{n_1-n_1'}f_{n,k} \cdot T^{-n_1'}(\chi_{n_1} \cdot \overline{\chi}_{n_1'})\, d\mu\Big)>0.
	$$
\end{proposition}
\begin{remark}
	We plan to apply this proposition in the next step for the  function $\tilde{f}_2$ given by \eqref{E:aa2} in place of $f$.
	\end{remark}
\begin{proof}	
By \eqref{E:seminorm4}, which is proved below, we have  $\nnorm{f}_3^8=\lim_{N\to\infty}\E_{n_1\in[N]}\nnorm{\Delta_{n_1}f}^4_2>0$. We deduce from Proposition~\ref{P:1'} below
 that there exist $\chi_{n_1}\in \mathcal{E}(T)$, $n_1\in\N$, such that
$$
 \liminf_{N\to\infty}\E_{n_1\in[N]} \,
 \Re \Big(\int \Delta_{n_1}f\cdot \chi_{n_1}\, d\mu\Big)>0.
$$
Hence,  there exist  $a>0$, and a subset $\Lambda$ of $\N$ with positive lower density,
such that
	$$
 \Re \Big( \int \Delta_{n_1}f \cdot \chi_{n_1}\, d\mu\Big)>a,  \quad n_1\in \Lambda.
	$$
Then we have
	$$
	\liminf_{N\to\infty}\E_{n_1\in \Lambda\cap [N]}\,  \Re \Big( \int \Delta_{n_1}f \cdot \chi_{n_1}\, d\mu\Big)>0.
	$$
	Since $f=\lim_{k\to\infty}\E_{n\in [N_k]}\, f_{n,k}$ (the limit is a weak limit), we deduce that
		$$
	\liminf_{N\to\infty}\E_{n_1\in \Lambda\cap [N]}\,  \lim_{k\to\infty}\E_{n\in [N_k]}\, \Re \Big( \int  T^{n_1}f_{n,k} \cdot \overline{f}\cdot  \chi_{n_1}\, d\mu\Big)>0.
	$$
	Since all the limits $ \lim_{k\to\infty}\E_{n\in [N_k]}$ exist, we can interchange the finite average $\E_{n_1\in\Lambda\cap [N]}$ with $ \lim_{k\to\infty}\E_{n\in [N_k]}$, and  after using the Cauchy-Schwarz inequality we deduce that
		$$
	\liminf_{N\to\infty}  \limsup_{k\to\infty}\E_{n\in[N_k]}\int |\E_{n_1\in\Lambda\cap [N]}\, T^{n_1}f_{n,k} \cdot  \chi_{n_1}|^2\, d\mu>0.
	$$
 Expanding the square, composing with $T^{-n_1'}$,  and using that the limsup of a finite sum is at most the sum of the limsups, we get
 	$$
	\liminf_{N\to\infty}  \E_{n_1,n_1'\in \Lambda\cap [N]} \limsup_{k\to\infty}\E_{n\in[N_k]} \Re \Big(\int \overline{f}_{n,k} \cdot T^{n_1-n_1'}f_{n,k} \cdot T^{-n_1'}(\chi_{n_1} \cdot \overline{\chi}_{n_1'})\, d\mu\Big)>0.
	$$
	This completes the proof.
	\end{proof}

 \subsection{Step 4}\label{SS:2.5} (Going from $\nnorm{\cdot}_{3}$ to $\nnorm{\cdot}_{2}$)
We will use the following elementary fact:
\begin{lemma}\label{L:VDC}
 	Let  $N\in \N$ and $v_1,\ldots, v_N$ be elements of an inner product space
	of norm at most $1$. Then
	$$
	\norm{\E_{n\in[N]} \,
		v_n}^2\leq  2\, \E_{m\in [N]}
	\Re\Big(\frac{1}{N}\sum_{ n=1}^{N-m}
	\langle v_{n+m}, v_n \rangle \Big) + \frac{1}{N},
	$$
where $\Re(z)$ denotes the real part of the complex number $z$.
 \end{lemma}
\begin{proof}
For fixed $n\in\N$ we note that
 $$
 \norm{\E_{n\in[N]}\, v_n}^2 = \frac{1}{N^2}\sum_{m,n\in[N]}\langle v_m, v_n\rangle.
 $$
 We split the sum on the right  into three terms, depending on whether  $m<n$, $m=n$, and $m>n$. We make the  substitution $n = m+h$ on the first sum, $m=n+h$ on the third sum, and bound $v_n$ by $1$ on the second sum. The asserted estimate  follows. 	
\end{proof}
Our next goal is to use Proposition~\ref{P:3} and the particular form of the iterates defining the function $\tilde{f}_2$ in order to establish the following
result:
 \begin{proposition}\label{P:4}
Let $(X,\mu,T)$ be an ergodic system, $\tilde{f}_2$ be as in \eqref{E:aa2}, 
and suppose that $\nnorm{\tilde{f}_2}_{3}>0$. Then $\nnorm{\tilde{f}_2}_{2}>0$.
\end{proposition}
\begin{proof}
  Recall that  $$
  \tilde{f}_2:=\lim_{k\to\infty} \E_{n\in [N_k]} \, T^{-n^2}g_k\cdot T^{-n^2+n}\overline{f}_1,
  $$
  where the average converges weakly and the functions involved are uniformly bounded, say by $1$.
  Applying Proposition~\ref{P:3} for $f_{n,k}:= T^{-n^2}g_k\cdot T^{-n^2+n}\overline{f}_1$, $n,k\in\N$, we get that there exist $a>0$, a subset $\Lambda$ of $\N$ with lower density at least $a$, and   $\chi_{n_1}\in \mathcal{E}(T)$,  $n_1\in \N$, such that
	\begin{equation}\label{E:U2first}
	  \Re \Big( \int \Delta_{n_1}\tilde{f}_2 \cdot \chi_{n_1}\, d\mu\Big)>a, \quad  n_1\in \Lambda,
	\end{equation}
  and
		\begin{multline}\label{E:U2second}
	\liminf_{N\to\infty}\E_{n_1,n_1'\in \Lambda\cap [N]}\,  \limsup_{k\to\infty}\E_{n\in[N_k]}\\  \Re \Big(\int T^{-n^2}(\Delta_{n_1-n_1'}g_k) \cdot T^{-n^2+n}(\Delta_{n_1-n_1'}\overline{f}_1)\cdot T^{-n_1'} (\chi_{n_1} \cdot \overline{\chi}_{n_1'})\, d\mu\Big)>0.
	\end{multline}
	Our goal is to use \eqref{E:U2second}  in order to show that the eigenfunctions
	$\chi_{n_1}$,   $n_1\in [N]$, in \eqref{E:U2first} depend only on $N$; this combined with   Lemma~\ref{L:lower} below, will enable  us to deduce that  $\nnorm{\tilde{f}_2}_{2}>0$.
	
		We analyse the second estimate first  (the first will be used at the end of this proof).
	We compose with  $T^{n^2}$,  use  the identity
		$$
	T^{n^2} \chi_{n_1}=e(n^2\alpha_{n_1})\cdot \chi_{n_1}, \quad n,n_1\in \N,
	$$
	which holds for some $\alpha_{n_1}\in [0,1)$, $n_1\in \N$, and then use the Cauchy-Schwarz inequality. We deduce that
		\begin{equation}\label{E:nnL}
	\liminf_{N\to\infty}\E_{n_1,n_1'\in \Lambda\cap [N]}\,  \limsup_{k\to\infty}\norm{\E_{n\in[N_k]}\,  e(n^2(\alpha_{n_1}-\alpha_{n_1'}))\, T^{n}(\Delta_{n_1-n_1'}\overline{f}_1)}_{L^2(\mu)}>0.
	\end{equation}
	Using the spectral theorem for unitary operators and Fatou's Lemma, we deduce from \eqref{E:nnL} that there  exist  positive measures $\sigma_{n_1,n_1'}$ on $\T$, with total mass at most $1$,  such that
	\begin{equation}\label{E:gnn}
	\liminf_{N\to\infty}\E_{n_1,n_1'\in \Lambda\cap [N]}\, \norm{g_{n_1,n_1'}}_{L^2(\sigma_{n_1,n_1'})}>0,
	\end{equation}
	where
	$$
	g_{n_1,n_1'}(t):=\limsup_{k\to\infty} |\E_{n\in[N_k]} \, e(nt+n^2(\alpha_{n_1}-\alpha_{n_1'}))|, \quad n_1,n_1'\in \N,\,  t\in [0,1).
	$$
	We deduce from \eqref{E:gnn}  that there exist $a'>0$,     $n'_{1,N}\in[N]$, $N\in \N$,
	such  that
	$$
	\liminf_{N\to\infty}\E_{n_1\in \Lambda\cap [N]}\, \norm{g_{n_1,n_{1,N}'}}_{L^2(\sigma_{n_1,n_{1,N}'})}>a'.
	$$
Hence,  there exist
subsets $\Lambda_N$  of $\Lambda\cap [N]$, $N\in \N$, such that
$$
\liminf_{N\to\infty} \frac{|\Lambda_N|}{N}>0
$$
and
$$	
\norm{g_{n_1,n_{1,N}'}}_{L^2(\sigma_{n_1,n_{1,N}'})}>a',\quad n_1\in \Lambda_N, \, N\in \N.
$$
Since the measures have mass at most $1$, this  immediately implies that  there exist $t_{n_1,N}\in[0,1)$, $n_1\in \Lambda_N$, $N\in \N$, such that
$$
g_{n_1,n_{1,N}'}(t_{n_1,N})>a',\quad n_1\in \Lambda_N,\,  N\in \N,
$$
or, equivalently,
	\begin{equation}\label{E:nnL'}
\limsup_{k\to\infty} |\E_{n\in[N_k]}\,  e(n t_{n_1,N}+n^2(\alpha_{n_1}-\alpha_{n_{1,N}'}))|>a',\quad n_1\in \Lambda_N,\,  N\in \N.
\end{equation}

 Now using  Weyl type estimates, as in  Lemma~\ref{L:Weyl2} below, we deduce from \eqref{E:nnL'}
 that there exists a finite set of rationals $R=R(a')$ in $[0,1)$ such that
$$
 \alpha_{n_1}-\alpha_{n_{1,N}'}\in R \pmod{1},\quad   n_1\in \Lambda'_N,\,  N\in \N.
 $$
  We deduce from this that there exist subsets $\Lambda_N'$ of $\Lambda_N$, $N\in \N$, such that
\begin{equation}\label{E:ldens}
	\liminf_{N\to\infty} \frac{|\Lambda'_N|}{N}>0,
	\end{equation}
	and $r_N\in \R$  such that
 $\alpha_{n_1}=\beta_N:=r_N+\alpha_{n_{1,N}'}\pmod{1}$ for all $N\in \N$ and $n_1\in \Lambda'_N$.  For
$n_1\in \Lambda'_N$ we then have that $\chi_{n_1}$ is a $T$-eigenfunction with unit modulus and  eigenvalue $e(\beta_N)$, $N\in \N$, namely,
\begin{equation}\label{E:bettaN}
T\chi_{n_1}=e(\beta_N)\cdot \chi_{n_1}, \quad   n_1\in \Lambda'_N,\,  N\in \N.
\end{equation}


Using   \eqref{E:U2first}  and since $\Lambda_N'\subset \Lambda$ satisfies \eqref{E:ldens}, we have that
$$
\liminf_{N\to\infty}\E_{n_1\in  \Lambda'_N}\,  \Re\Big( \int \Delta_{n_1}\tilde{f}_2 \cdot \chi_{n_1}\, d\mu\Big)>0.
 $$
 We are going to deduce from this and \eqref{E:bettaN} that $\nnorm{\tilde{f}_2}_2>0$.
 Composing with $T^{n_1'}$, averaging over $n_1'\in [N]$,   using that $$
 T^{n_1'}\chi_{n_1}=e(n_1'\beta_N)\cdot \chi_{n_1}, \quad
  n_1\in \Lambda'_N,\,  N\in \N,\, n_1'\in \N,
  $$
    and then using the Cauchy-Schwarz inequality, we get
 $$
 \liminf_{N\to\infty}\E_{n_1\in\Lambda'_N}\,   \int |\E_{n_1'\in [N]}\, e(n_1'\beta_N)\cdot T^{n_1'+n_1}\tilde{f}_2 \cdot T^{n_1'}\overline{\tilde{f}_2}|^2\, d\mu>0.
 $$
Since  \eqref{E:ldens} holds, we  get
 $$
 \liminf_{N\to\infty}\E_{n_1\in [N]}\,   \int |\E_{n_1'\in [N]}\, e(n_1'\beta_N)\cdot T^{n_1'+n_1}\tilde{f}_2 \cdot T^{n_1'}\overline{\tilde{f}_2}|^2\, d\mu>0.
 $$
Finally, using     Lemma~\ref{L:VDC} and composing with $T^{-n_1'}$,    we deduce that
$$
\liminf_{N\to\infty}\Re\Big(\E_{n_1,n_1',n_2\in[N]}\,  
{\bf 1}_{[N]}(n_1'+n_2) \,  e(n_2\beta_N) \int \tilde{f}_2\cdot
T^{n_2}\overline{\tilde{f}_2} \cdot T^{n_1}\overline{\tilde{f}_2} \cdot T^{n_1+n_2}\tilde{f}_2\, d\mu\Big)>0.
$$
Hence, for some $n'_{1,N}\in [N]$, $N\in \N$, we have
$$
\liminf_{N\to\infty}\Re\Big(\E_{n_1, n_2\in[N]}\,  
{\bf 1}_{[N]}(n_{1,N}'+n_2) \,  e(n_2\beta_N) \int \tilde{f}_2\cdot
T^{n_2}\overline{\tilde{f}_2} \cdot T^{n_1}\overline{\tilde{f}_2} \cdot T^{n_1+n_2}\tilde{f}_2\, d\mu\Big)>0.
$$
 Using the case $s=2$ of Lemma~\ref{L:lower} below, for $c_N(n_1,n_2):={\bf 1}_{[N]}(n_{1,N}'+n_2)  \, e(n_2\beta_N)$, $n_1,n_2\in [N]$, $N\in\N$, which do not depend on the variable $n_1$, we deduce that
 $\nnorm{\tilde{f}_2}_2>0$,  completing the proof of the claim.
\end{proof}

\subsection{Step 5} (Proof of Theorem~\ref{T:FW})
We are now ready to  conclude the proof of  Theorem~\ref{T:FW}.
 It is known that if a function is orthogonal to the rational Kronecker factor of a system, then the same property holds with respect to almost every ergodic component (see for example \cite[Theorem~3.2]{FrK06}). Hence, using the ergodic decomposition theorem, 	we can assume that the system is ergodic.
	
We first show that if 	$\E(f_1|\mathcal{K}_{rat}(T))=0$, then  \eqref{E:aa0} holds.
Arguing by contradiction,  suppose that  \eqref{E:aa0} fails.
Then if $\tilde{f}_2$ is given by \eqref{E:aa2} we get by Propositions~\ref{P:1} and \ref{P:2} that  $\nnorm{\tilde{f}_2}_3>0$.  We deduce from Proposition~\ref{P:4}  that $\nnorm{\tilde{f}_2}_2>0$. It follows by  Proposition~\ref{P:1'}  below that there exists $\chi\in \mathcal{E}(T)$ such that
$$
\lim_{k\to\infty} \E_{n\in [N_k]} \, \int T^{-n^2}g_k\cdot T^{-n^2+n}\overline{f}_1\cdot \chi\, d\mu\neq 0.
$$
Composing with $T^{n^2}$, using that  $T^{n^2}\chi=e(n^2\alpha)\cdot \chi$, $n\in \N$, for some $\alpha \in [0,1)$, and using the Cauchy-Schwarz inequality,  we get that
$$
\E_{n\in [N_k]} \, e(n^2\alpha)\,   T^n\overline{f}_1\not\to 0
$$
in $L^2(\mu)$ as $k\to\infty$.
Since $\lim_{N\to\infty}\E_{n\in [N]} \, e(n^2\alpha+nt)=0$ if $t$ is irrational and $\alpha\in \R$,
using the spectral theorem,  the bounded convergence theorem, and our assumption $\E(f_1|\mathcal{K}_{rat}(T))=0$ (which implies that the   spectral measure $\sigma_{f_1}$ has no rational point mass),  we  deduce that the last limit is zero, a contradiction.

Finally,  we show that if 	$\E(f_2|\mathcal{K}_{rat}(T))=0$, then  \eqref{E:aa0} holds.
By the previous step we can assume that $f_1$ belongs to the rational Kronecker factor of the system and by approximation that it is an eigenfunction with eigenvalue $e(\alpha)$ for some $\alpha\in \Q$. Hence, it suffices to show that
$$
\lim_{N\to\infty}\E_{n\in [N]} \, e(n \alpha) \, T^{n^2}f_2 =0
$$
in $L^2(\mu)$.
Since $\lim_{N\to\infty}\E_{n\in [N]} \, e(n^2t+n\alpha)=0$ if $t$ is irrational and $\alpha\in \R$,  as before,
using the spectral theorem,  the bounded convergence theorem, and our assumption $\E(f_2|\mathcal{K}_{rat}(T))=0$,  we  deduce that the last limit is zero.
This completes the proof of Theorem~\ref{T:FW}.

\subsection{Weyl-type estimates} We record some pretty standard  Weyl-type estimates that were used in the previous argument (see for example  \cite[Proposition 4.3]{GT09}).
\begin{lemma}\label{L:Weyl}
	Let 	$a>0$ and $d\in \N$. There exists $Q=Q(a,d)>0$, $C=C(a,d)>0$, such that if
	$$
	|\E_{n\in [N]\, }e(t_1 n+\cdots +t_d n^d)|\geq a
	$$
	for some $N\in \N$ and $t_1,\ldots, t_d\in [0,1)$, then for every $i\in [d]$ there exist non-negative integers  $p,q\leq Q$ (depending on $i$)
	such that $\big|t_i-\frac{p}{q}\big|\leq \frac{C}{N^i}$.
\end{lemma}
From this we deduce the following (we  only use it for $\ell=2$ and $p_1(n)=n$, $p_2(n)=n^2$):
\begin{lemma}\label{L:Weyl2}
	Let $a>0$ and $d\in\N$.  Then there exists a finite set of rationals $R=R(a,d)$ with the following property:  If $p_1,\ldots, p_\ell\in \Z[t]$ are rationally independent polynomials of degree at most $d$ such that
	 for some $l\in [\ell]$ and $t_l\in[0,1)$ we have
	$$
	\limsup_{N\to\infty}\sup_{t_i\in \R, i\in [\ell]\setminus\{l\}}
	|\E_{n\in [N]\, }e(p_1(n)t_1+\cdots + p_\ell(n)t_\ell)|\geq a,
	$$
	then $t_{l}\in R$.
\end{lemma}
\begin{proof}
	Without loss of generality we can assume that $l=\ell=d$ and that all the polynomials have zero constant term.
	Let $p_j(n)=\sum_{i=1}^\ell c_{i,j}n^i$, $j=1,\ldots, \ell$, for some $c_{i,j}\in \Z$.
	
	The assumption gives that there exist $N_k\to\infty$ and $t_{i,k}, t_\ell\in [0,1)$, $i=1,\ldots, \ell-1$,  $k\in\N$,  such that
	$$
	|\E_{n\in [N_k]}\, e( p_1(n)t_{1,k}+\cdots +p_{\ell-1}(n)t_{\ell-1,k} +p_\ell(n)t_\ell)|\geq \frac{a}{2}
	$$
	for all $k\in \N$. Then  for all $k\in \N$ we have
	$$
	|\E_{n\in [N_k]}\, e(ns_{1,k}+\cdots +n^\ell s_{\ell,k})|\geq \frac{a}{2}
	$$
	where
	$$
	s_{i,k}:= \sum_{j=1}^{\ell-1} c_{i,j}t_{j,k}+ c_{i,\ell}t_\ell, \quad i=1,\ldots, \ell, \, k\in \N.
	$$

	The previous lemma implies that there exist positive integers  $Q,C$ that depend only on $a,\ell$ such that for every $k\in\N$ and $i\in \{1,\ldots, \ell\}$ there exist non-negative integers $p_{i,k},q_k\leq Q$ and $m_{i,k}\in \Z$, $i=1,\ldots, \ell$, $k\in \N$,  such that
	$$
	\Big|s_{i,k}-\frac{p_{i,k}}{q_{k}}-m_{i,k}\Big|\leq \frac{C}{N_k^i}, \quad i=1,\ldots, \ell.
	$$
	Since the polynomials $p_1,\ldots, p_\ell$ are rationally independent, the matrix $(c_{i,j})_{i,j\in [\ell]}$ is invertible.  We deduce
	that there exist $l,l_1,\ldots, l_\ell\in \N$, with size smaller than  $L=L(Q)\geq Q$,  such that
	$$
	t_\ell=\sum_{i=1}^\ell \frac{l_i}{l}s_{i,k}.
	$$
	Combining the previous two facts  we get that there exist $\tilde{p}_k\in \Z$ and non-negative integers
	$\tilde{q}_k\leq L^2$, $k\in\N$, such that
		$$
	\lim_{k\to\infty}\frac{\tilde{p}_k}{\tilde{q}_k}=t_\ell.
	$$
	Since $t_\ell\in [0,1)$ and $\tilde{q}_k$, $k\in\N$,  are positive integers bounded by $L^2$,
it follows that $t_{\ell}=\frac{p}{q}$ for some non-negative integers $p,q\leq L^2$. 	So we can take $R$ to be the set of all rationals with
	numerator and denominator at most  $L^2$.
\end{proof}	

 \subsection{More general results}
 Before embarking into the proof of our main results let us make some remarks about the extend to which the previous argument applies to more general families of sequences.

If in place of $n, n^2$ we have sequences $a_1,\ldots, a_\ell$ that are good for seminorm estimates and equidistribution,  a similar, but technically more complicated argument can be used  to prove joint ergodicity and is given in the next two sections.

If in place of $n, n^2$ we have two rationally independent integer polynomials $p_1, p_2$, then modulo  a few additional technical complications, the previous argument
can be adapted as in the next two sections in order  to prove that the rational Kronecker factor is characteristic for $p_1,p_2$  (for this extension we also need the Weyl estimates of Lemma~\ref{L:Weyl2}).

If in place of $n,n^2$ we have rationally independent polynomials $p_1,\ldots, p_\ell$, where $\ell\geq 3$, or if we work in the more general setup of Theorem~\ref{T:CharacteristicFactors}, then  a new  non-trivial obstacle arises.\footnote{It is caused by the fact that we can  no longer use the spectral theorem (as in the case $\ell=2$) in order to carry out Step 4 of Section~\ref{SS:2.5}. For $\ell=3$ the problem occurs
when  $f_1, f_2\in \mathcal{K}_{rat}(T)$ and we want to deduce that  an estimate of the form \eqref{E:problemforpolies} below implies an estimate of the form
\eqref{E:nn''}. In the case of sequences that are good for equidistribution, this problem does not arise because  we can replace the functions $f_1, f_2$ by constants.} In order to overcome this problem
we use a consequence of the main result in \cite{HK05a} that enables us to reduce matters to the case where the system is totally ergodic (the reduction is carried out in Section~\ref{SS:CF}),
a case that is covered by Theorem~\ref{T:JointErgodicity}. Alternatively, we could have used an additional inductive argument, as in \cite{P19b}, and avoid the use of deep results form \cite{HK05a}, thus leading to a more  ``elementary proof'' for the special case of  Theorem~\ref{T:CharacteristicFactors} that covers all rational independent polynomials. We chose not to do so, firstly, because
this would  lead to a much more complicated argument and, secondly, because this approach   requires to impose  stronger  equidistribution assumptions (of quantitative nature) than those used in Theorem~\ref{T:CharacteristicFactors}.

\section{Preparation for the main result}\label{S:Preparation}
 In this section we  gather some basic notation and results about the ergodic seminorms
 and also prove some elementary estimates that will be used later in the proof of Theorem~\ref{T:JointErgodicity}.

\subsection{The  ergodic seminorms}\label{SS:GHK}
Throughout, we use the following notation:
\begin{definition}
Let $(X,\mu,T)$ be a system and $f\in L^\infty(\mu)$. 	If $\underline{n}=(n_1,\ldots, n_s)\in \Z_+^s$, $\underline{n}'=(n_1',\ldots, n_k')\in \Z_+^k$,   $\epsilon=(\epsilon_1,\ldots, \epsilon_s)\in \{0,1\}^s$, and $z\in \C$,  we let
\begin{enumerate}
\item $\epsilon\cdot \underline{n}:=\epsilon_1n_1+\cdots+ \epsilon_sn_s$;

\item $|\underline{n}|:=n_1+\cdots+ n_s$;

\item $\mathcal{C}^lz=z$ if $l$ is even and $\mathcal{C}^lz=\overline{z}$ if $l$ is odd;

\item $\underline{n}^\epsilon:=(n_1^{\epsilon_1}, \ldots, n_s^{\epsilon_s})$, where $n_j^0:=n_j$ and $n_j^1:=n_j'$ for $j=1,\ldots, s$;

\item $\Delta_{\underline{n}}f:=\Delta_{n_1}\cdots \Delta_{n_s}f=\prod_{\epsilon\in \{0,1\}^s}\mathcal{C}^{|\epsilon|}T^{\epsilon\cdot \underline{n}}f$ (here we allow $\underline{n}\in\Z^s$).
\end{enumerate}
\end{definition}
Given a system $(X,\mu,T)$
we will use the seminorms $\nnorm{\cdot}_s$, $s\in \N$,  that were introduced in  \cite{HK05a} for ergodic systems and can be  defined similarly for general systems  (see for example \cite[Chapter~8, Proposition~16]{HK18}). They are often refereed to as Gowers-Host-Kra seminorms and  are inductively defined for  $f\in L^\infty(\mu)$ as follows (for convenience we also define $\nnorm{\cdot}_0$, which is   not a seminorm):
$$
\nnorm{f}_0:=\int f\, d\mu,
$$
and for $s\in \Z_+$ we let
\begin{equation}\label{E:seminorm1}
\nnorm{f}_{s+1}^{2^{s+1}}:=\lim_{N\to\infty}\E_{n\in [N]} \nnorm{\Delta_{n}f}_s^{2^{s}}.
\end{equation}
We write $\norm{f}_{s,T}$ when  it is not clear from the context with respect to which transformation the seminorm is computed. The limit  in \eqref{E:seminorm1} can be shown to exist by successive applications of the mean ergodic theorem and for $f\in L^\infty(\mu)$ and $s\in \Z_+$ we have $\nnorm{f}_s\leq \nnorm{f}_{s+1}$ (see \cite{HK05a} or  \cite[Chapter~8]{HK18}).
It follows immediately from the definition that
 $$
\nnorm{f}_{1}=\norm{\E(f|\CI(T))}_{L^2(\mu)}
$$
where $\CI(T):=\{f\in L^2(\mu)\colon Tf=f\}$ and
\begin{equation}\label{E:seminorm2}
\nnorm{f}_s^{2^s}=\lim_{N_1\to\infty}\cdots \lim_{N_s\to\infty}\E_{n_1\in [N_1]}\cdots \E_{n_s\in [N_s]} \int \Delta_{n_1,\ldots, n_s}f\, d\mu.
\end{equation}
It can be shown that we can take any $s'\leq s$  of the  iterative limits to be simultaneous limits (i.e. average over $[N]^{s'}$ and let $N\to\infty$)
 without changing the value of the limit. This was originally proved in \cite{HK05a}, for a much simpler proof see \cite{BL15}.
 For $s'=s$ this gives the identity
\begin{equation}\label{E:seminorm3}
\nnorm{f}_s^{2^s}=\lim_{N\to\infty}\E_{\underline{n}\in [N]^s} \int \Delta_{\underline{n}}f\, d\mu,
\end{equation}
 and for $s'=s-2$ it gives the identity
\begin{equation}\label{E:seminorm4}
\nnorm{f}_s^{2^{s}}=\lim_{N\to\infty}\E_{\underline{n}\in [N]^{s-2}} \nnorm{\Delta_{\underline{n}}f}_2^{4}.
\end{equation}

  Also,  we can show that if $T,S$ are two commuting measure preserving transformations on $(X,\CX,\mu)$, then we have the implication
	\begin{equation}\label{E:TS}
	\CI(T)\subset \CI(S)\implies \nnorm{f}_{s,T}\leq \nnorm{f}_{s,S}
	\end{equation}
	for every $s\in \N$ and $f\in L^\infty(\mu)$.
We give the proof  for  $s=2$, the argument is similar for general $s\in \N$.  Using \eqref{E:seminorm2}, the mean ergodic theorem, and the hypothesis
$\CI(T)\subset \CI(S)$, we get
$$
\nnorm{f}_{2,T}^4=
\lim_{N\to\infty} \E_{n\in[N]} \norm{\E(T^nf\cdot f|\CI(T))}_{L^2(\mu)}^2\leq
\lim_{N\to\infty} \E_{n\in[N]} \norm{\E(T^nf\cdot f|\CI(S))}_{L^2(\mu)}^2.
$$
Using the mean ergodic theorem for the system $(X,\mu,S)$ and the fact that  the transformations $T,S$ commute, we get  that the last limit is equal to
$$
\lim_{N\to\infty}\lim_{M\to\infty}\E_{n\in[M]} \E_{m\in[N]} \int f\cdot T^n\overline{f}\cdot
 S^m\overline{f}\cdot T^nS^m f\, d\mu.
$$
Exchanging the limits over $M$ and $N$ (which we can do by \cite[Corollary~3]{Ho09}) and iterating the previous  procedure one more time, we deduce that
$$
\nnorm{f}_{2,T}^4\leq \lim_{N\to\infty}\lim_{M\to\infty}\E_{n\in[M]} \E_{m\in[N]} \int f\cdot S^n\overline{f}\cdot
S^m\overline{f}\cdot S^{n+m} f\, d\mu=\nnorm{f}_{2,S}^4.
$$

\subsection{Soft inverse theorems}
Recall that if $(X,\mu,T)$ is  a system, with   $\mathcal{E}(T)$ we denote the set of its eigenfunctions with modulus one.
\begin{proposition}\label{P:1'}
	Let $(X,\mu,T)$ be an ergodic system and $f\in L^\infty(\mu)$ be a function with $\norm{f}_{L^\infty(\mu)}\leq 1$. Then
	$$
	\nnorm{f}_2^4\leq  \sup_{\chi\in \mathcal{E}(T)} \Re\Big( \int f\cdot \chi\, d\mu\Big).
	$$
\end{proposition}
\begin{proof}
	Let  $\mathcal{K}(T)$ be the Kronecker factor of the system, meaning, the closed subspace of $L^2(\mu)$ spanned by all eigenfunctions of the system.
   It is well known (and not hard to prove, see for example \cite[Chapter~8, Theorem~1]{HK18}) that
	$$
	\nnorm{f}_{2}=\nnorm{\tilde{f}}_2
	$$
	where $\tilde{f}:=\E(f|\mathcal{K}(T))$.
	Since the system is ergodic, the subspace $\mathcal{K}(T)$ has an orthonormal basis of eigenfunctions of modulus one, say $(\chi_j)_{j\in\N}$  (the  basis is countable because the system is Lebesgue). Then
	$\tilde{f}=\sum_{j=1}^\infty c_j\, \chi_j$ where
	$$
	c_j:=\int \tilde{f}\cdot \overline{\chi}_j\, d\mu=\int f\cdot \overline{\chi}_j\, d\mu,\quad  j\in\N.
	$$
	It follows that 
	$$
	\nnorm{\tilde{f}}_2^4=\sum_{j=1}^\infty|c_j|^4\leq \sup_{j\in\N}(|c_j|^2)
	\sum_{j=1}^\infty|c_j|^2=\sup_{j\in\N}(|c_j|^2)\norm{f}_{L^2(\mu)}^2\leq
	\sup_{j\in\N}\Big|\int f\cdot \overline{\chi}_j\, d\mu\Big|,
	$$
	where the first identity follows by orthonormality and direct computation, the second identity by the Parseval identity, and the last estimate since all functions involved are bounded by $1$ .
	The result now follows since the set $\mathcal{E}(T)$ is invariant under multiplication by unit modulus  constants.
\end{proof}

\begin{proposition}\label{P:2'}
	Let $(X,\mu,T)$ be an ergodic system and $f\in L^\infty(\mu)$  be such that 	$\nnorm{f}_{s+2}>0$ for some $s\in\Z_+$. Then there  exist  
	 $\chi_{\underline{n}}\in \mathcal{E}(T)$, $\underline{n}\in \N^s$,  such that
	$$
\liminf_{N\to\infty} \E_{\underline{n}\in[N]^s}\Re \Big( \int \Delta_{\underline{n}}f \cdot \chi_{\underline{n}}\, d\mu\Big)>0. 
	$$
\end{proposition}
\begin{proof}
	By \eqref{E:seminorm4} we have   that
	$$
\lim_{N\to\infty}\E_{\underline{n}\in [N]^s} \nnorm{\Delta_{\underline{n}}f}_2^4>0.
	$$
	Using Proposition~\ref{P:1'} we deduce that
	$$
\liminf_{N\to\infty}\E_{\underline{n}\in [N]^s}\sup_{\chi\in \mathcal{E}(T)} \Re \Big(\int \Delta_{\underline{n}}f\cdot \chi\, d\mu\Big)>0.
	$$
	From this the asserted estimate readily follows.
\end{proof}

We will use the following  variant of the so called Gowers-Cauchy-Schwarz inequality:
\begin{lemma}\label{L:GCS} Let $(X,\mu,T)$ be a system,  and for $s\in \N$ let  $f_\epsilon \in L^\infty(\mu)$ be bounded by $1$ for $\epsilon\in \{0,1\}^s$, and $g_{\underline{n}}\in L^\infty(\mu)$ for $\underline{n}\in \N^s$. Let also $\underline{1}:=(1,\ldots,1)$. Then for every $N\in \N$ we have
$$
\Big|\E_{\underline{n}\in [N]^s}\,  \int \prod_{\epsilon\in \{0,1\}^s} T^{\epsilon\cdot \underline{n}}f_\epsilon\cdot g_{\underline{n}}\, d\mu\Big|^{2^s}\leq
\E_{\underline{n},\underline{n}'\in [N]^s}\,  \int \Delta_{\underline{n}-\underline{n}'}f_{\underline{1}}
\cdot T^{-|\underline{n}|}\big(\prod_{\epsilon\in \{0,1\}^s}\mathcal{C}^{|\epsilon|}g_{\underline{n}^\epsilon}
\big)\, d\mu.
$$
\end{lemma}
\begin{proof}
For notational simplicity we give the details only for $s=2$.
	The general case follows in a similar manner by  successively applying the Cauchy-Schwarz inequality with respect to the variables $n_s,\ldots, n_1$, exactly as we do below for $s=2$.
We have that
$$
\Big|\E_{n_1,n_2\in [N]}\int f_0 \cdot T^{n_1}f_1 \cdot T^{n_2}f_2\cdot  T^{n_1+n_2}f_3 \cdot  g_{n_1,n_2}\, d\mu\Big|^2
$$
is bounded by (we use that $f_0,f_1$ are bounded by $1$)
$$
\E_{n_1\in [N]}\int \Big|\E_{n_2\in [N]} T^{n_2}f_2\cdot  T^{n_1+n_2}f_3 \cdot  g_{n_1,n_2}\Big|^2\, d\mu.
$$
After expanding the square we find that this expression is equal to
 $$
 \E_{n_1\in [N]}\int \E_{n_2,n_2'\in [N]}\,  T^{n_2}f_2\cdot  T^{n_2'}\overline{f}_2\cdot  T^{n_1+n_2}f_3\cdot  T^{n_1+n_2'}\overline{f}_3 \cdot  g_{n_1,n_2}\cdot \overline{g_{n_1,n_2'}} \, d\mu.
 $$
 After composing with  $T^{-n_2}$, exchanging  $\E_{n_1\in [N]}$ with
 $\E_{n_2,n_2'\in [N]}$,
  using the Cauchy-Schwarz inequality, and that $f_2$ is bounded by $1$, we get that the square of the last expression is bounded by
 $$
 \E_{n_2,n_2'\in [N]}\int \Big|\E_{n_1\in [N]}  \,  T^{n_1}f_3\cdot  T^{n_1+n_2'-n_2}\overline{f}_3 \cdot
 T^{-n_2}( g_{n_1,n_2}\cdot \overline{g}_{n_1,n_2'}) \Big|^2 \, d\mu.
 $$
 As before, we  expand the square,   and compose with $T^{-n_1}$.  We arrive at the expression
 \begin{multline*}
 \E_{n_1,n_2,n_1',n_2'\in [N]} \\
 \int f_3\cdot T^{n_1'-n_1}\overline{f_3}\cdot
 T^{n_2'-n_2}\overline{f_3} \cdot T^{n_1'+n_2'-n_1-n_2}f_3\cdot
 T^{-n_1-n_2}( g_{n_1,n_2}\cdot \overline{g}_{n_1,n_2'}
 \cdot \overline{g}_{n_1',n_2}\cdot g_{n_1',n_2'})  \, d\mu,
 \end{multline*}
which is equal to the right hand side of the asserted estimate when  $s=2$
(for $\underline{n}:=(n_1,n_2)$,  $\underline{n}':=(n_1',n_2')$).
Combining the previous two estimates gives the asserted bound for $s=2$.
\end{proof}

\begin{lemma}\label{L:lower}
Let  $(X,\mu, T)$ be an ergodic  system and  $f\in L^\infty(\mu)$ be such that $\nnorm{f}_{s}=0$ for some $s\in \N$.
For $j=1,\ldots,s$, $N\in\N$,  let  $b_{j,N}\in \ell^\infty(\N^{s})$  be  sequences  that do not depend on the variable $n_j$
and are bounded by $1$, and let
$$
c_{N}(\underline{n}):=\prod_{j=1}^s b_{j,N}(\underline{n}), \quad  \underline{n}\in [N]^{s}, \, N\in\N.
$$
Then
$$
\lim_{N\to\infty}  	\norm{\E_{\underline{n}\in [N]^{s}}\, c_{N}(\underline{n})\cdot
\Delta_{\underline{n}}f}_{L^2(\mu)}=0.
$$
\end{lemma}
\begin{proof}
Using  the Cauchy-Schwarz inequality, the fact that the sequence $b_{s,N}$ does not depend on the variable $n_{s}$ and is bounded by $1$, and the identity
$$	
\Delta_{n_1,\ldots,n_s}f=T^{n_s} (\Delta_{n_1,\ldots, n_{s-1}}f)\cdot \Delta_{n_1,\ldots, n_{s-1}}\overline{f},
$$
we get that it suffices to show that (for later convenience we rename the variable $n_s$ as $n_s'$)
$$
\lim_{N\to\infty}  	\E_{\underline{n}\in [N]^{s-1}}\norm{\E_{n_s'\in [N]}  \, \prod_{j=1}^{s-1} b_{j,N}(\underline{n}, n_s')\cdot
	T^{n_s'}(\Delta_{\underline{n}}f)}^2_{L^2(\mu)}=0.
$$

 Using Lemma~\ref{L:VDC} for the average over $n_s'$ and composing with $T^{{-n_s'}}$,   we get that it suffices to show that
$$
\lim_{N\to\infty}  	\E_{\underline{n}\in [N]^{s-1}}\E_{n_s, n_s'\in [N]} \, {\bf 1}_{[N]}(n_s+n_s')\, \prod_{j=1}^{s-1}
	 b_{j,N}(\underline{n}, n'_s+n_s)\cdot  \overline{b}_{j,N}(\underline{n}, n'_s)
\int 	T^{n_s}(\Delta_{\underline{n}}f)\cdot 	\Delta_{\underline{n}}\overline{f}\, d\mu =0,
$$
or, equivalently, that
 $$
\lim_{N\to\infty}\E_{n_s'\in [N]} 	\E_{(\underline{n},n_s)\in [N]^s} \, {\bf 1}_{[N]}(n_s+n_s') \, \prod_{j=1}^{s-1}
   b_{j,N}(\underline{n}, n'_s+n_s) \cdot \overline{b}_{j,N}(\underline{n}, n'_s)
\int 	\Delta_{\underline{n},n_s}f\, d\mu =0.
$$
This would follow if we show that
 $$
\lim_{N\to\infty} \sup_{n_s'\in [N]}\Big|	\E_{(\underline{n},n_s)\in [N]^s} \, {\bf 1}_{[N]}(n_s+n_s') \, \prod_{j=1}^{s-1}
b_{j,N}(\underline{n}, n'_s+n_s) \cdot \overline{b}_{j,N}(\underline{n}, n'_s)
\int 	\Delta_{\underline{n},n_s}f\, d\mu\Big| =0,
$$
or equivalently, that
 for any choice of $n'_{s,N}\in [N]$, $N\in \N$, we have
  $$
 \lim_{N\to\infty}   \E_{(\underline{n},n_s)\in [N]^s}\, {\bf 1}_{[N]}(n_s+n_{s,N}') \, \prod_{j=1}^{s-1}
 b_{j,N}(\underline{n}, n'_{s,N}+n_s) \cdot  \overline{b}_{j,N}(\underline{n}, n'_{s,N})
 \int 	\Delta_{\underline{n},n_s}f\, d\mu =0.
 $$
Using the Cauchy-Schwarz inequality we get that it suffices to show that
   $$
 \lim_{N\to\infty}  \norm{\E_{(\underline{n},n_s)\in [N]^s} \, \prod_{j=1}^{s-1}
 b'_{j,N}(\underline{n},n_s)\cdot
 	\Delta_{\underline{n},n_s}f}_{L^2(\mu)}=0,
 $$
where  we let
$$
b'_{s-1,N}(\underline{n}, n_s):= {\bf 1}_{[N]}(n_s+n_{s,N}')  \cdot b_{s-1,N}(\underline{n}, n'_{s,N}+n_s)\cdot \overline{b}_{s-1,N}(\underline{n}, n'_{s,N})   ,  \quad (\underline{n},n_s)\in [N]^s, \, N\in \N,
$$
and for $j=1,\ldots,s-2$, we let
$$
b'_{j,N}(\underline{n},n_s):= b_{j,N}(\underline{n}, n'_{s,N}+n_s)\cdot \overline{b}_{j,N}(\underline{n}, n'_{s,N}), \quad (\underline{n},n_s)\in [N]^s,\,  N\in \N.
$$
Note that we now have a product of $s-1$ sequences, instead of $s$, and  for $j=1,\ldots, s-1$ the sequences $b'_{j,N}$ do not depend on the variable $n_j$. Hence,  we can continue like that (on the next step we eliminate the sequences $b'_{s-1,N}$, etc), and  after $s$ steps we deduce that  it suffices to show that
  $$
\lim_{N\to\infty}   	\E_{\underline{n}\in [N]^s} \,
 \big(1-\frac{n_1}{N}\big)	\ \int \Delta_{\underline{n}}f\, d\mu =0.
$$
Using partial summation with respect to the variable $n_1$  it suffices to show the following
  $$
\lim_{N\to\infty}   	\E_{\underline{n}\in [N]^s} \,
\int 	\Delta_{\underline{n}}f\, d\mu =0.
$$
This follows  from our assumption  $\nnorm{f}_s=0$ and \eqref{E:seminorm3} and  completes the proof.

	\end{proof}

\section{Joint ergodicity of general  sequences}\label{S:Main}
 The primary goal of this
section is to prove  Theorem~\ref{T:JointErgodicity}. At the end of the section we also deduce   Corollaries~\ref{C:MultipleRecurrence} and \ref{C:Nilsystems} from Theorem~\ref{T:JointErgodicity}. For a better understanding of the argument,
we advice the reader to  first go through  the technically less demanding model case that was treated  in Section~\ref{S:FW}.

	It is clear that Property~$(i)$ of Theorem~\ref{T:JointErgodicity} implies Property~$(ii)$. Indeed we can use appropriate  eigenfunctions to prove the  equidistribution property and the fact that $\nnorm{f}_1=0$ implies $\int f\, d\mu=0$ to prove the seminorm property with $s=1$.
So we only prove that Property~$(ii)$ implies Property~$(i)$.

\subsection{Goal}
It will be more convenient to prove the
 following statement (the case  $m=\ell$ implies  Theorem~\ref{T:JointErgodicity}):
 \begin{proposition}\label{P:10}
 	 Let  $(X,\mu,T)$ be an ergodic system and   $a_1,\ldots, a_\ell\colon \N\to \Z$
 	 be sequences that are good for seminorm estimates and  equidistribution for $(X,\mu,T)$. Then  for every $m\in \{0,1,\ldots, \ell\}$
 	 the following holds:

  ($P_m$) If  $f_1,\ldots,f_\ell\in L^\infty(\mu)$ are such that   
 $f_j\in \mathcal{E}(T)$ for  $j=m+1,\ldots, \ell$, then
\begin{equation}\label{E:desired}
\lim_{N\to\infty}\E_{n\in [N]}\, T^{a_1(n)}f_1 \cdot\ldots \cdot T^{a_\ell(n)}f_\ell=\int f_1\, d\mu\cdot \ldots \cdot   \int f_\ell\, d\mu
\end{equation}
where convergence takes place in $L^2(\mu)$.
\end{proposition}

In the following steps we fix $\ell\in\N$, a system, and  a collection of sequences, and  we are going to prove
that  property $(P_m)$ of  Proposition~\ref{P:10} holds
by using (finite) induction on $m\in \{0,1,\ldots, \ell\}$.

\subsection{Proof strategy}
Let us briefly describe the main idea of the proof, which is again motivated by the technique used in \cite{PP19}. If
\eqref{E:desired} fails, then  in Step~1 below, using our assumption that the sequences $a_1,\ldots, a_\ell$ are good for seminorm estimates for $(X,\mu,T)$, we deduce that $\nnorm{\tilde{f}_m}_{s+2}>0$ for some $s\in \Z_+$ where $\tilde{f}_m$ is given by \eqref{E:wlp}. This then implies that
\begin{equation}\label{E:Delta'}
\liminf_{N\to\infty} \E_{\underline{n}\in[N]^s}\,
\Re \Big(\int \Delta_{\underline{n}}\tilde{f}_m\cdot \chi_{\underline{n}}\, d\mu\Big)>0,
\end{equation}
where  $\chi_{\underline{n}}$, $\underline{n}\in\N^s$,  are  appropriate unit modulus eigenfunctions of the system $(X,\mu,T)$.
In Steps~2 and 3, using suitable manipulations of \eqref{E:Delta'}, the good  equidistribution properties of  the sequences $a_1,\ldots, a_\ell$, and the induction hypothesis,  we deduce  that  the eigenfunctions $\chi_{\underline{n}}$ satisfy some algebraic relations, which allow us to conclude that they are products of sequences in $s-1$ variables.\footnote{We crucially use in this step that the sequences $a_1,\ldots, a_\ell$ are good for equidistribution;  if $\ell\geq 3$ and we only knew that they were good for irrational equidistribution (for example if they were polynomial sequences), and our aim was to prove that the rational Kronecker factor is a characteristic factor, then our argument would have been much more complicated (as is the case in \cite{P19b}).} In view of this,  \eqref{E:Delta'} combined with Lemma~\ref{L:lower} enables us to show that
$\nnorm{\tilde{f}_m}_{s+1}>0$. Iterating this step $s+1$ times we deduce that $\int \tilde{f}_m\, d\mu>0$. Given this, it will be an easy matter in Step~4 to  contradict the assumptions of Proposition~\ref{P:10}.

Let us say a few  words about the aforementioned  ``suitable manipulations'' performed in Steps 2 and 3, since they  are  more complicated than those used  for the case $s=1$  in Section~\ref{S:FW}.
Suppose for simplicity that $\ell=m=s=2$;
hence  $\nnorm{\tilde{f}_2}_{4}>0$,
and our goal is to show that $\nnorm{\tilde{f}_2}_{3}>0$.
In this case, our assumption implies that
\begin{equation}\label{E:Delta''}
\liminf_{N\to\infty} \E_{\underline{n}\in[N]^2}\,
\Re \Big(\int \Delta_{\underline{n}}\tilde{f}_2\cdot \chi_{\underline{n}}\, d\mu\Big)>0,
\end{equation}
for some eigenfunctions $\chi_{\underline{n}}$ with modulus $1$. Using the Cauchy-Schwarz inequality twice (as in Lemma~\ref{L:GCS}) we deduce that 	(the symmetries of the ergodic seminorms are crucially used here)
$$
\liminf_{N\to\infty}\E_{\underline{n},\underline{n}'\in [N]^2}\,
\limsup_{k\to\infty}\E_{n\in[N_k]}\,  \Re \Big(\int \Delta_{\underline{n}-\underline{n}'}g_k\cdot
T^{a_1(n)}(\Delta_{\underline{n}-\underline{n}'}
\overline{f}_1) \cdot T^{a_2(n)}\chi_{\underline{n},\underline{n}'} \, d\mu\Big)>0,
$$
where for $\underline{n}=(n_1,n_2)\in [N]^2$ and  $\underline{n}'=(n_1',n_2')\in [N]^2$,  the eigenfunctions $\chi_{\underline{n},\underline{n}'}$ are given by
\begin{equation}\label{E:cnn}
\chi_{\underline{n},\underline{n}'}=\chi_{n_1,n_2}\cdot
\overline{\chi}_{n_1,n_2'}\cdot \overline{\chi}_{n_1',n_2}\cdot \chi_{n_1',n_2'}.
\end{equation}
Note that the last estimate  is  substantially simpler to analyze than the $\ell=2$ case of \eqref{E:desired}, since in the inside average over the variable $n$, the iterate $T^{a_2(n)}$ is applied to an eigenfunction of the system and not to an arbitrary function in $L^2(\mu)$.
Using Property $(P_1)$ of Proposition~\ref{P:10} (which is our induction hypothesis), we can assume that the function $\Delta_{\underline{n}-\underline{n}'}
\overline{f}_1$ is constant for all $\underline{n},\underline{n}'\in \N^2$.   Using the good equidistribution properties of the sequence $a_2$, we deduce that for $N\in \N$  there exist constants $\underline{n}_N'\in [N]^2$ and sets $\Lambda_N\subset [N]^2$ with $\liminf_{N\to\infty}|\Lambda_N|/N>0$, such that    $\chi_{\underline{n},\underline{n}_N'}$ is constant for $\underline{n}\in \Lambda_N$.
Using \eqref{E:cnn}, this gives that for those values of $\underline{n}\in [N]^2$ the  $2$-variable sequence $\chi_{n_1,n_2}$ is a product of two $1$-variable sequences. If we use this information and an appropriate variant of \eqref{E:Delta''}, where we average over $\Lambda_N$ instead of $[N]^2$ (by being more careful we can guarantee that this holds), we deduce by  Lemma~\ref{L:lower}
that  $\nnorm{\tilde{f}_2}_{3}>0$.

\subsection{Step 1} (Characteristic factors)
The first step is to show that if \eqref{E:desired} fails, then it also fails for some function $f_m$ of special form that is given by \eqref{E:wlp} below.
 \begin{proposition}\label{P:11}
	Let $a_1,\ldots, a_\ell\colon \N\to \Z$  be sequences, $(X,\mu,T)$ be a system, and $f_1,\ldots, f_\ell\in L^\infty(\mu)$ be such that
	\begin{equation}\label{E:aa11}
\limsup_{N\to\infty}\norm{ 	\E_{n\in [N]} \, T^{a_1(n)}f_1\cdots T^{a_\ell(n)}f_\ell}_{L^2(\mu)}>0.
	\end{equation}
	Let $m\in [\ell]$.  Then  there exist $N_k\to\infty$ and $g_k\in L^\infty(\mu)$, with  $\norm{g_k}_{L^\infty(\mu)}\leq 1$,  $k\in\N$, such that
	for
	\begin{equation}\label{E:wlp}
	\tilde{f}_m:=\lim_{k\to\infty} \E_{n\in [N_k]}
	 \, T^{-a_m(n)}g_k\cdot \prod_{j\in [\ell], j\neq m}T^{a_j(n)-a_m(n)}\overline{f}_j,
	\end{equation}
	where the limit is a weak limit (note that then $\tilde{f}_m\in L^\infty(\mu)$), we have
	\begin{equation}\label{E:aa13}
\limsup_{N\to\infty}\norm{ 	\E_{n\in[N]} \, T^{a_m(n)}\tilde{f}_m\cdot  \prod_{j\in [\ell], j\neq m}T^{a_j(n)}f_j}_{L^2(\mu)}>0.
	\end{equation}
	 Furthermore, if $\ell=m=1$, and $\int f_1\, d\mu=0$, then we can choose $\tilde{f}_1$ of the form  \eqref{E:wlp}  and such that  $\int \tilde{f}_1\, d\mu=0$.
\end{proposition}
 \begin{proof}
 	The first part of the statement can be proved
 	by adjusting the proof of  Proposition~\ref{P:1} in a straightforward way, so we omit its proof.
 	
 	We give a  proof of  the second part related to the case $\ell=1$ since it requires some minor adjustments.  We can assume that $\norm{f_1}_{L^\infty(\mu)}\leq 1$.  Suppose that for some sequence $N_k\to\infty$ we have
 	$$
 	\lim_{k\to\infty}\norm{	\E_{n\in [N_k]} \, T^{a_1(n)}f_1}^2_{L^2(\mu)}>0.
 	$$
 	After expanding the square and composing with $T^{-a_1(n)}$ we get
 	\begin{equation}\label{E:f1fk}
 	\lim_{k\to\infty} \int f_1\cdot \E_{n\in [N_k]}\,  T^{-a_1(n)}\overline{g}_k\, d\mu>0
 	\end{equation}
 	where
 	\begin{equation}\label{E:fk1}
 	g_k:=	\E_{n\in [N_k]} \, T^{a_1(n)}f_1, \quad k\in\N.
 	\end{equation}
 	After passing to a subsequence  we can assume that  the averages $\E_{n\in [N_k]}\, T^{-a_1(n)} \overline{g}_k $ converge weakly in $L^2(\mu)$  as  $k\to\infty$ and we let
 $$
 	\tilde{f}_1:=\lim_{k\to\infty}\E_{n\in [N_k]}\, T^{-a_1(n)} \overline{g}_k.
 	$$
Note that if $\int f_1\, d\mu=0$, then \eqref{E:fk1} gives that $\int g_k\, d\mu=0$ for every $k\in \N$
 and as a consequence
 $\int \tilde{f}_1\, d\mu=0$.
 	Moreover, using \eqref{E:f1fk} we conclude that $\int f_1\cdot \tilde{f}_1\, d\mu>0$, and the Cauchy-Schwarz inequality gives
 	$$
 	0<\int |\tilde{f}_1|^2\, d\mu =\lim_{k\to\infty}\E_{n\in [N_k]}\int \tilde{f}_1 \cdot T^{-a_1(n)} g_k\, d\mu=\lim_{k\to\infty}\int  \E_{n\in [N_k]}\, T^{a_1(n)} \tilde{f}_1 \cdot g_k\, d\mu.
 	$$
 	Using the Cauchy-Schwarz inequality again, we deduce that
 	$$
 	\limsup_{k\to\infty}\norm{\E_{n\in [N_k]}\,  T^{a_1(n)} \tilde{f}_1 }_{L^2(\mu)}>0
 	$$
 	as required.
 \end{proof}

\subsection{Step 2} (Seminorms of averages of functions)
Our next goal is to use Lemma~\ref{L:GCS} in order to show  that if the seminorm  of an average of functions is positive, then some related positiveness property holds for the individual functions
(the $s=1$ case was proved in Proposition~\ref{P:3}).


\begin{proposition}\label{P:3'}
	Let $(X,\mu,T)$ be an ergodic system, $f_{n,k}\in L^\infty(\mu)$, $k,n\in \N$, be  bounded by $1$, and $f\in L^\infty(\mu)$ be defined by
	$$
	f:=\, \lim_{k\to\infty}\E_{n\in [N_k]}\, f_{n,k},
	$$
	for some $N_k\to\infty$,
	where the average is assumed to converge weakly.
If	$\nnorm{f}_{s+2}>0$ for some $s\in \Z_+$, 	then there exist
	$a>0$, a subset $\Lambda$ of $\N^s$ with positive lower density, and
	 $\chi_{\underline{n}}\in \mathcal{E}(T)$, $\underline{n}\in \Lambda$, such that
		\begin{equation}\label{E:deltan}
  \Re \Big( \int \Delta_{\underline{n}}f \cdot \chi_{\underline{n}}\, d\mu\Big)>a,
 \quad  \underline{n} \in \Lambda,
	\end{equation}
	and 	
\begin{equation}\label{E:deltan'}
	\liminf_{N\to\infty}\E_{\underline{n},\underline{n}'\in [N]^s}\,
	   \limsup_{k\to\infty}\E_{n\in[N_k]}\,  \Re \Big(\int \Delta_{\underline{n}-\underline{n}'}f_{n,k} \cdot\chi_{\underline{n},\underline{n}'}\cdot  {\bf 1}_{\Lambda'}(\underline{n},\underline{n}') \, d\mu\Big)>0,
	\end{equation}
	where
$$
\chi_{\underline{n},\underline{n}'}:=	T^{-|\underline{n}|}\big(\prod_{\epsilon\in \{0,1\}^s} \mathcal{C}^{|\epsilon|} \chi_{\underline{n}^\epsilon}\big), \quad \underline{n},\underline{n}'\in \N^s,
	$$
	and (recall that $n^\epsilon$ is defined in Section~\ref{SS:GHK})
	$$
	\Lambda':=\{(\underline{n},\underline{n}')\in \N^{2s}\colon \underline{n}^\epsilon \in \Lambda \text{ for all } \epsilon\in \{0,1\}^s\}.
		$$
	\end{proposition}
\begin{remarks}
$\bullet$	The key point is that from a lower bound for the expression \eqref{E:deltan}   that involves the terms
	$\Delta_{\underline{n}}f$, we can infer a lower bound for the expression \eqref{E:deltan'} that involves the easier to handle terms $\Delta_{\underline{n}}f_{n,k}$.
	
	$\bullet$ For $s=0$ the conclusion is that there exists $\chi\in \mathcal{E}(T)$ such that $\Re \big(\int f\cdot \chi\, d\mu\big)>0$.

$\bullet$
	We plan to apply this proposition in the next step for the  function $\tilde{f}_m$ given by \eqref{E:wlp} in place of $f$.
	\end{remarks}
\begin{proof}
	By Proposition~\ref{P:2'} we have that there exist $a>0$, a subset $\Lambda$ of $\N^s$ with positive lower density, and  $\chi_{\underline{n}}\in \mathcal{E}(T)$, $\underline{n}\in \N^s$, such that
	$$
 \Re \Big( \int \Delta_{\underline{n}}f \cdot \chi_{\underline{n}}\, d\mu\Big)>a,
 \quad
  \underline{n} \in \Lambda.
$$
Hence,
$$
\liminf_{N\to\infty} \E_{\underline{n}\in  [N]^s}\,  \Re \Big( \int \Delta_{\underline{n}}f \cdot \chi_{\underline{n}}\cdot  {\bf 1}_\Lambda(\underline{n})\, d\mu\Big)>0.
$$
	Since $\Delta_{\underline{n}}f=\prod_{\epsilon\in \{0,1\}^s}\mathcal{C}^{|\epsilon|}T^{\epsilon\cdot \underline{n}}f$, $\underline{n}\in \N^s$,  and
	 $f=\lim_{k\to\infty}\E_{n\in [N_k]}\, f_{n,k}$ (the limit is a weak limit), we deduce that
	$$
\liminf_{N\to\infty} \lim_{k\to\infty}\E_{n\in [N_k]}\, \Re \Big(
\E_{\underline{n}\in  [N]^s}\,  \int \prod_{\epsilon\in \{0,1\}^s\setminus \{\underline{1}\}}\mathcal{C}^{|\epsilon|}T^{\epsilon\cdot \underline{n}}f \cdot T^{n_1+\cdots+n_s}f_{n,k}\cdot \chi_{\underline{n}}\cdot  {\bf 1}_\Lambda(\underline{n})\, d\mu\Big)>0.
	$$
For fixed $k,n\in \N$ we apply  Lemma~\ref{L:GCS} for $f_{\underline{1}}:=f_{n,k}$,   $f_\epsilon:=\mathcal{C}^{|\epsilon|}f $ for $\epsilon\in \{0,1\}^s\setminus \underline{1}$, and
$g_{\underline{n}}:=\chi_{\underline{n}}\cdot {\bf 1}_\Lambda(\underline{n})$, $\underline{n}\in \N^s$, and  deduce that
	$$
	\liminf_{N\to\infty} \limsup_{k\to\infty}\E_{n\in [N_k]}\,
\E_{\underline{n},\underline{n}'\in  [N]^s}\,  \int\Delta_{\underline{n}-\underline{n}'}f_{n,k}\cdot T^{-|\underline{n}|}\Big(\prod_{\epsilon\in \{0,1\}^s}\big(\mathcal{C}^{|\epsilon|}\chi_{\underline{n}^\epsilon}\cdot  {\bf 1}_\Lambda(\underline{n}^\epsilon)\big)\Big)\, d\mu>0.
$$	
Since the limsup of a sum is at most the sum  of the limsups,  the second asserted estimate follows immediately from this one.
This completes the proof.
\end{proof}


\subsection{Step 3} (Going from $\nnorm{\cdot}_{s+2}$ to  $\nnorm{\cdot}_{1}$)
Our next goal is to use Proposition~\ref{P:3'} and the fact that the sequences $a_1,\ldots, a_\ell$ are good for equidistribution  for the system $(X,\mu,T)$   in order to prove the following result:
 \begin{proposition}\label{P:4'}
	Let $(X,\mu,T)$ be an ergodic system and
	 $a_1,\ldots, a_\ell\colon \N\to \Z$ be good for equidistribution  for the system $(X,\mu,T)$. Suppose that property  $(P_{m-1})$ of  Proposition~\ref{P:10} holds for some $m\in [\ell]$ and let  $\tilde{f}_m$ be as in \eqref{E:wlp}, where all related functions are bounded by $1$ and  $f_{m+1},\ldots, f_\ell\in \mathcal{E}(T)$.
	Suppose that $\nnorm{\tilde{f}_m}_{s+2}>0$ for some $s\in \Z_+$. Then $\int \tilde{f}_m\, d\mu\neq 0$.
\end{proposition}
\begin{proof}
We can assume that the functions $f_1,\ldots, f_\ell$ are bounded by $1$. We will  show that   if   $\nnorm{\tilde{f}_m}_{s+2}>0$  for some $s\in \Z_+$, then    $\nnorm{\tilde{f}_m}_{s+1}>0$. Applying this successively $s+1$ times, we deduce that 
$\nnorm{\tilde{f}_m}_{1}>0$, or equivalently, that $\int \tilde{f}_m\, d\mu\neq 0$.

Since $\nnorm{\tilde{f}_m}_{s+2}>0$, we can use Proposition~\ref{P:3'} for $\tilde{f}_m$ in place of $f$ and
$$
f_{n,k}:=T^{-a_m(n)}g_k\cdot \prod_{j\in [\ell], j\neq m}T^{a_j(n)-a_m(n)}\overline{f}_j,\quad  k,n\in\N.
$$
 We deduce that there exist
      $a>0$, a  subset $\Lambda$ of $\N^s$ with positive lower density,
	 $\chi_{\underline{n}}\in \mathcal{E}(T)$,  $\underline{n}\in \N^s$, such that
\begin{equation}\label{E:U2first'}
\Re \Big( \int \Delta_{\underline{n}}\tilde{f}_m \cdot \chi_{\underline{n}}\, d\mu\Big)>a, \quad    \underline{n}\in \Lambda,
	\end{equation}
 and 	
\begin{multline}\label{E:U2second'}
\liminf_{N\to\infty}\E_{\underline{n},\underline{n}'\in [N]^s}\,  \limsup_{k\to\infty}
\\
  \E_{n\in[N_k]}\,  \Re \Big(\int T^{-a_m(n)}(\Delta_{\underline{n}-\underline{n}'}g_k) \cdot
\prod_{j\in [\ell], j\neq l}T^{a_j(n)-a_m(n)}(\Delta_{\underline{n}-\underline{n}'}\overline{f}_j)\cdot  \chi_{\underline{n},\underline{n}'}  \cdot  {\bf 1}_{\Lambda'}(\underline{n},\underline{n}')\, d\mu\Big)>0,
\end{multline}
where
$$
\chi_{\underline{n},\underline{n}'}:=	T^{-|\underline{n}|}\big(\prod_{\epsilon\in \{0,1\}^s} \mathcal{C}^{|\epsilon|} \chi_{\underline{n}^\epsilon}\big), \quad \underline{n},\underline{n}'\in \N^d,
	$$
	and
	$$
	\Lambda':=\{(\underline{n},\underline{n}')\in \N^{2d}\colon \underline{n}^\epsilon \in \Lambda \text{ for all } \epsilon\in \{0,1\}^s\}.
		$$
Our goal is to use \eqref{E:U2second'}  in order to show that the eigenfunctions
$\chi_{\underline{n}}$ in \eqref{E:U2first'} satisfy certain algebraic relations that enable us to
deduce, using Lemma~\ref{L:lower}, that $\nnorm{\tilde{f}_m}_{s+1}>0$.

We thus start by analyzing \eqref{E:U2second'}. After composing with $T^{a_m(n)}$
we get
\begin{multline*}
\liminf_{N\to\infty}\E_{\underline{n},\underline{n}'\in [N]^s}\,  \limsup_{k\to\infty}\\
  \E_{n\in[N_k]}\, 
 \Re \Big(\int \Delta_{\underline{n}-\underline{n}'}g_k \cdot
\prod_{j\in [\ell], j\neq m}T^{a_j(n)}(\Delta_{\underline{n}-\underline{n}'}\overline{f}_j)\cdot T^{a_m(n)} \chi_{\underline{n},\underline{n}'} \cdot  {\bf 1}_{\Lambda'}(\underline{n},\underline{n}') \, d\mu\Big)>0.
\end{multline*}
We let
 $$
g_{j, \underline{n},\underline{n}'}:=\Delta_{\underline{n}-\underline{n}'}\overline{f}_j, \quad  j\in [\ell], \,  j\neq m,\,  \underline{n}, \underline{n}'\in \N^s,
$$
and
$$
g_{l, \underline{n},\underline{n}'}:=\chi_{\underline{n},\underline{n}'}, \quad \underline{n}, \underline{n}'\in \N^s .
$$
Using the Cauchy-Schwarz inequality we deduce that
 $$
\liminf_{N\to\infty}\E_{\underline{n},\underline{n}'\in [N]^s}\,
 {\bf 1}_{\Lambda'}(\underline{n},\underline{n}') \, \limsup_{k\to\infty}
 \norm{\E_{n\in[N_k]}
\prod_{j=1}^\ell T^{a_j(n)}g_{j, \underline{n},\underline{n}'}}_{L^2(\mu)}>0.
$$
Since
$$
 {\bf 1}_{\Lambda'}(\underline{n},\underline{n}')=\prod_{\epsilon\in \{0,1\}^s}
 {\bf 1}_{\Lambda}(\underline{n}^\epsilon)\leq  {\bf 1}_{\Lambda}(\underline{n})
$$
  and the set $\Lambda$ has positive lower density, we get
\begin{equation}\label{E:problemforpolies}
\liminf_{N\to\infty} \E_{\underline{n}'  \in  [N]^s}\E_{\underline{n}\in   \Lambda\cap [N]^s}\,  \limsup_{k\to\infty}
 \norm{\E_{n\in[N_k]}
\prod_{j=1}^\ell T^{a_j(n)}g_{j, \underline{n},\underline{n}'}}_{L^2(\mu)}>0.
\end{equation}

Note that since $f_j\in \mathcal{E}(T)$ for $j=m+1,\ldots, \ell$, we have
$g_{j, \underline{n},\underline{n}'}\in \mathcal{E}(T)$ for $j=m+1,\ldots, \ell$, $\underline{n},\underline{n}'\in\N^s$. Moreover,
since $\chi_{\underline{n}}\in \mathcal{E}(T)$ for all $\underline{n}\in \N^s$ we get that $g_{m, \underline{n},\underline{n}'}\in \mathcal{E}(T)$ for all $\underline{n},\underline{n}'\in \N^s$.
Since by assumption property $(P_{m-1})$ of  Proposition~\ref{P:10} holds, we get that the previous expression remains unchanged if  for $j=1,\ldots, m-1$ we  replace the  functions
 $g_{j, \underline{n},\underline{n}'}$, $\underline{n},\underline{n}'\in\N^s$,  by constants (namely, their integrals).\footnote{When we deal with polynomial sequences $a_1,\ldots, a_\ell$,  we cannot replace
 	 the functions $g_{j, \underline{n},\underline{n}'}$  by constants but by their projection to the rational Kronecker factor, and this leads to serious complications. We bypass them in the proof of  Theorem~\ref{T:CharacteristicFactors} by appealing to Theorem~\ref{T:HostKra} below, which allows us to reduce matters to the case where the rational Kronecker factor is trivial.}

Next, note  that
\begin{equation}\label{E:eigen}
T^{a_m(n)} \chi_{\underline{n}}=e(a_m(n)\alpha_{\underline{n}})\, \chi_{\underline{n}}, \quad n\in \N,\, \underline{n}\in\N^s,
\end{equation}
for some $\alpha_{\underline{n}}\in \spec(T)$, $\underline{n}\in \N^s$.
Hence,
 we have that
$$
T^{a_m(n)}g_{m,\underline{n},\underline{n}'}=e(a_m(n)\beta_{\underline{n},\underline{n}'})
g_{m,\underline{n},\underline{n}'},
$$
where
$$
 \beta_{\underline{n},\underline{n}'}:=\sum_{\epsilon\in \{0,1\}^s} (-1)^{|\epsilon|}
 \alpha_{\underline{n}^\epsilon}.
 $$
(For instance, if $s=2$,  we have $ \beta_{\underline{n},\underline{n}'}= \beta_{n_1,n_2,n_1',n_2'}:=\alpha_{n_1,n_2}-\alpha_{n_1',n_2}-\alpha_{n_1,n_2'}+\alpha_{n_1',n_2'}$.)

For $j=m+1,\ldots, \ell$, $\underline{n}, \underline{n}'\in \N^s$, the eigenvalues of the eigenfunctions   $g_{j,\underline{n},\underline{n}'}$  can be expressed in the form  $e(\alpha_{j,\underline{n},\underline{n}'})$, where $\alpha_{j,\underline{n},\underline{n}'}\in \spec(T)$. Then  (recall the for $j<m$ the functions $g_{j,\underline{n},\underline{n}'}$  are assumed to be constant)
\begin{equation}\label{E:nn''}
\liminf_{N\to\infty}\E_{\underline{n}'  \in  [N]^s}\E_{\underline{n}\in   \Lambda\cap [N]^s}\, \limsup_{k\to\infty}
| \E_{n\in[N_k]}\,  e(a_m(n)\beta_{\underline{n},\underline{n}'}+\sum_{j=m+1}^\ell a_j(n)\alpha_{j,\underline{n}, \underline{n}'})|>0.
\end{equation}

From \eqref{E:nn''} we deduce that there exist 
integers $\underline{n}'_{N}\in [N]^s$,  and subsets $\Lambda_{N}$  of $\Lambda\cap [N]^s$  with
\begin{equation}\label{E:posden}
\liminf_{N\to\infty} \frac{|\Lambda_{N}|}{N^s}>0, \quad    N\in \N,
\end{equation}
and such that
$$
\limsup_{k\to\infty}  |\E_{n\in[N_k]}\,  e(a_m(n)\beta_{\underline{n},\underline{n}'_{N}}+\sum_{j=m+1}^\ell a_j(n)\alpha_{j,\underline{n}, \underline{n}'_N})| >0 , \quad  \underline{n}\in \Lambda_{N}.
$$
Since the sequences $a_1,\ldots, a_\ell$ are good for equidistribution for $(X,\mu,T)$ and
$\{\beta_{\underline{n},\underline{n}'_{N}}\}$, $\alpha_{j,\underline{n}, \underline{n}'_N}\in \spec(T) $,  $j=m+1,\ldots, \ell$,  $\underline{n}\in \Lambda_{N}, N\in\N$,
we deduce that
$$
 \beta_{\underline{n},\underline{n}_{N}'}=0 \pmod{1}, \quad  \underline{n}\in \Lambda_{N}, N\in\N.
 $$
 Then if $\underline{n}_N^\epsilon:=(n_1^{\epsilon_1}, \ldots, n_s^{\epsilon_s})$, where $n_j^0:=n_j$ and $n_j^1:=n_{j,N}'$ for $j=1,\ldots, s$, we get
   \begin{equation}\label{E:lessvar}
   \alpha_{\underline{n}}=
   -\sum_{\epsilon\in \{0,1\}^s\setminus \underline{0}}(-1)^{|\epsilon|}
   \alpha_{\underline{n}_N^\epsilon} \pmod{1}, \quad    \underline{n}\in \Lambda_{N},\,  N\in \N.
   \end{equation}
(For $s=2$ we get  $\alpha_{n_1,n_2}=\alpha_{n_{1,N}',n_2}+\alpha_{n_1,n_{2,N}'}-\alpha_{n_{1,N}',n_{2,N}'} \pmod{1}$.)
The important point is that we expressed  $\alpha_{\underline{n}}$ as a sum of sequences that depend on $s-1$ variables, chosen from the variables $n_1,\ldots, n_s$.



 Since  \eqref{E:U2first'} holds for all    $\underline{n}\in \Lambda_{N}\subset \Lambda$, $N\in \N$,   we deduce that 
$$
\liminf_{N\to\infty}	\E_{\underline{n}\in \Lambda_N}\,  \Re \Big( \int \Delta_{\underline{n}}\tilde{f}_m \cdot \chi_{\underline{n}}\, d\mu\Big)>0.
$$
Finally, we are going to combine this with \eqref{E:lessvar} in order to deduce that $\nnorm{\tilde{f}_m}_{s+1}>0$.
After
composing with the transformation $T^{n_s'}$, averaging over $n_s'\in [N]$,  and using that
$$
T^{n_s'} \chi_{\underline{n}}=e(n_s'\alpha_{\underline{n}})\cdot \chi_{\underline{n}},\quad  \underline{n}\in \N^s,\,  n_s'\in \N,
$$
 and then using the  Cauchy-Schwarz inequality  and \eqref{E:posden},
we deduce that
$$
\liminf_{N\to\infty}	\E_{\underline{n}\in  [N]^s}\,  \norm{ \E_{n_s'\in [N]} \,
T^{n_s'}(\Delta_{\underline{n}}\tilde{f}_m) \cdot e(n_s' \alpha_{\underline{n}})}_{L^2(\mu)}>0.
$$
  Using Lemma~\ref{L:VDC} for the average over  $n_s'$ and composing with $T^{-n_s'}$ on the integrals that arise,  we deduce that
$$
\liminf_{N\to\infty}  \Re \Big(	\E_{(\underline{n},n_{s+1},n_s')\in [N]^{s+2}}\, c_N(\underline{n}, n_{s+1},n_s')
    \int
\Delta_{(\underline{n},n_{s+1})}\tilde{f}_m \, d\mu\Big)>0,
$$
where
\begin{equation}\label{E:cNu}
 c_N(\underline{n}, n_{s+1},n_s') := {\bf 1}_{[N]}(n_s'+n_{s+1})\cdot e(n_{s+1} \alpha_{\underline{n}}), \quad \underline{n}\in [N]^s,\, n_{s+1}, n_s'\in [N],\,   N\in \N.
\end{equation}
Hence, for some $n_{s,N}'\in [N]$, $N\in \N$, we have
$$
\liminf_{N\to\infty}  \Re \Big(	\E_{(\underline{n},n_{s+1})\in [N]^{s+2}}\, c_N(\underline{n}, n_{s+1},n_{s,N}')
  \int
\Delta_{(\underline{n},n_{s+1})}\tilde{f}_m \, d\mu\Big)>0
$$
and  using the Cauchy-Schwarz inequality we deduce that
$$
\liminf_{N\to\infty}  \norm{	\E_{(\underline{n},n_{s+1})\in [N]^{s+1}}\, c_N(\underline{n},n_{s+1},n_{s,N}')   \cdot
\Delta_{(\underline{n},n_{s+1})}\tilde{f}_m }_{L^2(\mu)}>0.
$$

Using \eqref{E:lessvar} and \eqref{E:cNu} we get that  for  all   $N\in \N$ and  $\underline{n}\in \Lambda_{N}$ we have
$$
 c_N(\underline{n}, n_{s+1}) 
  =\prod_{j=1}^s b_{j,N}(\underline{n}, n_{s+1}), \qquad
   \underline{n}\in \Lambda_{N},\, n_{s+1}\in [N],\,  N\in \N,
$$
where   $b_{j,N}$ are bounded by $1$ and do not depend on the variable $n_j$, for $j=1,\ldots, s$, $N\in \N$. Since also  ${\bf 1}_{\Lambda_{N}}(\underline{n})$ does not depend on the variable $n_{s+1}$, we deduce
from Lemma~\ref{L:lower}   that
$$
\nnorm{\tilde{f}_m}_{s+1}>0.
$$
This completes the proof of Proposition~\ref{P:4'}.
\end{proof}

\subsection{Step 4}\label{SS:4.6} (Proof of Proposition~\ref{P:10})
Let the ergodic system $(X,\mu,T)$, the positive integer $\ell$,  and the sequences $a_1,\ldots, a_\ell$ be fixed.

    We first consider the case where $\ell=1$.  Note  that property $(P_0)$ of Proposition~\ref{P:10} is an immediate consequence of our assumption that the sequence $a_1$ is good for equidistribution for $(X,\mu,T_1)$.  It remains to show that for $\ell=1$  property $(P_0)$ implies property $(P_1)$.  We argue by contradiction.
  	We assume that property $(P_0)$ holds and there exists  $f_1\in L^\infty(\mu)$ such that  $ \int f_1\, d\mu=0$ but
$$
\limsup_{N\to\infty} \norm{\E_{n\in[N]}\, T_1^{a_1(n)}f_1}_{L^2(\mu)}>0.
$$
 Using  Proposition~\ref{P:11}  we get that there exists $\tilde{f}_1\in L^\infty(\mu)$  with $\int \tilde{f}_1\, d\mu=0$,
    and a sequence $N_k\to\infty$, such that
$$
\limsup_{k\to\infty}\norm{ \E_{n\in[N_k]}\, T_1^{a_1(n)}\tilde{f}_1}_{L^2(\mu)}>0.
$$
Since the sequence $a_1$ is good for seminorm estimates  for $(X,\mu, T_1)$, we deduce that $\nnorm{\tilde{f}_1}_s>0$ for some $s\geq 2$.  Since the assumptions of    Proposition~\ref{P:4'} are satisfied, we deduce  that
$\int \tilde{f}_1\, d\mu\neq 0$, which is a contradiction.

Suppose now that $\ell\geq 2$. We will prove that   property $(P_m)$ of  Proposition~\ref{P:10}   holds by induction on $m\in\{0,\ldots, \ell\}$.  We first consider the case   $m=0$. In this case our assumption is that
$Tf_j=e(t_j)f_j$ for some $t_j\in \spec(T)$, $j=1,\ldots, \ell$. Note first that if $t_j=0$ for $j=1,\ldots,\ell$,  then by ergodicity we have that $Tf_j=\int f_j \, d\mu$, $j=1,\ldots, \ell$, in which case \eqref{E:desired} is obvious.  Hence, in order to show that \eqref{E:desired} holds, it suffices to show that
$$
\lim_{N\to\infty}\E_{n\in [N]}\, e(a_1(n)t_1+\cdots + a_\ell(n)t_\ell) =0
$$
for $t_1,\ldots, t_\ell\in \spec(T)$, not all of them zero. This  holds since the sequences $a_1,\ldots, a_\ell$ are assumed to be good for equidistribution for the system $(X,\mu,T)$.

For $m\in [\ell]$ we assume that property $(P_{m-1})$ of  Proposition~\ref{P:10} holds and we are going to show that property $(P_m)$ holds.
Note  that in order to prove that  \eqref{E:desired} holds  it suffices to assume that at least one of  the functions $f_j$, for $j\in [\ell]\setminus\{m\}$,  has zero integral. Indeed, if we write
$f_j=\tilde{f}_j+\int f_j\, d\mu$ where $\tilde{f}_j=f_j-\int f_j\, d\mu$, 
$j\in [\ell]\setminus\{m\}$,  and expand the product on the average into a sum of $2^{\ell-1}$
terms,  we get a sum of $2^{\ell-1}-1$ averages with functions that have the required property, plus the term $\prod_{j\in [\ell]\setminus\{m\}}\int f_j\, d\mu \cdot \E_{n\in[N]}T^{a_m(n)}f_m$, which by the $\ell=1$ case we know that it converges in $L^2(\mu)$ to $\prod_{j\in [\ell]}\int f_j\, d\mu$.

So under this additional assumption, our goal is to show that
\begin{equation}\label{E:desired'}
\lim_{N\to\infty}\E_{n\in [N]}\, T^{a_1(n)}f_1 \cdot\ldots \cdot T^{a_\ell(n)}f_\ell=0
\end{equation}
where convergence takes place in $L^2(\mu)$. Arguing by contradiction, suppose that
$$
\E_{n\in [N]}\,  T^{a_1(n)}f_1 \cdot\ldots \cdot T^{a_\ell(n)}f_\ell \not\to 0
$$
in $L^2(\mu)$ as $N\to\infty$.
 Using  Propositions~\ref{P:11} we get that the same thing holds with  the function $\tilde{f}_m$,
  defined by \eqref{E:wlp}, in place of $f_m$. Recall that
 \begin{equation}\label{E:wlp'}
\tilde{f}_m=\lim_{k\to\infty} \E_{n\in [N_k]}
 \, T^{-a_m(n)}g_k\cdot \prod_{j\in [\ell], j\neq m}T^{a_j(n)-a_m(n)}\overline{f}_j,
 \end{equation}
for some $N_k\to\infty$ and   some $g_k\in L^\infty(\mu)$, $k\in \N$, where all functions are bounded by $1$,  and the limit is a weak limit.
 Since $f_{m+1},\ldots, f_\ell\in \mathcal{E}(T)$ and  the sequences $a_1,\ldots, a_\ell$ are good for seminorm estimates for the system $(X,\mu,T)$,
 we deduce that  $\nnorm{\tilde{f}_m}_{s}>0$, for some $s\in \N$.

  Since the assumptions of   Proposition~\ref{P:4'} are satisfied, we deduce that $\int \tilde{f}_m\, d\mu\neq 0$.
 Using \eqref{E:wlp'}, we deduce that
$$
\lim_{k\to\infty} \E_{n\in [N_k]} \, \int    T^{-a_m(n)}g_k\cdot \prod_{j\in [\ell], j\neq m}T^{a_j(n)-a_m(n)}\overline{f}_j\, d\mu\neq 0.
$$
Composing with $T^{a_m(n)}$  and using the Cauchy-Schwarz inequality  we get that
$$
\limsup_{k\to\infty}\norm{\E_{n\in [N_k]} \, \prod_{j\in [\ell], j\neq m}T^{a_j(n)}f_j}_{L^2(\mu)}>0.
$$
Since at least one of the  functions $f_j$, for $j\in [\ell]\setminus\{m\}$,  has zero integral, and $ f_{m+1},\ldots, f_\ell\in \mathcal{E}(T)$, using  property $(P_{m-1})$ of  Proposition~\ref{P:10} (with $f_m:=1\in \mathcal{E}(T)$) we get that the last limit is zero, a contradiction. We conclude that  property $(P_m)$ of  Proposition~\ref{P:10} holds. This completes the induction and the proof of  Proposition~\ref{P:10}.

\subsection{Proof of Corollary~\ref{C:MultipleRecurrence}}
Let $\mathcal{I}(T)$ be the $L^2(\mu)$-closed subspace of $T$-invariant functions. For $f\in L^2(\mu)$, we denote by  $\E(f|\mathcal{I}(T))$ the orthogonal projection, in $L^2(\mu)$, of $f$ onto $\mathcal{I}(T)$.
Using a standard argument we deduce from Theorem~\ref{T:JointErgodicity} that
the limit
$$
\lim_{N\to\infty}\E_{n\in [N]}\, \mu(A\cap T^{-a_1(n)}A\cap \cdots \cap T^{-a_\ell(n)}A)
$$
is equal to
$$
\int (\E({\bf 1}_A|\mathcal{I}(T)))^{\ell+1}\, d\mu\geq \Big(\int \E({\bf 1}_A|\mathcal{I}(T))\, d\mu\Big)^{\ell+1}=(\mu(A))^{\ell+1}.
$$



\subsection{Proof of Corollary~\ref{C:Nilsystems}}
It is clear that  $(i)$ implies $(ii)$.

We show that $(ii)$ implies $(i)$.
Let $(X,\mu,T)$ be an ergodic $k$-step nilsystem for some $k\in \N$.
By  \cite[Chapter~12, Theorem~17]{HK18} we have that if $f\in L^\infty(\mu)$ is such that $\nnorm{f}_{k+1}=0$,
then $f=0$. Hence, any collection of sequences is good for seminorm estimates for $(X,\mu,T)$. Since, by assumption the
sequences $a_1,\ldots, a_\ell$ are good for equidistribution for $(X,\mu,T)$, it follows from Theorem~\ref{T:JointErgodicity} that they are jointly ergodic for $(X,\mu,T)$.


\section{Characteristic factors of general  sequences}\label{S:Corollaries}
In this section we prove Theorem~\ref{T:CharacteristicFactors} and use it to deduce  Corollary~\ref{C:MultipleRecurrence'}.

\subsection{Proof of Theorem~\ref{T:CharacteristicFactors}} \label{SS:CF}
In this subsection we show how we can deduce  Theorem~\ref{T:CharacteristicFactors} from
Theorem~\ref{T:JointErgodicity}.
For this, we will need some deeper tools from ergodic theory
than those need in the proof of Theorem~\ref{T:JointErgodicity}.   In \cite{HK05a} it is shown that for every ergodic system $(X,\mathcal{X},\mu,T)$ and  $s\in\N$ there exists
a $T$-invariant sub-$\sigma$-algebra $\mathcal{Z}_s$ of $\mathcal{X}$ with the property that  for $f\in L^\infty(\mu)$ we have
$\E(f|\mathcal{Z}_s)=0$ if and only if $\nnorm{f}_{s+1}=0$.
We need some structural information about the
system $(X,\mathcal{Z}_s,\mu,T)$ that is a corollary of the main result in \cite{HK05a}. We
remark that this is  the only part of the article that we
make use of this structural theory.

 \begin{definition}
 	We say that a system $(X,\mu,T)$ has  {\em finite rational spectrum} if $\spec(T)\cap \mathbb{Q}$ is finite, or, equivalently, if  there exists $k\in \N$ such that the ergodic components of the system $(X,\mu,T^k)$ have trivial rational spectrum.
 \end{definition}
 \begin{theorem}[Host-Kra~\cite{HK05a}]\label{T:HostKra}
 	Let $(X,\mu,T)$ be an ergodic system such that $\mathcal{X}=\mathcal{Z}_s$ for some $s\in \N$. Then  $(X,\mu,T)$ is an inverse limit of systems with finite rational spectrum. 	
 \end{theorem}
\begin{remark}
	The main result in \cite{HK05a} states much more, namely,    $(X,\mu,T)$ is an inverse limit of nilsystems, but we will not need this.
	\end{remark}
 We are also going to use the following  fact:
\begin{lemma}\label{L:irrequi}
	Suppose that the sequences $a_1,\ldots, a_\ell\colon \N\to \Z$ are good for irrational equidistribution and seminorm estimates. For $k\in \N$ and  $i=1,\ldots, \ell$ let
	$$
	b_i(n):=\sum_{r=0}^{k-1}  {\bf 1}_{k\Z+r}(a_i(n)) \, \frac{a_i(n)-r}{k}, \quad n\in\N.
	$$
	Then	the sequences   $b_1,\ldots, b_\ell$ are   good for irrational equidistribution and seminorm estimates.
	\end{lemma}
\begin{proof}
We first establish the equidistribution statement.
Let $t_1,\ldots, t_\ell\in [0,1)$, not all of them rational.
We want to show that
$$
\lim_{N\to\infty}\E_{n\in [N]}\, e\big(\sum_{i=1}^\ell b_i(n)t_i\big)=0.
$$
By direct computation we see that it suffices to show that
$$
\lim_{N\to\infty}\E_{n\in [N]}\,  F(a_1(n),\ldots, a_\ell(n))\, e\big(\sum_{i=1}^\ell a_i(n)\frac{t_i}{k}\big)=0,
$$
where $F\colon \Z_k^\ell\to \C$ is defined by
$$
F(t_1,\ldots, t_\ell):=e\big(-\sum_{i=1}^\ell\sum_{r=1}^k {\bf 1}_{k\Z+r}(x_i)\frac{r}{k}t_i\big), \quad t_1,\ldots, t_\ell\in \Z_k.
$$
Using the Fourier expansion of $F$ on $\Z_k^\ell$ we see that it suffices to show that
$$
\lim_{N\to\infty}\E_{n\in [N]}\, e\big(\sum_{i=1}^\ell a_i(n)\frac{t_i+r_i}{k}\big)=0
$$
for all $r_1,\ldots, r_\ell\in \{0,\ldots, k-1\}$. Since, by assumption, the sequences $a_1,\ldots, a_\ell$ are good for irrational equidistribution, and $\frac{t_i+r_i}{k}$ is irrational for some $i\in \{1,\ldots, \ell\}$, the needed identity follows.

Next, we establish the statement about seminorm estimates.
	Let $(X,\mu,T)$ be a system and  $k\in \N$. Our plan is to show that
	the sequences $b_1,\ldots, b_\ell$ are good for seminorm estimates for the system $(X,\mu,T)$  by
	using that (by our assumption)  the sequences $a_1,\ldots, a_\ell$ are good for seminorm estimates for the ``$k$-th root'' of the system $(X,\mu,T)$, which is defined as follows:
	 	We consider measure preserving the system $(X_k,\mu_k,T_k)$ where
	$$
	X_k:=X\times \{0,\ldots, k-1\}, \quad \mu_k=\mu\times \nu_k, \quad \nu_k:=\frac{\delta_0+\cdots+\delta_{k-1}}{k},
	$$
	and for $x\in X$ we let
	$$
	T_k(x,i)=(x,i+1), \, i=0,\ldots, k-2, \quad T_k(x,k-1)=(Tx,0).
	$$
 	The key property is that
	$T_k^k(x,i)=(Tx,i), x\in X, i\in \{0,\ldots, k-1\}$, hence for $f\in L^\infty(\mu)$ we have
	\begin{equation}\label{E:Tk}
	(T_k^k(f\otimes 1))(x,i)=(Tf)(x), \quad x\in X, \, i\in  \{0,\ldots, k-1\}.
	\end{equation}

  Applying our good seminorm assumption  for the product system $(X_k\times X_k,\mu_k\times\mu_k,T_k\times T_k)$,
and taking into account the remarks following the definition in Section~\ref{SS:JE} for the good seminorm property,
we get that
there exists an   $s\in \N$ (we can assume that $s\geq 2$)
such that if  $g_1, \ldots, g_\ell\in L^\infty(\mu_k)$ and
 $\nnorm{g_m}_{s,T_k}=0$ for some $m\in [\ell]$, then for every bounded sequence $(c_n)$ we have
 \begin{equation}\label{E:assumption}
 \lim_{N\to\infty} \E_{n\in [N]}\, c_n\,  T_k^{a_1(n)}g_1\cdot\ldots \cdot  T_k^{a_m(n)}g_m= 0
 \end{equation}
  in $L^2(\mu_k)$.

 We claim that this $s\in \N$   produces  good seminorm estimates for the sequences $b_1,\ldots, b_\ell$ for the system $(X,\mu,T)$. To see this let $f_1,\ldots, f_\ell\in L^\infty(\mu)$ be such that  $\nnorm{f_m}_{s,T}=0$ for some $m\in [\ell]$ and $f_{m+1},\ldots, f_\ell\in \mathcal{E}(T)$. It suffices to show that   for every bounded sequence $(d_n)$ we have
 $$
 \lim_{N\to\infty} \E_{n\in [N]}\, d_n\,  T^{b_1(n)}f_1\cdot\ldots \cdot  T^{b_m(n)}f_m= 0
 $$
 in $L^2(\mu)$.
 Let $g_i\in L^\infty(\mu_k)$ be defined by
 $g_i:=f_i\otimes 1$, $i=1,\ldots, \ell$, and recall that $\mu_k=\mu\times \nu_k$. By \eqref{E:Tk}  it suffices to show that
 $$
 \lim_{N\to\infty} \E_{n\in [N]}\,  d_n\, T_k^{kb_1(n)}g_1\cdot\ldots \cdot  T_k^{ka_m(n)}g_m= 0
 $$
 in $L^2(\mu_k)$, or equivalently, that
 $$
 \lim_{N\to\infty} \E_{n\in [N]}\,  d_n\, T_k^{a_1(n)}(T_k^{-r_1(n)}g_1)\cdot\ldots \cdot  T_k^{a_m(n)}(T_k^{-r_m(n)}g_m)= 0
 $$
 in $L^2(\mu_k)$, 	for some  sequences $r_1,\ldots, r_m\colon \N\to\{0,\ldots, k-1\}$. Note that
  the last average can be written as a sum of $k^m$ weighted averages of the form
  \begin{equation}\label{E:assumption'}
  \E_{n\in [N]}\, d'_n\,  T_k^{a_1(n)}(T_k^{-r_1}g_1)\cdot\ldots \cdot  T_k^{a_m(n)}(T_k^{-r_m}g_m)
  \end{equation}
 for some $r_1,\ldots, r_m\in \{0,\ldots, k-1\}$ and $d'_n\in \{0,d_n\}$, $n\in\N$. Hence, it suffices to show that averages of the form \eqref{E:assumption'}
 converge to $0$ in $L^2(\mu_k)$ and $N\to\infty$.
	
	 Since $g_m=f_m\otimes 1$ and $\mu_k=\mu\times \nu_k$, our assumption $\nnorm{f_m}_{s,T}=0$ and  \eqref{E:Tk} give that
	 $\nnorm{g_m}_{s,T_k^k}=\nnorm{f_m}_{s,T}=0$.

 Since $\CI(T_k)\subset \CI(T_k^k)$  we deduce from \eqref{E:TS} the implication
  $$
	\nnorm{g_m}_{s,T_k^k}=0 \implies \nnorm{g_m}_{s,T_k}=0,
	$$ hence we also have  $\nnorm{T_k^{-r_m}g_m}_{s,T_k}=\nnorm{g_m}_{s,T_k}=0$. It then follows from \eqref{E:assumption} that the averages  \eqref{E:assumption'} converge to $0$ in $L^2(\mu_k)$ as $N\to\infty$. This completes the proof.
\end{proof}

\begin{proof}[Proof of Theorem~\ref{T:CharacteristicFactors}]
		It is straightforward to verify that Property~$(i)$ implies Property~$(ii)$ (the seminorm property holds with $s=2$ and one can use appropriate rotations on $\T^\ell$ to verify the  equidistribution property).
	So we only prove that Property~$(ii)$ implies Property~$(i)$.

	Let $\mu=\int \mu_x\, d\mu$ be the ergodic decomposition of the measure $\mu$. Since
	 $\E(f|\mathcal{K}_{rat}(\mu))=0$ implies that $\E(f|\mathcal{K}_{rat}(\mu_x))=0$ for $\mu_x$ almost every $x\in X$ (see for example \cite[Theorem~3.2]{FrK06}), we can assume that the system $(X,\mu,T)$ is ergodic.

We claim that it suffices to show that  the Kronecker factor $\mathcal{K}(T)$ is characteristic for mean convergence of the averages
	\begin{equation}\label{E:ain}
	\E_{n\in [N]}\, T^{a_1(n)}f_1 \cdot\ldots \cdot  T^{a_\ell(n)}f_\ell,
	\end{equation}
	meaning that the previous averages converge to $0$ in $L^2(\mu)$ as $N\to\infty$ if at least one of the functions is orthogonal to $\mathcal{K}(T)$ (recall that $\mathcal{K}(T)$ is $L^2(\mu)$-closure of the linear span of   the eigenfunctions of the system).
Indeed, if this is the case, then by approximation, we can assume that all functions are eigenfunctions; hence,  for $i=1,\ldots, \ell$ we have
$$
T^{a_i(n)}f_i=e(a_i(n)t_i)\, f_i, \quad n\in \N,
$$
for some $t_i\in [0,1)$, not all of them rational (since at least one of the functions is orthogonal to $\mathcal{K}_{rat}(T)$). In this case, the needed convergence to zero follows from
our assumption that the sequences $a_1,\ldots, a_\ell$ are good for irrational equidistribution.

So it remains to show  that if $\E(f_i|\mathcal{K}(T))=0$,  for some $i\in \{1,\ldots, \ell\}$, then the  averages \eqref{E:ain} converge to $0$ in $L^2(\mu)$. Without loss of generality we can assume that $i=1$, hence  $\E(f_1|\mathcal{K}(T))=0$.
	Since the sequences $a_1,\ldots, a_\ell$ are very good for seminorm estimates and $\E(f|\mathcal{Z}_s(T))=0$ implies $\nnorm{f}_{s+1,T}=0$, we can assume that all functions $f_1,\ldots, f_\ell$ are $\mathcal{Z}_s(T)$-measurable for some $s\in \N$. In this case we can assume that  $\mathcal{X}=\mathcal{Z}_s(T)$,\footnote{The assumption that the sequences $a_1,\ldots, a_\ell$ are very good for seminorm estimates (versus simply  ``good'') is  needed in order to get this reduction.} hence  using Theorem~\ref{T:HostKra}  and an approximation argument we can assume that the system $(X,\mu,T)$  has finite rational spectrum, in which case  there exists $k\in \N$ such  the system $(X,\mu,T^k)$ has trivial rational spectrum.

	For $i=1,\ldots, \ell$ we let
	$$
	b_i(n):=\sum_{r=0}^{k-1}  {\bf 1}_{k\Z+r}(a_i(n)) \, \frac{a_i(n)-r}{k}, \quad n\in\Z.
	$$
	By Lemma~\ref{L:irrequi}, the sequences $b_1,\ldots, b_\ell$ are good for irrational equidistribution and seminorm estimates.
	Then for $i=1,\ldots, \ell$ and  $S:=T^k$ we have that 	
	$$
	T^{a_i(n)}f_i=\sum_{r=0}^{k-1} {\bf 1}_{k\Z+r}(a_i(n))\, (T^k)^{\frac{a_i(n)-r}{k}}(T^rf_i)=
	\sum_{r=0}^{k-1} {\bf 1}_{k\Z+r}(a_i(n))\, S^{b_i(n)}(T^rf_i), \quad n\in\N.
	$$
	We insert this identity in \eqref{E:ain} and expand the product.
	We deduce
	that it suffices to show the following: If the (not necessarily ergodic) system $(X,\mu,S)$ has trivial rational spectrum, the sequences $b_1,\ldots, b_\ell\colon \N\to \Z$ are good for irrational equidistribution and seminorm estimates,
	and $\E(f_1|\mathcal{K}(T))=0$, then for every bounded sequence $(c_n)$ we have
	$$
\lim_{N\to\infty}\E_{n\in [N]}\, c_n\cdot S^{b_1(n)}f_1 \cdot\ldots \cdot  S^{b_\ell(n)}f_\ell=0
	$$
	in $L^2(\mu)$.
	
 To prove this, we first remark that  since the system $(X,\mu,S)$ has trivial  rational spectrum, the same holds for the system
  	$(X\times X,\mu\times\mu,S\times S)$. Furthermore,  it is known that
	$\E(f_1|\mathcal{K}(S))=0$ implies that
	 $\E(f_1\otimes \overline{f}_1|\mathcal{I}(S\times S))=0$. Combining these facts (whose proof follows for example from \cite[Lemma~4.18]{Fu81a}) we deduce that  if $$
	 \mu\times \mu=\int (\mu\times\mu)_{(x,y)}\, d(\mu\times\mu)
	 $$ is the ergodic decomposition of the measure $\mu\times \mu$ with respect to  the transformation $S\times S$, then
	  for  $(\mu\times\mu)$-almost every $(x,y)\in X\times X$ the system 	$(X\times X,(\mu\times\mu)_{(x,y)},S\times S)$ is ergodic, has trivial rational spectrum, and  $\int f_1\otimes \overline{f}_1\, d(\mu\times\mu)_{(x,y)}=0$. By
	  Theorem~\ref{T:JointErgodicity}  
	  we deduce that  for  $(\mu\times\mu)$-almost every $(x,y)\in X\times X$ we have
	 $$
	\E_{n\in [N]}\, (S\times S)^{b_1(n)}(f_1\otimes\overline{f}_1) \cdot\ldots \cdot  (S\times S)^{b_\ell(n)}(f_\ell\otimes \overline{f}_\ell)\to^{L^2((\mu\times\mu)_{(x,y)})}0
	$$
	as $N\to \infty$.
	 This implies that
    $$
		\E_{n\in [N]}\, (S\times S)^{b_1(n)}(f_1\otimes\overline{f}_1) \cdot\ldots \cdot  (S\times S)^{b_\ell(n)}(f_\ell\otimes \overline{f}_\ell)\to^{L^2(\mu\times\mu)}0
	$$
	as $N\to \infty$,
	and using the Cauchy-Schwarz inequality we deduce that
	$$
	\lim_{N\to\infty}\E_{n\in [N]}\, \int (f_{0,N}\otimes\overline{f}_{0,N})\cdot (S\times S)^{b_1(n)}(f_1\otimes\overline{f}_1) \cdot\ldots \cdot  (S\times S)^{b_\ell(n)}(f_\ell\otimes \overline{f}_\ell)\, d(\mu\times\mu)=0
	$$
	for all functions $f_{0,N}\in L^\infty(\mu)$, $N\in \N$, that are uniformly bounded.
	Hence,
		$$
\lim_{N\to\infty}	\E_{n\in [N]}\, \Big|\int f_{0,N} \cdot  S^{b_1(n)}f_1 \cdot\ldots \cdot  S^{b_\ell(n)}f_\ell\, d\mu\Big|^2= 0,
	$$
	which implies using the Cauchy-Schwarz inequality  that
		$$
\lim_{N\to\infty}	\E_{n\in [N]}\,c_n\,  \int f_{0,N} \cdot  S^{b_1(n)}f_1 \cdot\ldots \cdot  S^{b_\ell(n)}f_\ell\, d\mu= 0
	$$
	for every bounded sequence $(c_n)$.
	If we let
	$$
	f_{0,N}:=	\overline{\E_{n\in [N]}\, c_n \cdot  S^{b_1(n)}f_1 \cdot\ldots \cdot  S^{b_\ell(n)}f_\ell}, \quad N\in \N,
	$$
	we deduce that
		$$
	\lim_{N\to\infty}\E_{n\in [N]}\, c_n \cdot   S^{b_1(n)}f_1 \cdot\ldots \cdot  S^{b_\ell(n)}f_\ell = 0
	$$
	in $L^2(\mu)$. This completes the proof.
\end{proof}

\subsection{Proof of Corollary~\ref{C:MultipleRecurrence'}}
Let $A\in \mathcal{X}$ and $\varepsilon>0$. Let $\mathcal{K}_r$ denote the closed subspace of $L^2(\mu)$ consisting of  all $T^r$-invariant functions. Then there exists $r\in \N$ such that
\begin{equation}\label{E:Krat}
\norm{\E({\bf 1}_A|\mathcal{K}_{rat})-\E({\bf 1}_A|\mathcal{K}_{r})}_{L^2(\mu)}\leq \frac{\varepsilon}{\ell}.
\end{equation}

Let
$$
S_r=
\{n\in \N\colon a_1(n)\equiv 0, \ldots, a_\ell(n)\equiv 0 \!  \!  \! \pmod{r}\}.
$$
By assumption, we have  $\bar{d}(S_r)>0$, hence there exist $N_k\to \infty$ such that
\begin{equation}\label{E:dSr}
\lim_{k\to\infty}\frac{|S_r\cap [N_k]|}{N_k}>0.
\end{equation}

First, we claim that $\mathcal{K}_{rat}(T)$ is a characteristic factor for the averages
\begin{equation}\label{E:SrNk}
\E_{n\in S_r\cap [N_k]}\,  T^{a_1(n)}f_1\cdot \ldots \cdot T^{a_\ell(n)}f_\ell,
\end{equation}
meaning, if $\E(f_{j_0}|\mathcal{K}_{rat}(T))=0$ for some $j_0\in \{1,\ldots, \ell\}$, then the averages converge to zero in $L^2(\mu)$ as $k\to \infty$.
To see this, note that for every $k\in \N$ the averages \eqref{E:SrNk} are equal to
\begin{equation}\label{E:SrNk'}
\frac{N_k}{|S_r\cap [N_k]|}\cdot
\E_{n\in [N_k]}\,  (T\times R)^{a_1(n)}(f_1\otimes g)\cdot \ldots \cdot T^{a_\ell(n)}(f_\ell\otimes g),
\end{equation}
where $R$ is the shift transformation on a cyclic group of order $r$ and $g$ is the indicator function of the identity element in this cyclic group.
Then  $\E(f_{j_0}\otimes g|\mathcal{K}_{rat}(T\times R))=0$, and since by assumption the sequences $a_1,\ldots, a_\ell$ are very good for seminorm estimates and
good for irrational equidistribution, we get by Theorem~\ref{T:CharacteristicFactors} that the  averages in \eqref{E:SrNk} converge to $0$ in $L^2(\mu)$ as $k\to \infty$.
We deduce from this and \eqref{E:dSr} that the averages \eqref{E:SrNk'} converge to $0$ in $L^2(\mu)$ as $k\to \infty$, completing the proof of our claim.

It follows from what we just proved that the limit
$$
\liminf_{k\to \infty}\E_{n\in S_r\cap [N_k]}\, \mu(A\cap T^{-a_1(n)}A\cap \cdots \cap T^{-a_\ell(n)}A)
$$
is equal to the limit
$$
\liminf_{k\to \infty}\E_{n\in S_r\cap [N_k]}\int \E({\bf 1}_A|\mathcal{K}_{rat}(T))\cdot T^{a_1(n)}\E({\bf 1}_A|\mathcal{K}_{rat}(T))
\cdot \ldots \cdot T^{a_\ell(n)}\E({\bf 1}_A|\mathcal{K}_{rat}(T))\, d\mu.
$$
Using \eqref{E:Krat}  and telescoping, we get that the last limit  is greater or equal than
$$
\liminf_{k\to \infty}\E_{n\in S_r\cap [N_k]}\int \E({\bf 1}_A|\mathcal{K}_{r}(T))\cdot T^{a_1(n)}\E({\bf 1}_A|\mathcal{K}_{r}(T))
\cdot \ldots \cdot T^{a_\ell(n)}\E({\bf 1}_A|\mathcal{K}_{r}(T))\, d\mu -\varepsilon.
$$

Note that for $n\in S_r$ we have $T^{a_j(n)}\E({\bf 1}_A|\mathcal{K}_{r}(T))=\E({\bf 1}_A|\mathcal{K}_{r}(T))$, $j=1,\ldots, \ell$, hence the last limit is equal to
$$
\int (\E({\bf 1}_A|\mathcal{K}_{r}(T)))^{\ell+1}\, d\mu\geq \Big(\int \E({\bf 1}_A|\mathcal{K}_{r}(T))\, d\mu\Big)^{\ell+1}=(\mu(A))^{\ell+1}.
$$
Combining the above estimates we get that
$$
\liminf_{k\to \infty}\E_{n\in S_r\cap [N_k]}\, \mu(A\cap T^{-a_1(n)}A\cap \cdots \cap T^{-a_\ell(n)}A)\geq (\mu(A))^{\ell+1}-\varepsilon,
$$
completing the proof of Corollary~\ref{C:MultipleRecurrence'}.



\section{Joint ergodicity   of special sequences}\label{S:Hardy}
In this section we prove the results of Sections~\ref{SS:16} and  \ref{SS:17}.

 \subsection{Definition of Hardy fields}\label{SS:Hardy}
 Let $B$ be the collection of equivalence classes of real valued
 functions  defined on some half line $[c,+\infty)$, where we
 identify two functions if they agree eventually.\footnote{The
 	equivalence classes just defined are often called ``germs of
 		functions''. We choose to use the word function when we refer to
  	elements of $B$ instead, with the understanding that all the
 	operations defined and statements made for elements of $B$ are
 	considered only for sufficiently large values of $t\in \R$.}
 A
 \emph{Hardy field} $\mathcal{H}$ is a subfield of the ring $(B,+,\cdot)$ that is
 closed under differentiation.
 For the purposes of this section we assume that  all Hardy
 	fields $\mathcal{H}$ considered are contained in some other Hardy field $\mathcal{H}'$ that  satisfies property~\eqref{E:Hardy}.
 A particular example of such a Hardy field $\mathcal{H}$ is the collection  of
 \emph{logarithmic-exponential functions}, meaning all functions
 defined on some half line $[c,+\infty)$ by a finite combination of
 the symbols $+,-,\times, :, \log, \exp$, operating on the real
 variable $t$ and on real constants; linear combinations of functions of the form  $t^a(\log{t})^b c^t$, $a,b\in \R$, $c>0$, are examples of such functions.
 The reader can find more information about Hardy fields in \cite{Bos94} and the references therein.

 \subsection{Good equidistribution properties for Hardy field sequences}
 We will use  the following equidistribution result:
 \begin{theorem}[Boshernitzan~\cite{Bos94}]\label{T:Boshernitzan}
 	Let $a\colon [c,+\infty)\to \R$ be a Hardy field function with at most polynomial growth. Then the sequence
 	$(a(n))$ is equidistributed on $\T$  	if and only if it stays logarithmically away from rational polynomials (see definition in Section~\ref{SS:16}).
 \end{theorem}
We will also use the following reduction, variants of which have been frequently used in the literature. We give its proof for completeness.
\begin{lemma}\label{L:Equi}
	Let $a_1,\ldots, a_\ell\colon [c,+\infty)\to \R$ be such that for every
	  $t_1,\ldots, t_\ell\in \R$, not all of them zero, we have
	\begin{equation}\label{E:wanteda}
	\lim_{N\to\infty} \E_{n\in[N]}\,  e(a_1(n)t_1+\cdots+ a_\ell(n)t_\ell) =0.
	\end{equation}
	Then the sequences $[a_1(n)],\ldots, [a_\ell(n)]$ are good for equidistribution.
\end{lemma}
\begin{remark}
	If we assume that \eqref{E:wanteda} holds for all  $t_1,\ldots, t_\ell\in \R$, not all of them
	rational, then a similar argument gives that the sequences $[a_1(n)],\ldots, [a_\ell(n)]$ are good for irrational equidistribution.
	\end{remark}
\begin{proof}
	We first remark that it suffices to show the following: If $F\colon \T\to \C$ and $G\colon \T^\ell\to \R$ are Riemann-integrable, then
	for all $t_1,\ldots, t_\ell\in [0,1)$, not all of them $0$,  we have
	\begin{equation}\label{E:wantedb}
	\lim_{N\to\infty} \E_{n\in[N]}\, F(a_1(n)t_1+\cdots+ a_\ell(n)t_\ell) \, G(a_1(n),\ldots, a_\ell(n))=\int F\, dm_\T\cdot \int G \, dm_{\T^\ell}.
	\end{equation}
	Indeed, if this is the case, then using \eqref{E:wantedb}
	for the continuous function $F(x):=e(x)$, $x\in \T$, and  the Riemann integrable function $G(x_1,\ldots, x_\ell):=e(-\{x_1\}t_1-\cdots- \{x_\ell\}t_\ell)$, $x_1,\ldots, x_\ell\in \T$,
	we get that
	$$
	\lim_{N\to\infty} \E_{n\in[N]}\, e([a_1(n)]t_1+\cdots+ [a_\ell(n)]t_\ell)=0
	$$
	for all $t_1,\ldots, t_\ell\in [0,1)$, not all of them $0$.
	
	We move now to the proof of \eqref{E:wantedb}. After approximating from above and below by continuous functions, we can assume  that $F\in C(\T)$ and $G\in C(\T^\ell)$. After a further approximation by trigonometric polynomials we can assume that $F(x)=e(kx)$, $x\in \T$, and  $G(x_1,\ldots, x_\ell)=e(k_1x_1+\cdots + k_\ell x_\ell)$, $x_1,\ldots, x_\ell\in \T$, for some $k,k_1,\ldots, k_\ell\in \Z$.
	If $k=k_1=\cdots=k_\ell=0$, then the identity is obvious. Hence,  it suffices to show that for all $k,k_1,\ldots, k_\ell\in \Z$, not all of them zero, and $t_1,\ldots, t_\ell\in [0,1)$, not all of them zero, we have
	$$
	\lim_{N\to\infty} \E_{n\in[N]}\,e((k_1+kt_1)a_1(n)+\cdots + (k_\ell+kt_\ell) a_\ell(n))=0.
	$$
	Since $k_1+kt_1,\ldots, k_\ell+kt_\ell$ are not all of them zero, this follows from our assumptions, completing the proof.
\end{proof}

Combining the previous two results we get the following:
 	\begin{proposition}\label{L:WeylHardy}
 		Let $a_1,\ldots, a_\ell\colon [c,+\infty)\to \R$ be functions of at most polynomial growth from a Hardy field  such that every non-trivial linear combination 
 		of these functions stays logarithmically away from rational polynomials.
Then
  the sequences $[a_1(n)],\ldots, [a_\ell(n)]$ are good for equidistribution.
 	\end{proposition}
In a similar fashion, using the variant recorded on the remark following Lemma~\ref{L:Equi},
we get the following:
	\begin{proposition}\label{L:HardyPoly}
Let  $a_1,\ldots, a_\ell\colon [c,+\infty)\to \R$ be functions of at most polynomial growth from a Hardy field  such that every non-trivial linear combination   of these functions, with at least one irrational coefficient,
stays logarithmically away from rational polynomials.
Then
	the sequences $[a_1(n)],\ldots, [a_\ell(n)]$ are good for irrational equidistribution.
\end{proposition}

 \subsection{Good seminorm estimates  for Hardy field sequences}
We will use the following known result.\footnote{As far as we know this result is proved in detail only under the assumption \eqref{E:Hardy}, which is the reason why we impose this assumption on all Hardy fields considered in this section.}
 \begin{proposition}\label{P:SeminormHardy}
 Let $a_1,\ldots, a_\ell\colon [c,+\infty)\to \R$ be functions from a Hardy field such that
 the functions and their pairwise differences are non-constant functions in $\mathcal{T}+\mathcal{P}$. Then the  sequences
 $[a_1(n)],\ldots, [a_\ell(n)]$ are very  good for seminorm estimates.
 \end{proposition}
If $\mathcal{T}$ is replaced with the class of functions $a\colon \R_+\to \R$  that satisfy the slightly more restrictive growth condition  $t^{k+\varepsilon}\prec a(t)\prec t^{k+1}$ for some $k\in \Z_+$, then the argument used to prove \cite[Theorem~2.9]{Fr10} can be applied without any change to prove Proposition~\ref{P:SeminormHardy}. For the more extended class of functions used above one can employ the argument used to prove  \cite[Theorem~4.2]{BMR21} without essential changes. We omit the details.

\subsection{Proof of Theorem~\ref{T:Hardy}}
	By Proposition~\ref{P:SeminormHardy} the sequences $[a_1(n)],\ldots, [a_\ell(n)]$ are good for seminorm estimates, and by Proposition~\ref{L:WeylHardy} they are also good for equidistribution.  Hence by Theorem~\ref{T:JointErgodicity} they are jointly ergodic.


\subsection{Proof of Theorem~\ref{T:HardyPolyChar}}
By Proposition~\ref{P:SeminormHardy} the sequences $[a_1(n)],\ldots, [a_\ell(n)]$ are very good for seminorm estimates 	and by Proposition~\ref{L:HardyPoly} they are good for irrational equidistribution.  Hence, by Theorem~\ref{T:JointErgodicity}
(or  Corollary~\ref{C:JointTotalErgodicity})
they are jointly ergodic for totally ergodic systems and by  Theorem~\ref{T:CharacteristicFactors} the rational Kronecker factor  is characteristic for these sequences.

\subsection{Proof of Theorem~\ref{T:NilsystemSpecial}}
By Corollary~\ref{C:Nilsystems} it suffices to show that these collections of sequences are good for equidistribution. By Lemma~\ref{L:Equi} it suffices to show that, after removing the integer parts, every non-trivial linear combination of the given collections of sequences is equidistributed on the circle.

For the sequences in Part~$(i)$ this follows by combining \cite[Lemma~9]{BP98} (which is the main result in \cite{Kar71}) with \cite[Lemma~16]{BP98}.

 For the sequences in Parts~$(ii)$ and $(iii)$ this follows from
 \cite[Theorem~2.2]{BBK02} and \cite[Theorem~4.1]{BK90}.


\section{Joint ergodicity for flows}\label{S:Flows}
In this section we  prove Theorem~\ref{T:flows}, which we repeat for convenience.
 \begin{theorem}
	Let $a_1,\ldots, a_\ell\colon [c,+\infty)\to \R_+$ be functions from a Hardy field. Suppose that there exists  
	 $\delta>0$ such that  $t^\delta\prec a_1(t)$ and
	$(a_{j+1}(t))^\delta\prec a_j(t)\prec (a_{j+1}(t))^{1-\delta}$ for $j=1,\ldots, \ell-1$.
	Then for all measure preserving actions  $T_1^t,\ldots, T_\ell^t$, $t\in\R$,  on a probability space $(X,\mathcal{X},\mu)$ and $f_1,\ldots, f_\ell\in L^\infty(\mu)$, we have
	$$
	\lim_{y\to+\infty} \frac{1}{y} \int_0^y f_1(T_1^{a_1(t)}x)\cdot\ldots\cdot f_\ell(T_\ell^{a_{\ell}(t)}x)\, dt=\tilde{f_1}\cdots \tilde{f_\ell}
	$$
	pointwise for $\mu$-almost every $x\in X$, where for $j=1,\ldots, \ell$ we denote by  $\tilde{f_j}$  the orthogonal projection of $f_j$ on the space of functions that are $T_j^t$-invariant for every $t\in \R$.
\end{theorem}


We start with the following crucial change of variables property (a variant of this property also appears in \cite{Au11c}):
\begin{lemma}\label{L:CV}
	Let $a\colon \R_+\to \R_+$ be a  function from a Hardy field such that   $t^\delta\prec a(t)$ for some $\delta>0$.
Let
 $f\in L^\infty(m_\R)$ and	suppose that the following limit exists
	$$
\lim_{y\to+\infty}	\frac{1}{y}\int_0^y f(t)\, dt.
	$$
 	Then also  the following  limit exists
		$$
	\lim_{y\to+\infty}	\frac{1}{y}\int_0^{y} f(a(t)) \,  dt
	$$
	and the two limits are equal.
	\end{lemma}
\begin{remark}
	More generally, our argument works if there exist $\delta, M>0$ such that the functions  $a,a^{-1}$ are three times differentiable, their derivatives are non-zero  on $[M,+\infty)$,   and also $a(t)\geq t^\delta$ and  $\frac{a(t)}{ta'(t)}$ is bounded for $t\geq M$.
	\end{remark}
\begin{proof}
	We can assume that the first limit is $0$.
	Let
		$$
	F(t)=\frac{1}{t}\int_0^t f(s)\, ds, \quad t>0.
	$$
	Let $\varepsilon>0$. Using our assumption we have that there exists $M>0$ such that $|F(t)|\leq \varepsilon$ for $t>M$ and also $a,a^{-1}\in C^3([M,+\infty))$.
	Since $f$ is bounded, 
	it suffices to show that $\lim_{y\to +\infty}\frac{1}{y}\int_{a^{-1}(M)} ^y f(a(t))\, dt=0$.	
	We assume that $y$ is large enough so that $a(y)\geq M$. Using the change of variables $s=a(t)$ we get
		$$
		\frac{1}{y}\int_{a^{-1}(M)} ^{y} f(a(t)) \,  dt=	\frac{1}{y}\int_M^{a(y)} f(t) \cdot (a^{-1})'(t)\, dt.
	$$
	
Since
$
f(t)=(t F(t))'
$
for Lebesgue almost every $t\in  \R$,
we have
$$
	\frac{1}{y}\int_M^{a(y)} f(t) \cdot (a^{-1})'(t)\, dt
	=	\frac{1}{y}\int_M^{a(y)} (t F(t))' \cdot (a^{-1})'(t)\, dt.
$$
Integration by parts  ($tF(t)$ and $(a^{-1})'(t)$ are absolutely continuous on $[M,a(y)]$) gives that the last integral is equal to
\begin{equation}\label{E:1y}
\frac{a(y) F(a(y)) (a^{-1})'(a(y))}{y}- \frac{C}{y}- \frac{1}{y}\, \int_M^{a(y)} t F(t) \cdot (a^{-1})''(t)\, dt
\end{equation}
for some $C\in \R_+$.

The first term in \eqref{E:1y} is equal to
$$
\frac{a(y) F(a(y)) (a^{-1})'(a(y))}{y}=\frac{\int_0^{a(y)}f(t)\, dt}{ya'(y)} =\frac{\int_0^{a(y)}f(t)\, dt}{a(y)}\cdot \frac{a(y)}{ya'(y)}.
$$
 Notice that since $a(t)$ is a Hardy field function with $a(t)\succ t^\delta$, we have that
   $$
 \lim_{y\to+\infty} \frac{a(y)}{ya'(y)}=\lim_{y\to+\infty}\frac{\log{y}}{\log{a(y)}}<\frac{1}{\delta}.
 $$
 Moreover, by assumption we have
 $\lim_{y\to+\infty}	\frac{1}{y}\int_0^y f(t)\, dt=0$  and    $\lim_{y\to+\infty}a(y)= +\infty$. We deduce  that the first term  in \eqref{E:1y} converges to $0$ as $y\to+\infty$.

It remains to show that the limsup  as $y\to+\infty$ of last term in \eqref{E:1y} is bounded by a constant multiple of $\varepsilon$.  Since $|F(t)|\leq \varepsilon$ for $t>M$, it suffices to show that the limsup  as $y\to+\infty$ of the expression
$$
\frac{1}{y}\int_M^{a(y)} t  \cdot |(a^{-1})''(t)|\, dt
$$
is bounded by a quantity that is independent of $\varepsilon$ . Since $(a^{-1})'' $ has eventually constant sign, say in $[M_1,+\infty)$,  we can replace $M$ with $M_1$ and remove the absolute value.
After doing so, integration by parts leads to the expression
$$
\frac{1}{y}\Big( a(y) \cdot (a^{-1})'(a(y))-C_1 -\int_{M_1}^{a(y)}  (a^{-1})'(t)\, dt\Big)
$$
for some $C_1\in \R_+$.
The last expression  is equal to
$$
\frac{a(y)}{ya'(y)} - \frac{C_1}{y}  -\frac{y-a^{-1}(M_1)}{y}.
$$
As we showed before, the limit of the first term  is bounded and the limit of the other two terms as $y\to+\infty$ is $1$. This completes the proof.
\end{proof}

We review some basic facts from the spectral theory of unitary $\R$-actions that we will use. Proofs of the stated facts can be found for example in \cite{EW17}
(we use Theorem~9.58 and the variant of Theorems~9.17 that applies to flows).  If $(X,\mu,T^t)$, $t\in \R_+$, is a  measure preserving flow, then   for every $f\in L^2(\mu)$ there exists a positive and bounded measure $\sigma_f$ on $\R$, that is called the {\em spectral measure of $f$}, such that
\begin{equation}\label{E:spectral}
\int T^tf\cdot \overline{f}\, d\mu=\int e(ts)\, d\sigma_f(s), \quad t\in \R.
\end{equation}
 More generally,  for every $f,g\in L^2(\mu)$ there exists a complex measure $\sigma_{f,g}$,  with bounded variation, such that
$$
\int T^tf\cdot \overline{g}\, d\mu=\int e(ts)\, d\sigma_{f,g}(s), \quad t\in \R.
$$
Furthermore,  if $h\in L^\infty(\R)$, then there exists a bounded operator $h(T)\colon L^2(\mu)\to L^2(\mu)$ that commutes with $T^t$, $t\in \R$,  and satisfies
$$
\int h(T)f \cdot \overline{g}\, d\mu=\int h\, d\sigma_{f,g}
$$
for all $f,g\in L^2(\mu)$. We then have
\begin{equation}\label{E:p1}
(h_1+h_2)(T)=h_1(T)+h_2(T), \quad h_1,h_2\in L^\infty(\mu),
\end{equation}
and
\begin{equation}\label{E:p2}
d\sigma_{h(T)f}=h\, d\sigma_f, \quad \norm{h(T)f}_{L^2(\mu)}=\norm{h}_{L^2(\sigma_f)}.
\end{equation}
We will use the following fact:
\begin{lemma}\label{L:compact}
	Let $(X,\mu,T^t)$, $t\in \R_+$, be a measure preserving flow. Then the set
	$$
	\mathcal{G}:=\{f\in L^2(\mu)\colon \sigma_f \text{ has compact support}\}
	$$
	is dense in $L^2(\mu)$.
	\end{lemma}
\begin{proof}
Let $f\in L^2(\mu)$. 	Using the previous notation we have by \eqref{E:p2}
that for $n\in \N$
the spectral measure of the function $f_n={\bf 1 }_{[-n,n]}(T)f$ is supported on the interval $[-n,n]$, hence   $f_n\in \mathcal{G}$. Moreover, since ${\bf 1}_\R(T)f=f$, using \eqref{E:p1} and \eqref{E:p2} we get
$$
\norm{f-f_n}_{L^2(\mu)}=\norm{{\bf 1}_{\R\setminus [-n,n]}(T)f}_{L^2(\mu)}=
\norm{{{\bf 1}_{\R\setminus [-n,n]}}}_{L^2(\sigma_f)}\to 0
$$	
as $n\to\infty$,  since $\sigma_f$ is a bounded measure.
\end{proof}

\begin{lemma}\label{L:dense}
	Let $b\colon \R_+\to \R_+$ be a function from a Hardy field that satisfies $t^\delta\prec b(t)\prec t^{1-\delta}$ for some $\delta>0$, and $(X,\mu,T^t)$, $t\in \R_+$, be a measure preserving flow. Then  for every $f\in L^\infty(\mu)$   and $c\in \R$ we have
	\begin{equation}\label{E:tozero}
	\lim_{y\to+\infty}	\frac{1}{y}\int_0^y|f(T^{b(t+c)}x)-f(T^{b(t)}x)|^2\, dt=0
	\end{equation}
	for $\mu$-almost every $x\in X$.
\end{lemma}
\begin{remark}
More generally, our argument works if the function $b$ satisfies the properties mentioned in the remark following Lemma~\ref{L:CV} and also $b(t)\prec t^{1-\delta}$ and  $b'$ monotonically decreases to $0$ as $t\to+\infty$.
\end{remark}
\begin{proof}
	We can assume that $c\geq 0$.	
	Our assumptions imply that   there exists $t_0>0$ such that for $t\geq t_0$
	we have  $0\leq b'(t)\leq 1$,  $b'(t)$ is  decreasing, and $b(t)\leq t^{1-\delta}$. Since $f\in L^\infty(\mu)$ 
	 it suffices to show that
	\begin{equation}\label{E:tozero'}
	\lim_{y\to+\infty}	\frac{1}{y}\int_{t_0}^{y}|f(T^{b(t+c)}x)-f(T^{b(t)}x)|^2\, dt=0
	\end{equation}
		for $\mu$-almost every $x\in X$.
	
		Let $\varepsilon>0$. If $\mathcal{G}$ is the dense subset of $L^2(\mu)$ given by  Lemma~\ref{L:compact},
	we have a decomposition
	$$
	f=f_1+f_2
	$$
	where $f_1\in \mathcal{G}$ and $\norm{f_2}_{L^2(\mu)}\leq \varepsilon$.
	
	For $g\in L^2(\mu)$ let
	$$
A_y(g):=\frac{1}{y}\int_{t_0}^y|g(T^{b(t+c)}x)-g(T^{b(t)}x)|^2\, dt, \quad y\geq t_0.
	$$
	
	We first deal with the contribution of $f_2$.
We	clearly  have
	$$
 \limsup_{y\to+\infty}  A_y(f_2)	 \leq 4  \lim_{y\to+\infty}  \frac{1}{y}\int_{t_0}^y|f_2(T^{b(t)}x)|^2\, dt.
$$
	Hence, using Lemma~\ref{L:CV} and the pointwise ergodic theorem for flows we deduce that
		\begin{equation}\label{E:f2}
\int \limsup_{y\to+\infty}  A_y(f_2) \, d\mu\leq  4 \int \Big( \lim_{y\to+\infty}  \frac{1}{y}\int_{t_0}^y|f_2(T^tx)|^2\, dt\Big) d\mu = 4\int |f_2|^2\, d\mu\leq 4\varepsilon^2.
 \end{equation}
	
	Next, we deal with the contribution of $f_1$.
	Since $f_1\in \mathcal{G}$, there exists $M>0$ such that the spectral measure $\sigma_{f_1}$ of $f_1$ is supported on the set $[-M,M]$.
	Using the Fubini-Tonelli theorem	we get
	$$
	\norm{A_y(f_1)}_{L^1(\mu)}=
	\frac{1}{y}\int_{t_0}^y\int_X|f_1(T^{b(t+c)}x)-f_1(T^{b(t)}x)|^2\,d\mu \, dt.
	$$
	Using \eqref{E:spectral}    and the  Fubini-Tonelli theorem again, we get that  the last expression is equal to
	$$
	\int_{-M}^M\frac{1}{y}\int_{t_0}^y|e(sb(t+c))-e(sb(t))|^2\, dt\, d\sigma_{f_1}(s).
	$$
	After bounding the integrant pointwise,  using the mean value theorem,  the fact that $b'$  is non-negative, and decreasing for $t\geq t_0$, and that $c\geq 0$, we get  that the last expression is bounded by a constant multiple of
	$$
	\int_{-M}^M\frac{1}{y}\int_{t_0}^y \big (scb'(t))^2\, dt\, d\sigma_{f_1}(s)\leq M^2c^2\norm{f_1}^2_{L^2(\mu)}\frac{1}{y}\int_{t_0}^y (b'(t))^2\, dt.
	$$
Since $0\leq b'(t)\leq 1$ for $t\geq t_0$ we have for $y\geq t_0$ that
	$$	
	\int_{t_0}^y \big(b'(t))^2dt\leq \int_{t_0}^y b'(t)\, dt =b(y)-b(t_0).
	$$
	Since  $b(y)\leq  y^{1-\delta}$ for $y>t_0$, combining the above we get for $C:=Mc^2\nnorm{f_2}^2_{L^2(\mu)}$ that
	$$
	\norm{A_y(f_1)(x)}_{L^1(\mu)}\leq \frac{C}{y^\delta}
	$$
	for all $y>t_0$.  Using the Borel-Cantelli lemma, we get that  for $a>1/\delta$ we have  $\lim_{N\to\infty} A_{N^a}(f_1)(x)= 0$  for $\mu$-almost every $x\in X$.
	From this we deduce  that
	\begin{equation}\label{E:f1}
	\lim_{y\to+\infty} A_y(f_1)(x)=0
	\end{equation}
	for $\mu$-almost every $x\in X$. Indeed, if $N_y\in \Z_+$ is such that $N_y^a\leq y<(N_y+1)^a$, then
	$\norm{A_y(f_1)-A_{N_y^a}(f_1)}_{L^\infty(\mu)}\leq 2\, \norm{f_1}_{L^\infty(\mu)} |1-N_y^a/y|\to 0$ as $y\to+\infty$.
	
Combining \eqref{E:f2} and \eqref{E:f1} we get that
$$
\int \limsup_{y\to+\infty}  A_y(f) \, d\mu\leq 8\varepsilon^2.
$$
Since $\varepsilon$ is arbitrary, we get that \eqref{E:tozero'} holds for $\mu$-almost every $x\in X$, completing the proof.
\end{proof}

\begin{proof}[Proof of Theorem~\ref{T:flows}]

	We prove  the statement by induction on $\ell$. For $\ell=1$ the result follows from Lemma~\ref{L:CV} and the pointwise ergodic theorem for flows. Suppose that the statement holds for $\ell-1$, we shall show that it 	 holds for $\ell$.
		

		Without loss of generality we can assume that
	$\norm{f_j}_{L^\infty(\mu)}\leq 1$ for $j=1,\ldots, \ell$.
		We let
		$$
		A_y(f_\ell)(x):=	\frac{1}{y} \int_0^y f_1(T_1^{a_1(t)}x)\cdot\ldots\cdot f_\ell(T_\ell^{a_{\ell}(t)}x)\, dt, \quad y\in \R_+.
		$$

	Let $\varepsilon>0$.	
 Using a standard  Hilbert space argument we get a decomposition
	$$
	f_\ell=f_{\ell,1}+f_{\ell,2}+f_{\ell,3},
	$$
	where $f_{\ell, 1}=\tilde{f_\ell}$, $\norm{f_{\ell,2}}_{L^2(\mu)}\leq \varepsilon$, and
	$f_{\ell,3}$ belongs to the linear subspace spanned by the functions $T_\ell^ch-h$, $c\in \R$, $h\in L^2(\mu)$. After approximating $h$ in $L^2(\mu)$ by functions in  $L^\infty(\mu)$ and incorporating the error in $f_{\ell,2}$, we can assume that $h\in L^\infty(\mu)$. Furthermore, we have
$\norm{f_{\ell, 1}}_{L^\infty(\mu)}\leq \norm{f_{\ell}}_{L^\infty(\mu)}\leq 1$ and
	  $\norm{f_{\ell,3}}_{L^2(\mu)}\leq 2+\varepsilon$. 	
		An application of the Cauchy-Schwarz inequality and the
 Fubini-Tonelli theorem shows that  for $\mu$-almost every $x\in X$ the
	quantities 	$A_y(f_{\ell,j})$, $j=1,2,3$, are well defined finite numbers.

Our goal is to show \eqref{E:yfl} below. 	We first deal with the contribution of the term $f_{\ell,1}$. Since $T_\ell^tf_{\ell, 1}=f_{\ell,1}=\tilde{f_\ell}$ for every $t\in \R_+$, the induction hypothesis gives
	\begin{equation}\label{E:flows1}
	\lim_{y\to+\infty}	A_y(f_{\ell,1})=\tilde{f_1}\cdots \tilde{f_\ell}.
	\end{equation}
		 for $\mu$-almost every $x\in X$.

	Next we deal with the contribution of the term $f_{\ell,2}$.
For every $y>0$, using the Cauchy-Schwarz inequality   we have
		$$
		|A_y(f_{\ell,2})(x)|^2\leq   \frac{1}{y}\int_0^y |f_{\ell,2}(T_\ell^{a_\ell(t)}x)|^2\, dt.
		$$
 	Using Lemma~\ref{L:CV},
		 the pointwise ergodic theorem for flows, and the Cauchy-Schwarz inequality,  we get (note that in the ergodic case there is no need for the integral on the left hand side)
	\begin{equation}\label{E:flows2}
	\int 	\limsup_{y\to+\infty}|A_y(f_{\ell,2})|\, d\mu \leq
	\norm{f_{\ell,2}}_{L^2(\mu)}  \leq \varepsilon.
	\end{equation}
		
		It remains to deal with  the contribution of the  term $f_{\ell,3}$. We claim that
			\begin{equation}\label{E:flows3}
		\lim_{y\to+\infty}|A_y(f_{\ell,3})(x)|=0
		\end{equation}
		 for $\mu$-almost every $x\in X$.
		 By Lemma~\ref{L:CV},\footnote{We remark that since not all Hardy fields are closed under composition and compositional inversion, we cannot assume that the functions $b_j$ belong to some Hardy field. Nevertheless, one can easily verify that these functions  satisfy the necessary assumptions mentioned on the remark following Lemma~\ref{L:CV}, so we are entitled to apply this lemma. } it suffices to show that
 for $\mu$-almost every $x\in X$ we have
\begin{equation}\label{E:yzero}
\lim_{y\to+\infty}\frac{1}{y}\int_0^{y} \prod_{j=1}^{\ell-1}f_j(T_j^{b_j(t)}x)\cdot f_{\ell,3}(T_\ell^{t}x)\, dt=0
\end{equation}
where $b_j:=a_j\circ a^{-1}_\ell$ for $j=1,\ldots, \ell-1$.

 In order to establish \eqref{E:yzero} it suffices to verify  that if
   $$
 f_{\ell,3}:=T_\ell^ch- h,
 $$
 for some $c\in \R$ and $h\in L^\infty(\mu)$, then \eqref{E:yzero}  holds for $\mu$-almost every $x\in X$.
 So let  $c>0$ and $h\in L^\infty(\mu)$.
After inserting $ f_{\ell,3}=T_\ell^ch- h$  in \eqref{E:yzero}  and using the change of variables $t\mapsto t-c$ in the first of the two integrals, we get that it suffices to show that
 \begin{equation}\label{E:diff}
 \lim_{y\to+\infty}\frac{1}{y}\int_0^{y} \Big(\prod_{j=1}^{\ell-1}f_j(T_j^{b_j(t-c)}x) - \prod_{j=1}^{\ell-1}f_j(T_j^{b_j(t)}x)\Big) \cdot h(T_\ell^{t}x)\, dt=0
 \end{equation}
for $\mu$-almost every $x\in X$. It suffices to show that for
$d=1,\ldots, \ell-1$ we have
\begin{multline}\label{E:diffd}
\lim_{y\to+\infty}\frac{1}{y}\int_0^{y} \Big(\prod_{j=1}^{d}f_j(T_j^{b_j(t-c)}x)\prod_{j=d+1}^{\ell-1}f_j(T_j^{b_j(t)}x) -\\ \prod_{j=1}^{d-1}f_j(T_j^{b_j(t-c)}x)\prod_{j=d}^{\ell-1}f_j(T_j^{b_j(t)}x)\Big) \cdot h(T_\ell^{t}x)\, dt=0
\end{multline}
for $\mu$-almost every $x\in X$.
Finally, notice that our growth assumptions give that $t^{\delta'}\prec b_d(t)\prec t^{1-\delta'}$  for some $\delta'>0$.
Hence,  by Lemma~\ref{L:dense},\footnote{Again here, in order to avoid the assumption that the function $b_d$ belongs to some Hardy field,  we can verify that the assumptions mentioned on the remark following Lemma~\ref{L:dense}
are satisfied.} for $d=1,\ldots, \ell-1$, we have
 \begin{equation}\label{E:tozerod}
 \lim_{y\to+\infty}	\frac{1}{y}\int_0^y|f_d(T_d^{b_d(t-c)}x)-f_d(T_d^{b_d(t)}x)|^2\, dt=0
 \end{equation}
 for  $\mu$-almost every $x\in X$.
 Hence,  using the Cauchy-Schwarz inequality, equation \eqref{E:tozerod}, and the fact that all the functions $f_j$ and $h$ are bounded, we get that   \eqref{E:diffd} holds for $\mu$-almost every $x\in X$. Combning the above we deduce that \eqref{E:flows3} holds.

 From \eqref{E:flows1}, \eqref{E:flows2},  \eqref{E:flows3}, we deduce that
 \begin{equation}\label{E:yfl}
 \int 	\limsup_{y\to+\infty}|A_y(f_{\ell})-\tilde{f_1}\cdots \tilde{f_\ell}|\, d\mu \leq \varepsilon.
 \end{equation}
 Since $\varepsilon$ was arbitrary, we deduce that
 $$
\lim_{y\to+\infty} A_y(f_{\ell})(x)=\tilde{f_1}\cdots \tilde{f_\ell}
 $$
 for $\mu$-almost every $x\in X$, completing the proof.
 \end{proof}

\end{document}